\documentclass[11pt,letterpaper]{amsart}
\pdfoutput=1
\usepackage{eucal}
\usepackage{mathrsfs}
\usepackage{graphicx}
\usepackage{amsopn,amssymb}

\usepackage{tikz}
\usetikzlibrary{matrix,arrows}

\usepackage{color}
\definecolor{darkblue}{rgb}{0,0,0.7}
\usepackage{hyperref}
\hypersetup{pdfauthor={David Dumas},pdftitle={Skinning maps are finite-to-one},colorlinks=true,linkcolor=darkblue,citecolor=darkblue}

\usepackage{caption}
\newcommand{\boldpoint}[1]{\smallskip\par\noindent\textbf{#1}}
\newcommand{\emphpoint}[1]{\medskip\par\noindent\textit{#1}\par\medskip}

\newcommand{\ordinc}{\mathfrak{o}}
\newcommand{\into}{\hookrightarrow}
\newcommand{\tensor}{\otimes}
\newcommand{\alttwo}{{\textstyle\bigwedge}\!^2\,}
\newcommand{\symb}{\mathcal{S}}
\newcommand{\V}{\mathcal{V}}
\newcommand{\Z}{\mathbb{Z}}
\newcommand{\R}{\mathbb{R}}
\newcommand{\C}{\mathbb{C}}
\newcommand{\N}{\mathbb{N}}
\newcommand{\n}{\mathbf{n}}
\renewcommand{\L}{\mathcal{L}}
\newcommand{\Q}{\mathcal{Q}}
\newcommand{\QQ}{\mathbb{Q}}
\newcommand{\sR}{\mathcal{R}}
\newcommand{\sH}{\mathcal{H}}
\newcommand{\sC}{\mathcal{C}}
\DeclareMathOperator{\im}{Im}
\DeclareMathOperator{\re}{Re}
\renewcommand{\Re}{\re}
\renewcommand{\Im}{\im}
\newcommand{\sslash}{/\!\!/}
\newcommand{\half}{\frac{1}{2}}
\renewcommand{\bar}{\overline}

\renewcommand{\leq}{\leqslant}
\renewcommand{\geq}{\geqslant}

\newcommand{\rbdy}{\partial_\R}
\newcommand{\cbdy}{\partial_\C}
\newcommand{\bdy}{\partial}

\newcommand{\I}{\mathscr{I}}
\newcommand{\T}{\mathcal{T}}

\newcommand{\F}{\mathcal{F}}
\newcommand{\CP}{\mathbb{CP}}
\renewcommand{\H}{\mathbb{H}}
\newcommand{\ML}{\mathcal{ML}}
\newcommand{\MF}{\mathcal{MF}}
\renewcommand{\th}{\mathrm{Th}}
\renewcommand{\P}{\mathbb{P}}

\newcommand{\odd}{-}

\newcommand{\X}{\mathcal{X}}
\newcommand{\E}{\mathcal{E}}
\newcommand{\ddtzero}[1]{\left . \frac{\partial\;}{\partial t} {#1}
  \right |_{t=0}}

\newcommand{\verts}[1]{{#1}^{(0)}}
\newcommand{\edges}[1]{{#1}^{(1)}}
\newcommand{\tris}[1]{{#1}^{(2)}}
\newcommand{\tets}[1]{{#1}^{(3)}}
\newcommand{\identity}{\mathrm{Id}}
\newcommand{\noproof}{\hfill\qedsymbol}
\DeclareMathOperator{\Ext}{\mathrm{Ext}}
\DeclareMathOperator{\Mod}{\mathrm{Mod}}
\newcommand{\PSL}{\mathrm{PSL}}
\newcommand{\SL}{\mathrm{SL}}
\DeclareMathOperator{\hol}{hol}
\DeclareMathOperator{\per}{Per}
\DeclareMathOperator{\tr}{tr}
\DeclareMathOperator{\Hom}{Hom}
\DeclareMathOperator{\Cone}{Cone}
\DeclareMathOperator{\GF}{GF}
\DeclareMathOperator{\AH}{AH}
\DeclareMathOperator{\QF}{QF}

\newcommand{\spin}{\varepsilon}
\renewcommand{\hat}{\widehat}

\newcommand{\param}{{\mathchoice{\mkern1mu\mbox{\raise2.2pt\hbox{$\centerdot$}}\mkern1mu}{\mkern1mu\mbox{\raise2.2pt\hbox{$\centerdot$}}\mkern1mu}{\mkern1.5mu\centerdot\mkern1.5mu}{\mkern1.5mu\centerdot\mkern1.5mu}}}

\numberwithin{equation}{section}

\theoremstyle{plain}
\newtheorem{thm}{Theorem}[section]
\newtheorem{cor}[thm]{Corollary}

\newtheorem{lem}[thm]{Lemma}  
\newtheorem{prop}[thm]{Proposition} 
\newtheorem*{nothm}{Theorem}
\newtheorem{bigthm}{Theorem}

\theoremstyle{definition}

\theoremstyle{definition}
\newtheorem*{remark}{Remark}
\newtheorem*{remarksenv}{Remarks}

\newenvironment{rmenumerate}{\begin{enumerate}}{\end{enumerate}}

\newenvironment{prooflist}{\begin{proof}\mbox{}\begin{list}{(\roman{enumi})}{\usecounter{enumi}\setlength{\labelwidth}{30pt}\setlength{\itemindent}{0pt}\setlength{\leftmargin}{20pt}\setlength{\itemsep}{5pt}}}{\end{list}\end{proof}}

\begin{document}

\title{Skinning maps are finite-to-one}
\author{David Dumas}
\thanks{Work partially supported by the National Science Foundation.}

\address{
Department of Mathematics, Statistics, and Computer Science\\
University of Illinois at Chicago\\
\tt{ddumas@math.uic.edu}
}

\date{June 25, 2015 (revised).  March 1, 2012 (original).}

\maketitle


\section{Introduction}

Skinning maps were introduced by William Thurston in the proof of the
Geometrization Theorem for Haken $3$-manifolds (see
\cite{otal:hyperbolization-haken}).
At a key step in the proof
one has a compact $3$-manifold $M$ with nonempty boundary whose
interior admits a hyperbolic structure.  The interplay between
deformations of the hyperbolic structure and the topology of $M$ and
$\bdy M$ determines a holomorphic map of Teichm\"uller spaces, the
\emph{skinning map}
$$ \sigma_M : \T(\bdy_0M) \to \T(\bar{\bdy_0M}),$$
where $\bdy_0M$ is the union of the non-torus boundary components and
$\bar{\bdy_0M}$ denotes the boundary with the opposite orientation.
The problem of finding a hyperbolic structure on a related
closed manifold is solved by showing that the composition of
$\sigma_M$ with a certain isometry $\tau : \T(\bar{\bdy_0M}) \to
\T(\bdy_0M)$ has a fixed point.

Thurston's original approach to the fixed point problem involved
extending the skinning map of an acylindrical manifold to a continuous
map $\Hat{\sigma}_M : \AH(M) \to \T(\bar{\bdy_0M})$ defined on the
compact space $\AH(M)$ of hyperbolic structures on $M$.  (A proof of
this Bounded Image Theorem can be found in
\cite[Sec.~9]{kent:skinning}.)  McMullen provided an alternate
approach based on an analytic study of the differential of the
skinning map \cite{mcmullen:theta} \cite{mcmullen:iteration}.  In each
case there are additional complications when $M$ has essential
cylinders.

More recently, Kent studied the diameter of the image of the skinning
map (in cases when it is finite), producing examples where this
diameter is arbitrarily large or small and relating the diameter to
hyperbolic volume and the depth of an embedded collar around the
boundary \cite{kent:skinning}.  However, beyond the contraction and
boundedness properties used to solve the fixed point problem---and the
result of Kent and the author that skinning maps are never constant
\cite{dumas-kent:slicing}---little is known about skinning maps in
general.

Our main theorem concerns the fibers of skinning maps:

\begin{bigthm}
\label{thm:main}
Skinning maps are finite-to-one.  That is, let $M$ be a compact
oriented $3$-manifold whose boundary is nonempty and not a union of
tori.  Suppose that the interior of $M$ admits a complete hyperbolic
structure without accidental parabolics, so that $M$ has an associated
skinning map $\sigma_M$.  Then for each $X \in \T(\bar{\bdy_0M})$, the preimage $\sigma_M^{-1}(X)$
is finite.
\end{bigthm}

As a holomorphic map with finite fibers, it follows from Theorem
\ref{thm:main} that skinning maps are open (answering a question in
\cite{dumas-kent:slicing}) and locally biholomorphic away from the
analytic hypersurface defined by the vanishing of the Jacobian
determinant.  Thus our results give strong nondegeneracy
properties for all skinning maps.

Our proof of Theorem \ref{thm:main} does not bound the size of the
finite set $\sigma_M^{-1}(X)$; instead, we show that each fiber of the
skinning map is both compact and discrete.  In particular it is not
clear if the number of preimages of a point is uniformly bounded over
$\T(\bar{\bdy_0M})$.

In case $M$ is acylindrical, we also show that Thurston's extension of
the skinning map to $\AH(M)$ is finite-to-one; this result appears as Theorem
\ref{thm:main-extended} below.

\subsection*{The intersection problem}
To study the fibers of the skinning map we translate the problem
to one of intersections of certain subvarieties of
the $\SL_2\C$-character variety of $\bdy_0M$.  The same reduction to
an intersection problem is used in \cite{dumas-kent:slicing}.  The
relevant subvarieties are:
\begin{itemize}
\item The \emph{extension variety} $\E_M$, which is the smallest closed algebraic
subvariety containing the characters of all homomorphisms $\pi_1(\bdy_0M) \to
\SL_2\C$ that can be extended to $\pi_1M
\to \SL_2\C$.
\item The \emph{holonomy variety} $\sH_X$, which is the analytic
subvariety consisting of characters of the $\SL_2\C$-lifts of holonomy
representations of $\CP^1$-structures on $\bar{\bdy_0M}$ compatible
with the complex structure $X \in \T(\bar{\bdy_0M})$.
\end{itemize}
Precise definitions of these objects are provided in Section
\ref{sec:3-manifolds}, with additional details of the disconnected
boundary case in Section \ref{sec:disconnected}.

The main theorem is derived from the following result about the
intersections of holonomy and extension varieties, the proof of which
occupies most of the paper.

\begin{bigthm}[Intersection Theorem]
\label{thm:main-intersection}
Let $M$ be an oriented $3$-manifold with nonempty boundary that is not
a union of tori.  Let $X$ be a marked Riemann surface structure on
$\bdy_0M$.  Then the intersection $\sH_X \cap \E_M$ is a discrete
subset of the character variety.
\end{bigthm}

This theorem applies in a more general setting than the specific
intersection problem arising from skinning maps.  For example, while
skinning maps are defined for manifolds with incompressible boundary,
such a hypothesis is not needed in Theorem
\ref{thm:main-intersection}.

While the theorem above involves an oriented manifold $M$, the set
$\E_M$ is independent of the orientation.  Thus we also obtain
discreteness of intersections $\sH_{\bar{X}} \cap \E_M$ where
$\bar{X}$ is a Riemann surface structure on $\bdy_0M$ that induces
an orientation opposite that of the boundary orientation of $M$.

\subsection*{Steps to the intersection theorem}
Our study of $\sH_X \cap \E_M$ is based on the parameterization of the
irreducible components of $\sH_X$ by the vector space $Q(X)$ of
holomorphic quadratic differentials; this parameterization is the
\emph{holonomy map} of $\CP^1$-structures, denoted by ``$\hol$''.
The overall strategy is to show that the preimage of $\E_M$, i.e.~the
set
$$ \V_M = \hol^{-1}(\E_M),$$
is a complex analytic subvariety of $Q(X)$ that is subject to
certain constraints on its behavior at infinity, and ultimately to
show that only a discrete set can satisfy these.

We now sketch the main steps of the argument and state some
intermediate results of independent interest.  In this sketch we
restrict attention to the case of a $3$-manifold $M$ with
\emph{connected} boundary $S$.  Let $X \in \T(S)$ be a marked Riemann
surface structure on the boundary.

\emphpoint{Step 1. Construction of an isotropic cone in the space of
  measured foliations.}

The defining property of this cone in $\MF(S)$ is that it determines
which quadratic differentials $\phi \in Q(X)$ have dual trees $T_\phi$
that admit ``nice'' equivariant maps into trees on which $\pi_1M$ acts
by isometries.  Because of the way this cone is used in a later step
of the argument, here we must consider actions of $\pi_1M$ not only
$\R$-trees but also on $\Lambda$-trees, where $\Lambda$ is a more
general ordered abelian group.  And while an isometric embedding of
$T_\phi$ into an $\R$-tree on which $\pi_1M$ acts is the prototypical
example of a nice map, we must also consider the \emph{straight maps}
defined in \cite{dumas:holonomy} and certain partially-defined maps
arising from non-isometric trees with the same length function.

In the case of straight maps, we show:

\begin{bigthm}
\label{thm:main-floyd}
There is an isotropic piecewise linear cone $\L_{M,X} \subset \MF(S)$
with the following property:  If $\pi_1M$ acts on a $\Lambda$-tree $T$
by isometries, and if $\phi \in Q(X)$ is a quadratic differential
whose dual tree $T_\phi$ admits an equivariant straight map $T_\phi
\to T$, then the horizontal foliation of $\phi$ lies in $\L_{M,X}$.
\end{bigthm}

This result is stated precisely in Theorem
\ref{thm:straight-isotropic} below, and Theorem
\ref{thm:length-isotropic} presents a further refinement that is used
in the proof of the main theorem.

\emphpoint{Step 2. Limit points of $\V_M$ have foliations in the isotropic cone.}

In \cite{dumas:holonomy} we analyze the large-scale behavior of the
holonomy map, showing that straight maps arise naturally when
comparing limits in the Morgan-Shalen compactification---which are
represented by actions of $\pi_1S$ on $\R$-trees---to the dual trees
of limit quadratic differentials.  The main result can be summarized
as follows:

\begin{nothm}
If a divergent sequence in $Q(X)$ can be rescaled to have limit
$\phi$, then any Morgan-Shalen limit of the associated holonomy
representations corresponds to an $\R$-tree $T$ that admits an
equivariant straight map $T_\phi \to T$.
\end{nothm}

Here rescaling of the divergent sequence uses the action of $\R^+$ on
$Q(X)$.  The precise limit result we use is stated in Theorem
\ref{thm:holonomy-limits}, and other related results and discussion
can be found in \cite{dumas:holonomy}.

When we restrict attention to the subset $\V_M \subset Q(X)$, the
associated holonomy representations lie in $\E_M$ and so they arise as
compositions of $\pi_1M$-representations with the map $i_* : \pi_1S
\to \pi_1M$ induced by the inclusion of $S$ as the boundary of $M$.
(Strictly speaking, this describes a Zariski open subset of $\E_M$.)
We think of these as representations that ``extend'' from
$i_*(\pi_1S)$ to the larger group $\pi_1M$.

Note that \emph{a priori} the passage from a $\pi_1S$-representation
to a $\pi_1M$-representation could radically change the geometry of
the associated action on $\H^3$, as measured for example by the
diameter of the orbit of a finite generating set.  This possibility,
combined with the need to keep track of the limiting $\pi_1S$- and
$\pi_1M$-dynamics simultaneously, requires us to consider
$\Lambda$-trees more general than $\R$-trees.

Using the valuation constructions of Morgan-Shalen, we show that there
is a similar extension property for the trees obtained as limits
points of $\E_M$, or more precisely, for their length functions
(Theorem \ref{thm:length-function-extension}).  In the generic case of
a non-abelian action, the combination of this construction with the
holonomy limit theorem above gives an $\R^n$-tree $\hat{T}$ (where
$\R^n = \Lambda$ is given the lexicographical order) on which $\pi_1M$
acts by isometries and a straight map
$$ T_\phi \to \hat{T}$$
where $\phi$ is the rescaled limit of a divergent sequence in
$\V_M$. These satisfy the hypotheses of
Theorem~\ref{thm:main-floyd}, so the horizontal foliation of $\phi$
lies in $\L_{M,X}$.  Limit points of $\E_M$ that correspond to abelian
actions introduce minor additional complications that are handled by
Theorem \ref{thm:length-isotropic}.

\emphpoint{Step 3. The foliation map $\F : Q(X) \to \MF(S)$ is symplectic.}

Hubbard and Masur showed that the foliation map $\F : Q(X) \to \MF(S)$
is a homeomorphism.  In order to use the isotropic cone $\L_{M,X}$ to
understand the set $\V_M$, we analyze the relation between
the foliation map, the complex structure of $Q(X)$, and the symplectic
structure of $\MF(S)$.  

We introduce a natural K\"ahler structure on $Q(X)$ corresponding to
the Weil-Petersson-type hermitian pairing
$$\langle \psi_1, \psi_2 \rangle_\phi = \int_X \frac{\psi_1
  \bar{\psi}_2}{4 |\phi|} $$ Here we have a base point $\phi \in Q(X)$
and the quadratic differentials $\psi_i \in T_\phi Q(X) \simeq Q(X)$
are considered as tangent vectors.  This integral pairing is not
smooth, nor even well-defined for all tangent vectors, due to
singularities of the integrand coming from higher-order zeros of
$\phi$.  However we show that the pairing does give a well-defined
K\"ahler structure relative to a stratification of $Q(X)$.

We show that the underlying symplectic structure of this K\"ahler
metric is the one pulled back from $\MF(S)$ by the foliation map:

\begin{bigthm}
\label{thm:main-symplectomorphism}
For any $X \in \T(S)$, the map $\F : Q(X) \to \MF(S)$ is a
real-analytic stratified symplectomorphism, where $Q(X)$ is given the
symplectic structure coming from the pairing $\langle \psi_1, \psi_2
\rangle_\phi$ and where $\MF(S)$ has the Thurston symplectic form.
\end{bigthm}

The lack of a smooth structure on $\MF(S)$ means that the regularity
aspect of this result must be interpreted carefully.  We show
that for any point $\phi \in Q(X)$ there is a neighborhood in its
stratum and a train track chart containing $\F(\phi)$ in which the
local expression of the foliation map is real-analytic and symplectic.
The details are given in Theorem \ref{thm:symplectomorphism}.

\emphpoint{Step 4. Analytic sets with totally real limits are discrete.}

In a K\"ahler manifold, an isotropic submanifold is totally real.
While the piecewise linear cone $\L_{M,X}$ is not globally a manifold,
Theorem \ref{thm:main-symplectomorphism} allows us to describe
$\F^{-1}(\L_{M,X})$ locally in a stratum of $Q(X)$ in terms of totally
real, real-analytic submanifolds.  Since limit points of $\V_M$
correspond to elements of $\F^{-1}(\L_{M,X})$, this gives a kind of ``totally
real'' constraint on the behavior of $\V_M$ at infinity. 

To formulate this constraint we consider the set $\partial_\R\V_M
\subset S^{2N-1}$ of points in the unit sphere of $Q(X) \simeq \C^N$
that can be obtained as $\R^+$-rescaled limits of divergent sequences
in $\V_M$.  Projecting this set through the Hopf fibration $S^{2N-1}
\to \CP^{N-1}$ we obtain the set $\partial_\C \V_M$ of boundary points
of $\V_M$ in the complex projective compactification of $Q(X)$.  Using
the results of steps 1--3 we show (in Theorem
\ref{thm:real-and-complex}) that:
\begin{rmenumerate}
\item In a neighborhood of some point, $\cbdy \V_M$ is contained in a totally
real manifold, and
\item The intersection of $\rbdy \V_M$ with a fiber of the Hopf map
has empty interior.
\end{rmenumerate}

Using extension and parameterization results from analytic geometry
it is not hard to show that among analytic subvarieties of $Q(X)$,
only a discrete subset of can have both of these properties.
Condition (i) forces any analytic curve in $\V_M$ to extend to an
analytic curve in a neighborhood of some boundary point $p \in
\CP^{N-1}$.  Within this extension there is a generically a circle of
directions in which to approach the boundary point, some arc of which
is realized by the original curve.  Analyzing the correspondence
between this circle and the Hopf fiber over $p$, one finds that $\rbdy
\V_M$ contains an open arc of this fiber, violating condition (ii).

This contradiction shows that $\V_M$ contains no analytic curves,
making it a discrete set.  The intersection theorem follows.

\subsection*{Applications and complements}

The construction of the isotropic cone in Theorem \ref{thm:main-floyd}
was inspired by the work of Floyd on the space of boundary curves of
incompressible, $\bdy$-incompressible surfaces in $3$-manifolds
\cite{floyd:boundary-curves}.  Indeed, in the incompressible boundary
case, lifting such a surface to the universal cover and considering
dual trees in the boundary and in the $3$-manifold gives rise to an
isometric embedding of $\Z$-trees.  Using Theorem~\ref{thm:main-floyd} we
recover Floyd's result in this case.  This connection is explained in
more detail in Section \ref{sec:floyd}, where we also note that the
same ``cancellation'' phenomenon is at the core of both arguments.

Since Theorem \ref{thm:main-symplectomorphism} provides an
interpretation of Thurston's symplectic form in terms of Riemannian
and K\"ahler geometry of a (stratified) smooth manifold, we hope that
it might allow new tools to be applied to problems involving the space
of measured foliations.  In Section \ref{sec:hubbard-masur-constant}
we describe work of Mirzakhani in this direction, where it is shown
that a certain function connected to $\Mod(S)$-orbit counting problems
in Teichm\"uller space is constant.

As a possible extension of Theorem \ref{thm:main-intersection} one
might ask whether $\sH_X$ and $\E_M$ always intersect transversely, or
equivalently, whether skinning maps are always immersions.  In the
slightly more general setting of manifolds with rank-$1$ cusps a
negative answer was recently given by Gaster \cite{gaster}.  In
addition to Gaster's example, numerical experiments conducted jointly
with Richard Kent suggest critical points for some other manifolds
with rank-$1$ cusps \cite{dumas-kent:experiments}.

Nevertheless it would be interesting to understand the discreteness of
the intersection $\sH_X \cap \E_M$ through local or differential
properties rather than the compactifications and asymptotic arguments
used here.  The availability of rich geometric structure on the
character variety and its compatibility with the subvarieties in
question offers some hope in this direction; for example, both $\sH_X$
and $\E_M$ are Lagrangian with respect to the complex symplectic
structure of the character variety (see \cite{kawai}
\cite[Sec.~7.2]{labourie:notes} \cite[Sec.~1.7]{krasnov-schlenker}).

\subsection*{Structure of the paper}

Section \ref{sec:preliminaries} recalls some definitions and basic
results related to $\Lambda$-trees, measured foliations, and
Teichm\"uller theory.

Sections \ref{sec:isotropic1}--\ref{sec:analytic-curves} contain the
proofs of the main theorems in the case of a $3$-manifold with
connected boundary.  Working in this setting avoids some cumbersome
notation and other issues related to disconnected spaces, while all
essential features of the argument are present.  Sections
\ref{sec:isotropic1}--\ref{sec:isotropic2} are devoted to the
isotropic cone construction, Section \ref{sec:kahler} introduces the
stratified K\"ahler structure, and Sections
\ref{sec:3-manifolds}--\ref{sec:analytic-curves} combine these with
the results of \cite{dumas:holonomy} to prove Theorem \ref{thm:main}
in the connected boundary case.

In Section \ref{sec:disconnected} we adapt the definitions and results
of the previous sections as necessary to handle a manifold with
disconnected boundary, possibly including torus components, completing
the proof of Theorem \ref{thm:main}.

Finally, in Section \ref{sec:acylindrical} we prove an analogue of
Theorem \ref{thm:main} for the extended skinning map of an
acylindrical $3$-manifold.

\subsection*{Acknowledgments}

A collaboration with Richard Kent (in \cite{dumas-kent:dense}
\cite{dumas-kent:slicing}) provided essential inspiration in the early
stages of this project.  The author is grateful for this and for
helpful conversations with Athanase Papadopoulos, Kasra Rafi, and
Peter Shalen.  The author also thanks Maryam Mirzakhani for allowing
the inclusion of Theorem \ref{thm:mirzakhani} and the anonymous
referees for many helpful comments and suggestions.

\section{Preliminaries}
\label{sec:preliminaries}

\subsection{Ordered abelian groups}

An \emph{ordered abelian group} is a pair $(\Lambda,<)$ consisting of
an abelian group $\Lambda$ and a translation-invariant total order $<$
on $\Lambda$.  We often consider the order to be implicit and denote
an ordered abelian group by $\Lambda$ alone.  Note that an order on
$\Lambda$ also induces an order on any subgroup of $\Lambda$.

The positive subset of an ordered abelian group $\Lambda$ is the set
$\Lambda^+ = \{ g \in \Lambda \: | \: g > 0 \}$.  If $x \in \Lambda$
is nonzero, then exactly one of $x,-x$ lies in $\Lambda^+$, and we
denote this element by $|x|$.

If $x,y \in \Lambda$, we say that $y$ is \emph{infinitely larger} than
$x$ if $n |x| < |y|$ for all $n \in \N$.  If neither of $x,y$ is
infinitely larger than the other, then $x$ and $y$ are
\emph{archimedean equivalent}.  When the set of archimedean
equivalence classes of nonzero elements is finite, the number of such
classes is the \emph{rank} of $\Lambda$.  In what follows we consider
\emph{only} ordered abelian groups of finite rank.

A subgroup $\Lambda' \subset \Lambda$ is \emph{convex} if whenever
$g,k \in \Lambda'$, $h \in \Lambda$, and $g < h < k$ we have $h
\in \Lambda'$.  The convex subgroups of a given group are ordered by
inclusion.  A convex subgroup is a union of archimedean equivalence
classes and is uniquely determined by the largest archimedean
equivalence class that it contains (which exists, since the rank is
finite).  In this way the convex subgroups of a given ordered abelian
group are in one-to-one order-preserving correspondence with its
archimedean equivalence classes.

\subsection{Embeddings and left inverses}

If we equip $\R^n$ with the lexicographical order, then the inclusion
$\R \into \R^n$ as one of the factors is order-preserving.  This
inclusion has a left inverse $\R^n \to \R$ given by projecting onto
the factor.  This projection is of course a homomorphism but it is
\emph{not} order-preserving.

Similarly, the following lemma shows that an order-preserving
embedding of $\R$ into any ordered abelian group of finite rank has a
left inverse; this is used in Section \ref{sec:construction-of-isotropic}.

\begin{lem}
\label{lem:left-inverse}
If $\Lambda$ is an ordered abelian group of finite rank and $\ordinc :
\R \into \Lambda$ is an order-preserving homomorphism, then $\ordinc$
has a left inverse.  That is, there is a homomorphism $\varphi :
\Lambda \to \R$ such that $\varphi \circ \ordinc = \identity$.
\end{lem}

To construct a left inverse we use the following structural result for
ordered abelian groups; part (i) is the \emph{Hahn embedding theorem}
(see e.g.~\cite{gravett} \cite[Sec.~II.2]{kokorin-kopytov}):

\begin{thm}
\label{thm:hahn}
Let $\Lambda$ be an ordered abelian group of rank $n$.
\begin{rmenumerate}
\item There exists an order-preserving embedding $\Lambda \into
\R^{n}$ where $\R^n$ is given the lexicographical order.

\item If $n=1$, the order-preserving embedding $\Lambda \into \R$ is
unique up to multiplication by a positive constant.
\end{rmenumerate}
\noproof
\end{thm}

\begin{proof}[{Proof of Lemma \ref{lem:left-inverse}.}]
By Theorem \ref{thm:hahn}, it suffices to consider the case of an
order-preserving embedding $\ordinc: \R \to \R^n$ where $\R^n$ has the
lexicographical order.  Since the only order-preserving
self-homomorphisms of the additive group $\R$ are multiplication by
positive constants, it is enough to find a homomorphism $\varphi: \R^n
\to \R$ such that $\varphi \circ \ordinc$ is order-preserving and
fixes a point.

Let $a = \ordinc(1)$, and write $a = (a_1, \ldots, a_n)$.  Since $a > 0$
in $\R^n$, the first nonzero element of the tuple $a$ is positive.
Let $k$ be the index of this element, i.e.~$k = \min \{ i \: | \: a_i
> 0 \}$.

We claim that for any $t \in \R$, the image $\ordinc(t)$ has the form
$(0,\ldots,0,b_k,\ldots,b_n)$.  If not, then after possibly replacing
$t$ by $-t$ we have $t \in \R_+$ such that $\ordinc(t)$ is infinitely
larger than $\ordinc(1)$.  The existence of a positive integer $n$ such
that $t < n$ shows that this contradicts the order-preserving property
of $\ordinc$.

Similarly, we find that if $t>0$ then $b = \ordinc(t)$ satisfies $b_k > 0$:
The order-preserving property of $\ordinc$ implies that $b_k \geq 0$, so
the only possibility to rule out is $b_k = 0$.  But if $b_k=0$ then
$\ordinc(1)$ is infinitely larger than $\ordinc(t)$, yet there is a
positive integer $n$ such that $1 < nt$, a contradiction.

Now define $\varphi : \R^n \to \R$ by
$$ \varphi(x_1, \ldots, x_n) = \frac{x_k}{a_k}.$$
This is a homomorphism satisfying $\varphi(\ordinc(1)) = 1$, and the properties
of $\ordinc$ derived above show that $\varphi \circ \ordinc$ is
order-preserving, as desired.
\end{proof}

We now consider the properties of embeddings $\Lambda \to \R^n$, such
as those provided by Theorem \ref{thm:hahn}, with respect to convex
subgroups.  First, a proper convex subgroup maps into $\R^{n-1}$:

\begin{lem}
\label{lem:sub-rn}
Let $F : \Lambda \to \R^n$ be an order-preserving embedding, where
$\Lambda$ has rank $n$.  If $\Lambda' \subset \Lambda$ is a
proper convex subgroup, then $F(\Lambda') \subset \{ (0,x_2, \ldots,
x_n) \}$.
\end{lem}

\begin{proof}
Since $\Lambda' \neq \Lambda$, the convex subgroup $\Lambda'$ does not
contain the largest Archimedean equivalence class of $\Lambda$.  Thus
there exists a positive element $g \in \Lambda_+$ such that $h < g$ for all $h \in
\Lambda'$.

Suppose that there exists $h \in \Lambda'$ such that $F(h) = (a_1,
a_2, \ldots, a_n)$ with $a_1 \neq 0$.  Then we have $F(kh) = ka_1 >
F(g)$ for some $k \in \Z$.  This contradicts the order-preserving
property of $F$, so no such $h$ exists and $F(\Lambda')$ has the
desired form.
\end{proof}

Building on the previous result, the following lemma shows that in
some cases the embeddings given by Hahn's theorem behave functorially
with respect to rank-$1$ subgroups.  This result is used in Section
\ref{sec:extension}.

\begin{lem}
\label{lem:super-hahn}
Let $\Lambda$ be an ordered abelian group of finite rank and $\Lambda'
\subset \Lambda$ a subgroup contained in the minimal nontrivial convex
subgroup of $\Lambda$.  Then there is a commutative diagram of
order-preserving embeddings
$$
\begin{tikzpicture}[baseline=(current bounding box.center)]
\matrix (m) [matrix of math nodes, column sep=1.5em, row sep = 1.5em, text height=1.5ex, text depth=0.25ex]
{
\Lambda' & \Lambda\\
\R & \R^n\\
};
\path[->,font=\scriptsize]
(m-1-1) edge (m-1-2)
(m-1-2) edge (m-2-2)
(m-1-1) edge (m-2-1)
(m-2-1) edge (m-2-2);
\end{tikzpicture}
$$
where $i_n(x) = (0, \ldots, 0, x)$ and $n$ is the rank of $\Lambda$.
\end{lem}

\begin{proof}
Let $\Lambda_1 \subset \Lambda_2 \subset \cdots \subset \Lambda_n =
\Lambda$ be the convex subgroups of $\Lambda$.  We can assume that
$\Lambda' = \Lambda_1$ since all other cases are handled by
restricting the maps from this one.

We are given the inclusion $i : \Lambda_1 \to \Lambda$ and the Hahn
embedding theorem provides an order-preserving embedding $F : \Lambda
\to \R^n$.  Applying Lemma \ref{lem:sub-rn} to each step in the chain
of convex subgroups of $\Lambda$, we find that for all $g \in
\Lambda_1$ we have
$$ F(g) = (0, \ldots, 0, f(g))$$
and the induced map $f : \Lambda_1 \to \R$ is order-preserving.  Since
$F \circ i = i_n \circ f$ by construction, these maps complete the
commutative diagram.
\end{proof}

\subsection{$\Lambda$-metric spaces and $\Lambda$-trees}
  \label{sec:lambda}

We refer to the book \cite{chiswell} for general background on
$\Lambda$-metric spaces and $\Lambda$-trees.  Here we recall the
essential definitions and fix notation.

As before let $\Lambda$ denote an ordered abelian group.  A
$\Lambda$-metric space is a pair $(M,d)$ where $M$ is a set and $d : M
\times M \to \Lambda$ is a function which satisfies the usual axioms
for the distance function of a metric space.  In particular an
$\R$-metric space (where $\R$ has the standard order) is the usual
notion of a metric space.

An isometric embedding of one $\Lambda$-metric space into another is
defined in the natural way.  Generalizing this, let $(M,d)$ be a
$\Lambda$-metric space and $(M',d')$ a $\Lambda'$-metric space.  An
\emph{isometric embedding} of $M$ into $M'$ is a pair $(f,\ordinc)$
consisting of a map $f : M \to M'$ and an order-preserving
homomorphism $\ordinc : \Lambda \to \Lambda'$ such that
$$ d'(f(x),f(y)) = \ordinc(d(x,y)) \; \text{for all} \; x,y \in M.$$
More generally we say $f : M \to M'$ is an isometric embedding if
there exists an order-preserving homomorphism $\ordinc$ such that the
pair $(f,\ordinc)$ satisfy this condition.

An ordered abelian group $\Lambda$ is an example of a $\Lambda$-metric
space, with metric $d(g,h) = |g-h|$.  An isometric embedding of the
subspace $[g,h] := \{ k \in \Lambda \: | \: g \leq h \leq k \} \subset
\Lambda$ into a $\Lambda$-metric space is a \emph{segment}.  A
$\Lambda$-metric space is \emph{geodesic} if any pair of points can be
joined by a segment.

A $\Lambda$-tree is a $\Lambda$-metric space $(T,d)$ satisfying three
conditions:
\begin{itemize}
\item $(T,d)$ is geodesic,
\item If two segments in $T$ share an endpoint but have no other
intersection points, then their union is a segment, and
\item If two segments in $T$ share an endpoint, then their
intersection is a segment (or a point).
\end{itemize}

The notion of $\delta$-hyperbolicity for metric spaces generalizes
naturally to $\Lambda$-metric spaces, where now $\delta \in \Lambda$,
$\delta \geq 0$.  In terms of this generalization, any $\Lambda$-tree
is $0$-hyperbolic.  (The converse holds under mild additional
assumptions on the space.)  The $0$-hyperbolicity condition has
various equivalent characterizations, but the one we will use in the
sequel is the following condition on $4$-tuples of points:

\begin{lem}[$0$-hyperbolicity of $\Lambda$-trees]
If $(T,d)$ is a $\Lambda$-tree, then for all $x,y,z,t \in T$ we have
$$ d(x,y) + d(z,t) \leq \max \left ( d(x,z) + d(y,t), \:d(x,t) + d(y,z)
\right).$$
\noproof
\end{lem}

For a proof and further discussion see 
\cite[Lem.~1.2.6 and Lem.~2.1.6]{chiswell}.
  By permuting a given $4$-tuple $x,y,z,t$ and
considering the inequality of this lemma, we obtain the following
corollary (see \cite[p.~35]{chiswell}):

\begin{lem}[Four points in a $\Lambda$-tree]
\label{lem:four-points-lambda-tree}
Let $(T,d)$ be a $\Lambda$-tree and $x,y,z,t \in T$.  Then among the
three sums
\begin{equation*}
d(x,y) + d(z,t), \;\:
d(x,z) + d(y,t), \;\:
d(x,t) + d(y,z),
\end{equation*}
two are equal, and these two are not less than the third. \noproof
\end{lem}

Given a $\Lambda$-tree, there are natural constructions that associate
trees to certain subgroups or extensions of $\Lambda$; in what follows
we require two such operations.  First, let $(T,d)$ be a
$\Lambda$-tree and $\Lambda' \subset \Lambda$ a convex subgroup.  For any
$x \in T$ we can consider the subset $T_{\Lambda',x} = \{ y \in T \: |
\: d(x,y) \in \Lambda'\}$.  Then the restriction of $d$ to
$T_{\Lambda',x}$ takes values in $\Lambda'$, and this gives
$T_{\Lambda',x}$ the structure of a $\Lambda'$-tree
\cite[Prop.~II.1.14]{morgan-shalen:valuations-trees}.

Second, suppose that $\ordinc : \Lambda \to \Lambda'$ is an
order-preserving homomorphism and that $(T,d)$ is a $\Lambda$-tree.
Then there is a natural \emph{base change} construction that produces
a $\Lambda'$-tree $\Lambda' \tensor_\Lambda T$ and an isometric
embedding $T \to \Lambda' \tensor_\Lambda T$ with respect to $\ordinc$
(see \cite[Thm.~4.7]{chiswell} for details).  Roughly speaking, if one
views $T$ as a union of segments, each identified with some interval
$[g,h] \subset \Lambda$, then $\Lambda' \tensor_\Lambda T$ is obtained
by replacing each such segment with $[\ordinc(g),\ordinc(h)] \subset
\Lambda'$.

\subsection{Group actions on $\Lambda$-trees and length functions}

Every isometry of a $\Lambda$-tree is either \emph{elliptic},
\emph{hyperbolic}, or an \emph{inversion}; see
\cite[Sec.~3.1]{chiswell} for a detailed discussion of this
classification.  Elliptic isometries are those with fixed points,
while hyperbolic isometries have an invariant axis (identified with a
subgroup of $\Lambda$) on which they act as a translation.
An inversion is an isometry that has no fixed point but which induces
an elliptic isometry after an index-$2$ base change; permitting such
base change allows us to make the standing assumption that
\emph{isometric group actions on $\Lambda$-trees that we consider are
  without inversions}.

The \emph{translation length} $\ell(g)$ of an isometry $g : T \to T$
of a $\Lambda$-tree is defined as
$$
\ell(g) = \begin{cases}
0 \; \text{ if } g \text{ is elliptic,}\\
|t| \; \text{ if } g \text{ is hyperbolic and acts on its axis as }h
  \mapsto h + t.
  \end{cases}$$ Note that $\ell(g) \in \Lambda^+ \cup \{0\}$.  It can
  be shown that the translation length is also given by $\ell(g) = \min
  \{ d(x,g(x)) \: | \: x \in T \}$.

 When a group $G$ acts on a $\Lambda$-tree by isometries, taking
the translation length of each element of $G$ defines a function $\ell :
 G \to \Lambda^+ \cup \{0\}$, the \emph{translation length function}
 (or briefly, the \emph{length function}) of the action.

When the translation length function takes values in a convex
subgroup, one can extract a subtree whose distance function takes
values in the same subgroup:

\begin{lem}
\label{lem:subtree}
Let $G$ act on a $\Lambda$-tree $T$ with length function $\ell$.  If
$\Lambda' \subset \Lambda$ is a convex subgroup and $\ell(G) \subset
\Lambda'$ then there is a $\Lambda'$-tree $T' \subset T$ that is
invariant under $G$ and such that $\ell$ is also the length function of the
induced action of $G$ on $T'$.
\end{lem}

This lemma is implicit in the proof of Theorem 3.7 in
\cite{morgan:group-actions-on-trees}, which uses the structure theory
of actions developed in \cite{morgan-shalen:valuations-trees}.  For the convenience of the
reader, we reproduce the argument here.

\begin{proof}
Because $\Lambda'$ is convex, there is an induced order on the
quotient group $\Lambda / \Lambda'$.  Define an equivalence relation
on $T$ where $x \sim y$ if $d(x,y) \in \Lambda'$.  Then the quotient
$T_0 = T / \sim$ is a $(\Lambda/\Lambda')$-tree, and each fiber of the
projection $T \to T_0$ is a $\Lambda'$-tree.  The action of $G$ on $T$
induces an action on $T_0$ whose length function is the composition of
$\ell$ with the map $\Lambda \to \Lambda/\Lambda'$, which is
identically zero since $\ell(G) \subset \Lambda'$.  It follows that
the action of $G$ on $T_0$ has a global fixed point
\cite[Prop.~II.2.15]{morgan-shalen:valuations-trees}, and thus $G$ acts on the fiber of $T$ over
$T_0$, which is a $\Lambda'$-tree $T'$.  By \cite[Prop.~II.2.12]{morgan-shalen:valuations-trees}
the length function of the action of $G$ on $\Lambda'$ is $\ell$.
\end{proof}

\subsection{Measured foliations and train tracks}
\label{sec:mf}

Let $\MF(S)$ denote the space of measured foliations on a compact
oriented surface $S$ of genus $g$.  Then $\MF(S)$ is a piecewise
linear manifold which is homeomorphic to $\R^{6g-6}$.  A point $[\nu]
\in \MF(S)$ is an equivalence class up to Whitehead moves of a
singular foliation $\nu$ on $S$ equipped with a transverse measure of
full support.  For detailed discussion of measured foliations and of
the space $\MF(S)$ see \cite{FLP}.

Piecewise linear charts of $\MF(S)$ correspond to sets of measured
foliations that are carried by a train track; we will now discuss the
construction of these charts in some detail.  While this material is
well-known to experts, most standard references that discuss train
track charts use the equivalent language of measured laminations,
whereas our primary interest in foliations arising from quadratic
differentials makes the direct consideration of foliations preferable.
Additional details of the carrying construction from this perspective
can be found in \cite{papadopoulos:reseaux}
\cite{mosher:train-track-expansions}.

A \emph{train track} on $S$ is a $C^1$ embedded graph in which all
edges incident on a given vertex share a tangent line at that point.
Vertices of the train track are called \emph{switches} and its edges
are \emph{branches}.  We consider only \emph{generic} train tracks in
which each switch is has three incident edges, two \emph{incoming} and
one \emph{outgoing}, such that the union of any incoming edge and the
outgoing edge forms a $C^1$ curve.

The complement of a train track is a finite union of subsurfaces with
cusps on their boundaries.  In order to give a piecewise linear chart
of $\MF(S)$, each complementary disk must have at least three cusps on
its boundary and each complementary annulus must have at least one
cusp.  We will always require this of the train tracks we consider.

If $\tau$ is such a train track, let $W(\tau)$ denote the vector space
of real-valued functions $w$ on its set of edges that obey the
relation $w(a) + w(b) = w(c)$ for any switch with incoming edges
$\{a,b\}$ and outgoing edge $c$.  This \emph{switch relation} ensures
that $w$ determines a signed transverse measure, or \emph{weight}, on
the embedded train track.  Within $W(\tau)$ there is the finite-sided
convex cone of nonnegative weight functions, denoted $\MF(\tau)$.  It
is this cone which forms a chart for $\MF(S)$.

A measured foliation is \emph{carried} by the train track $\tau$ if the
foliation can be cut open near singularities and along saddle
connections and then moved by an isotopy so that all of the leaves lie
in an arbitrarily small open neighborhood of $\tau$ and are nearly parallel to its
branches, as depicted in Figure \ref{fig:carrying}.  Here ``cutting
open'' refers to the procedure of replacing a union of leaf segments
and saddle connections coming out of singularities with a subsurface
with cusps on its boundary.  The result of cutting open a measured
foliation is a \emph{partial measured foliation} in which there are
\emph{non-foliated regions}, each of which has a union of leaf
segments of the original foliation as a spine.

\begin{figure}
\begin{center}
\includegraphics[width=0.7\textwidth]{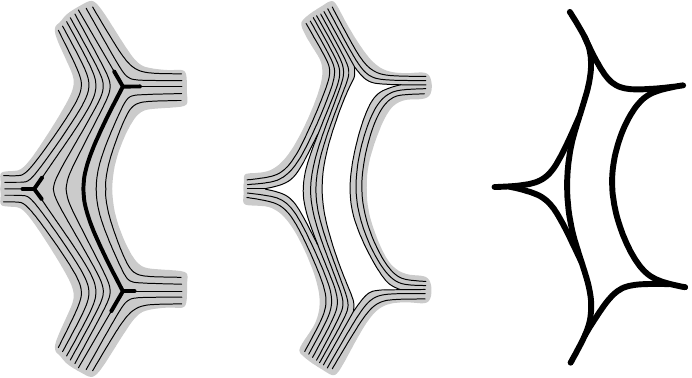}
\caption{Local picture of carrying: Cutting a measured foliation along
  saddle connections and singular leaf segments and inflating
  these to cusped subsurfaces allows an isotopy into a small
  neighborhood of the train track.\label{fig:carrying}}
\end{center}
\end{figure}

A measured foliation $\nu$ determines a weight on any train track that
carries it, as follows: For each branch $e \subset \tau$ choose a
\emph{tie} $r_e$, a short closed arc that crosses $e$ transversely at
an interior point and which is otherwise disjoint from $\tau$.  Now
select an open neighborhood of $U$ of $\tau$ that intersects each tie
$r_e$ in a connected open interval.  Let $\nu'$ be a partial measured
foliation associated to $\nu$ that has been isotoped to lie in $U$ and
to be transverse to each tie.  Note that each tie $r_e$ then has
endpoints in non-foliated regions of $\nu'$, since the endpoints of
$r_e$ lie outside $U$.  Let $w(e)$ be the total transverse measure
of $r_e$ with respect to $\nu'$.

The resulting function $w$ lies in $\MF(\tau)$ and regarding this
construction as a map $\nu \mapsto w$ gives a one-to-one correspondence
between equivalence classes of measured foliations that are carried by
$\tau$ and the convex cone $\MF(\tau)$.  Furthermore, these cones in
train track weight spaces form the charts of a piecewise linear atlas
on $\MF(S)$.

\subsection{The symplectic structure of $\MF(S)$}

The orientation of $S$ induces a natural antisymmetric bilinear map
$\omega_{\th}: W(\tau) \times W(\tau) \to \R$ on the space of weights
on a train track $\tau$.  This \emph{Thurston form} can be defined as
follows (compare \cite[Sec.~3.2]{penner-harer}
\cite[Sec.~3]{bonahon:shearing}): For each switch $v \in \tau$, let
$a_v, b_v$ be its incoming edges and $c_v$ its outgoing edge, where
$a_v,b_v$ are ordered so that intersecting $\{a_v, b_v, c_v\}$ with a
small circle around $v$ gives a positively oriented triple.  Then we
define
\begin{equation}
\label{eqn:thurston}
\omega_{\th}(w_1,w_2) = \frac{1}{2} \sum_v \det \begin{pmatrix}
w_1(a_v) & w_1(b_v)\\
w_2(a_v) & w_2(b_v)
\end{pmatrix}.
\end{equation}
If $\tau$ defines a chart of $\MF(S)$ then this form is nondegenerate,
and the induced symplectic forms on train track charts $\MF(\tau)$ are
compatible.  This gives $\MF(S)$ the structure of a piecewise linear
symplectic manifold.  

The Thurston form can also be interpreted as a homological
intersection number.  If $\tau$ can be consistently oriented then each
weight function $w$ defines a $1$-cycle $c_w = \sum_{e} w(e) \vec{e}$,
where $\vec{e}$ denotes the singular $1$-simplex defined by the
oriented edge $e$ of $\tau$.  In terms of these cycles, we have
$\omega_\th(w_1,w_2) = c_{w_1} \cdot c_{w_2}$.  For a general train
track, there is a branched double cover $\hat{S} \to S$ (with
branching locus disjoint from $\tau$) such that the preimage
$\hat{\tau} \subset \hat{S}$ is orientable.  Lifting weight functions
we obtain cycles $\hat{c}_{w_i} \in H^1(\hat{S},\R)$ such that
\begin{equation}
\label{eqn:thurston-intersection}
 \omega_\th(w_1,w_2) = \frac{1}{2}(\hat{c}_{w_1} \cdot \hat{c}_{w_2}).
\end{equation}

Note that if $\bar{S}$ denotes the opposite orientation of the surface
$S$, then there is a natural identification between measured foliation
spaces $\MF(S) \simeq \MF(\bar{S})$, but this identification does
\emph{not} respect the Thurston symplectic forms.  Rather, in
corresponding local charts we have $\omega_{\th}^S = -
\omega_{\th}^{\bar{S}}$.
 
\subsection{Dual trees}
  \label{sec:dual-trees}

  Let $\nu$ be a measured foliation on $S$ and $\Tilde{\nu}$ its
  lift to the universal cover $\Tilde{S}$.  There is a pseudo-metric
  $d$ on $\Tilde{S}$ where $d(x,y)$ is the minimum
  $\Tilde{\nu}$-transverse measure of a path connecting $x$ to $y$.
  The quotient metric space $T_\nu := \Tilde{S} / d^{-1}(0)$ is an
  $\R$-tree (see \cite{bowditch:trees-dendrons}
  \cite{morgan-shalen:free-actions}). The action of $\pi_1S$ on
  $\Tilde{S}$ by deck transformations determines an action on $T_\nu$
  by isometries.  The dual tree of the zero foliation $0 \in \MF(S)$
  is a point.

This pseudo-metric construction can be applied to the partial measured
foliation $\nu'$ obtained by cutting $\nu$ open along leaf segments from
singularities, as when $\nu$ is carried by a train track $\tau$.  The
result is a tree naturally isometric to $T_\nu$, which we identify with
$T_\nu$ from now on.  Non-foliated regions of $\Tilde{\nu}'$ are
collapsed to points in this quotient, so in particular each
complementary region of the lift $\Tilde{\tau}$ has a well-defined
image point $T_\nu$.

Similarly, the lift of a tie $r_e$ of $\tau$ to the universal cover
projects to a geodesic segment in $T_\nu$ of length $w(e)$; the
endpoints of this segment are the projections of the two complementary
regions adjacent to the lift of the edge $e$.

To summarize, we have the following relation between carrying and dual
trees:

\begin{prop}
\label{prop:carrying}
Let $\nu$ be a measured foliation carried by the train track $\tau$
with associated weight function $w$.  Let $\Tilde{\tau}$ denote the
lift of $\tau$ to the universal cover.  If $A,B$ are complementary
regions of $\Tilde{\tau}$ that are adjacent along an edge $\Tilde{e}$
of $\Tilde{\tau}$, and if $a,b$ are the associated points in $T_\nu$,
then we have
$$ w(e) = d(a,b)$$
where $d$ is the distance function of $T_\nu$.\noproof
\end{prop}

\subsection{Teichm\"uller space and quadratic differentials}

Let $\T(S)$ be the Teichm\"uller space of marked isomorphism
classes of complex structures on $S$ compatible with its orientation.
For any $X \in \T(S)$ we denote by $Q(X)$ the set of holomorphic
quadratic differentials on $X$, a complex vector space of dimension
$3g-3$.

Associated to $\phi \in Q(X)$ we have the following structures on $X$:
\begin{itemize}
\item The flat metric $|\phi|$, which has cone singularities at the
zeros of $\phi$,
\item The measured foliation $\F(\phi)$ whose leaves integrate the
distribution $\ker(\Im(\sqrt{\phi}))$, with transverse measure 
given by $|\Im(\sqrt{\phi})|$, and
\item The dual tree $T_\phi := T_{\F(\phi)}$ and the
$\pi_1S$-equivariant map $\pi : \Tilde{X} \to T_\phi$ that collapses
leaves of the lifted foliation $\Tilde{\F(\phi)}$ to points of $T_\phi$.
\end{itemize}

The dual tree construction is homogeneous with respect to the action
of $\R^+$ on $Q(X)$ in the sense that for any $c \in \R^+$ we have
$$ T_{c \phi} = c^{1/2} T_\phi$$
where the right hand side represents the metric space obtained from
$T_\phi$ by multiplying its distance function by $c^{1/2}$.

Note that the point $0 \in Q(X)$ is a degenerate case in which there
is no corresponding flat metric, and by convention $\F(0)$ is the
empty measured foliation whose dual tree is a point.

We say that a $|\phi|$-geodesic is \emph{nonsingular} if its interior
is disjoint from the zeros of $\phi$.  Choosing a local coordinate $z$
in which $\phi = dz^2$ (a \emph{natural coordinate} for $\phi$),
a nonsingular $|\phi|$-geodesic segment $I$ becomes a line segment in
the $z$-plane.  The vertical variation of this segment in $\C$
(i.e.~ $|\Im(z_2 - z_1)|$, where $z_i$ are the endpoints) is the
\emph{height} of $I$.

Note that leaves of the foliation $\F(\phi)$ are geodesics of the
$|\phi|$-metric.  Conversely, a nonsingular $|\phi|$-geodesic $I$ is
either a leaf of $\F(\phi)$ or it is transverse to $\F(\phi)$.  In the
latter case, the height $h$ of $I$ is equal to its $\F(\phi)$-transverse
measure, and any lift of $I$ to $\Tilde{X}$ projects homeomorphically
to a geodesic segment in $T_\phi$ of length $h$.

\section{The isotropic cone: Embeddings}
\label{sec:isotropic1}

The goal of this section is to establish the following result relating
$3$-manifold actions on $\Lambda$-trees and measured foliations:

\begin{thm}
\label{thm:embedding-isotropic}
Let $M$ be a $3$-manifold with connected boundary $S$.  There exists
an isotropic piecewise linear cone $\L_M \subset \MF(S)$ with the
following property: If $\nu$ is a measured foliation on $S$ whose dual
tree embeds isometrically and $\pi_1S$-equivariantly into a
$\Lambda$-tree $T$ equipped with an isometric action of $\pi_1M$, then
$[\nu] \in \L_M$.
\end{thm}

Here a \emph{piecewise linear cone} refers to a closed
$\R^+$-invariant subset of $\MF(S)$ whose intersection with any train
track chart $\MF(\tau)$ is a finite union of finite-sided convex cones
in linear subspaces of $W(\tau)$.  Such a cone is \emph{isotropic} if
the linear spaces can be chosen to be isotropic with respect to the
Thurston symplectic form.  Since transition maps between these charts
are piecewise linear and symplectic, it suffices to check these
conditions in any covering of the set by train track charts.

The first step in the proof of Theorem \ref{thm:embedding-isotropic}
will be to use the foliation and embedding to construct a weight
function on the $1$-skeleton of a triangulation of $M$.  We begin with
some generalities about train tracks, triangulations, and weight
functions.

\subsection{Complexes and weight functions}
  \label{sec:triangulations}

Let $\Delta$ be a simplicial complex, and let $\Delta^{(k)}$ denote
the its set of $k$-simplices.  Given an abelian group $G$, define
the space of \emph{$G$-valued weights on $\Delta$} as the $G$-module
consisting of functions $\Delta^{(1)} \to G$; we denote this space by
$$ W(\Delta,G) := G^{\Delta^{(1)}}. $$
The case $G = \R$ will be of primary interest and so we abbreviate
$W(G) := W(G,\R)$.  If $w \in W(\Delta,G)$ and $e \in \Delta^{(1)}$ we
say that $w(e)$ is the \emph{weight} of $e$ with respect to $w$.

A homomorphism of groups $\varphi: G \to G'$ induces a homomorphism of
weight spaces $\varphi_* : W(\Delta,G) \to W(\Delta,G')$, and an
inclusion of simplicial complexes $i : (M,\Delta) \to (M',\Delta')$
induces a $G$-linear restriction map $i^* : W(\Delta',G) \to
W(\Delta,G)$.  These functorial operations commute,
i.e.~$\varphi_*\circ i^* = i^* \circ \varphi_*$.

A map $f$ from $\verts{\Delta}$ to a $\Lambda$-metric space induces a
$\Lambda$-valued weight on $\Delta$ in a natural way: For each $e \in
\edges{\Delta}$ we define the weight to be the distance between the
$f$-images of its endpoints.  We write $w_f$ for the weight function
defined in this way.

This construction has a natural extension to equivariant maps on
regular covers.  Suppose that $\Tilde{\Delta}, \Delta$ are simplicial
complexes such that there is a regular covering $\pi : \Tilde{\Delta}
\to \Delta$ which is also a simplicial map.  Suppose also that the
deck group $\Gamma$ of this covering acts isometrically on a
$\Lambda$-metric space $E$.  Then if $f : \Tilde{\Delta} \to E$ is a
$\Gamma$-equivariant map, the resulting weight function $\Tilde{w}_f \in
W(\Tilde{\Delta},\Lambda)$ is also $\Gamma$-invariant, hence it
descends to a weight function $w_f \in W(\Delta,\Lambda)$ on the base
of the covering.

\subsection{Extending triangulations and maps}
  \label{sec:tree-map}

We will now consider the space of weight functions as defined above in
cases where the complex $\Delta$ is a triangulation of a $2$- or
$3$-manifold, possibly with boundary.

For example, let $\tau$ be a maximal, generic train track on a
surface $S$.  Then there is a triangulation $\Delta_\tau$ of $S$ dual
to the embedded trivalent graph underlying $\tau$.  Each triangle of
$\Delta_\tau$ contains one switch of the train track, each edge of
$\Delta_\tau$ corresponds to an edge of $\tau$, and each vertex of
$\Delta_\tau$ corresponds to a complementary region of $\tau$. The
correspondence between edges gives a natural (linear) embedding
$$ W(\tau) \into W(\Delta_\tau).$$

Now suppose that $\nu$ is a measured foliation on $S$ that is carried
by the train track $\tau$, so we consider the class $[\nu]$ as an
element of $\MF(\tau) \subset W(\tau)$.  Let $\Tilde{\tau}$ denote the
lift of $\tau$ to the universal cover $\Tilde{S}$.  As in Section
\ref{sec:dual-trees}, the carrying relationship between $\nu$ and
$\tau$ gives a map from complementary regions of $\tau$ to the dual
tree $T_\nu$.  In terms of the dual triangulation $\Delta :=
\Delta_\tau$, this is a map
$$ f : \verts{\Tilde{\Delta}} \to T_\nu,$$
and it is immediate from the definitions above and Proposition
\ref{prop:carrying} that the associated weight function $w_f \in
W(\Delta)$ is the image of $[\nu]$ under the embedding $W(\tau) \into
W(\Delta)$.

Let us further assume that, as in the hypotheses of Theorem
\ref{thm:embedding-isotropic}, there is an equivariant isometric
embedding of $T_\nu$ into a $\Lambda$-tree $T$ equipped with an action
of $\pi_1M$, where $M$ is a $3$-manifold with $\partial M = S$.  Using
this embedding we can consider the map $f$ constructed above as taking
values in $T$.
  We extend the triangulation $\Delta$ of $S$ to a
triangulation $\Delta_M$ of $M$, and the map $f$ to a
$\pi_1M$-equivariant map
$$ F : \verts{\Tilde{\Delta}_M} \to T.$$
Such an extension can be constructed by choosing a fundamental domain
$V$ for the $\pi_1M$-action on $\verts{\Tilde{\Delta}_M}$ and mapping
elements of $V \setminus \verts{\Tilde{\Delta}}$ to arbitrary points in
$T$.  Combining these with the values of $f$ on $\verts{\Tilde{\Delta}}$
and the free action of $\pi_1M$ gives a unique equivariant extension
to all of $\verts{\Tilde{\Delta}_M}$.

Associated to the map $F$ is the weight function $w_F \in
W(\Delta_M,\Lambda)$.  By construction, its values on the edges of
$\Delta$ are the coordinates of $[\nu]$ relative to the train track
chart of $\tau$, considered as elements of $\Lambda$ using the
embedding $\ordinc : \R \to \Lambda$ that is implicit in the isometric
map $T_\nu \to T$.

We record the constructions of this paragraph in the following
proposition.

\begin{prop}
\label{prop:tree-map}
Let $\nu$ be a measured foliation on $S = \partial M$ carried by a maximal generic
train track $\tau$, and let $\Delta_M$ be a triangulation of $M$
extending the dual triangulation of $\tau$.  Suppose that
there exists a $\Lambda$-tree $T$ equipped with an isometric action of
$\pi_1M$ and an equivariant isometric embedding
$$ h : T_\nu \to T,$$
relative to an order-preserving embedding $\ordinc: \R \to \Lambda$.
Then there exists a weight function $w \in W(\Delta_M,\Lambda)$ with
the following properties:
\begin{rmenumerate}
\item The weight $w$ is induced by an equivariant map $F :
\verts{\Tilde{\Delta}_M} \to T$
\item The restriction of $w$ to $\Delta_\tau$ is the image of $[\nu] \in
W(\tau)$ under the natural inclusion $W(\tau) \into W(\Delta_\tau)
\xrightarrow{\ordinc_*} W(\Delta_\tau,\Lambda)$.  \noproof
\end{rmenumerate}
\end{prop}

\subsection{Triangle forms and the symplectic structure}
In \cite{penner-papadopoulos:symplectic}, Penner and Papadopoulos
relate the Thurston symplectic structure of $\MF(S)$ for a punctured
surface $S$ to a certain linear $2$-form on the space of weights on a
``null-gon track'' dual to an ideal triangulation of $S$.  In this
section we discuss a related construction for a triangulation of a
compact surface dual to a train track.

Let $\sigma$ be an oriented triangle with edges $e,f,g$ (cyclically
ordered according to the orientation).  Let $de, df, dg$ denote the
corresponding linear functionals on $\R^{\{e,f,g\}}$, which evaluate
  a function on the given edge.  We call the alternating $2$-form 
$$ \omega_\sigma := -\frac{1}{2} \left ( de
\wedge df + df \wedge dg + dg \wedge de \right )$$ the \emph{triangle
  form} associated with $\sigma$.  Note that if $-\sigma$ represents
the triangle with the opposite orientation, then $\omega_{-\sigma} = -\omega_{\sigma}$.

Given a triangulation $\Delta$ of a compact oriented $2$-manifold $S$,
the triangle form corresponding to any $\sigma \in \Delta^{(2)}$ (with
its induced orientation) is naturally an element of $\alttwo
W(\Delta)^*$.  Denote the sum of these by
\begin{equation}
\label{eqn:omega-delta}
\omega_\Delta = \sum_{\sigma \in \Delta^{(2)}} \omega_\sigma.
\end{equation}
This $2$-form on $W(\Delta)$ is an analogue of the Thurston symplectic
form, in a manner made precise by the following:

\begin{lem}
\label{lem:symplectic-on-weights}
If $\tau$ is a maximal generic train track on $S$ with dual
triangulation $\Delta = \Delta_\tau$, then the Thurston form on
$W(\tau)$ is the pullback of $\omega_{\Delta}$ by the natural
inclusion $W(\tau) \to W(\Delta)$.
\end{lem}

\begin{proof}
By direct calculation: Both the Thurston form and
$\omega_{\Delta}$ are given as a sum of $2$-forms, one for each
triangle of $\Delta$ (equivalently, switch of $\tau$).  The image of
$W(\tau)$ in $W(\Delta)$ is cut out by imposing a switch condition for each
triangle $\sigma \in \Delta^{(2)}$, which for an appropriate labeling
of the sides as $e,f,g$ can be written as
$$de + df = dg.$$
On the subspace defined by this constraint the triangle form pulls
back to
$$ -\frac{1}{2} \left ( de \wedge df + df \wedge dg + dg \wedge de
\right ) = \frac{1}{2} de \wedge df$$
which is the associated summand in the Thurston form \eqref{eqn:thurston}.
\end{proof}

\subsection{Tetrahedron forms}

\begin{figure}
\begin{center}
\includegraphics[width=3.5in]{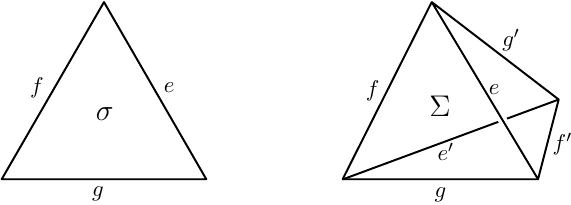}
\caption{Labeled edges of an oriented $2$-simplex $\sigma$ and an
  oriented $3$-simplex $\Sigma$.  The $2$-forms $\omega_\sigma$ and
  $\omega_\Sigma$ are defined in terms of these
  labels.\label{fig:tetrahedron}}
\end{center}
\end{figure}

Let $\Sigma$ be an oriented $3$-simplex.  Call a pair of edges of
$\Sigma$ \emph{opposite} if they do not share a vertex.  Label the edges of $\Sigma$
as $e,f,g,e',f',g'$ so that the following conditions are satisfied:
\begin{rmenumerate}
\item The pairs $\{e,e'\}$, $\{f,f'\}$, and $\{g,g'\}$ are opposite.
\item The ordering $e,f,g$ gives the oriented boundary of one of the
faces of $\Sigma$.
\end{rmenumerate}
An example of such a labeling is shown in Figure
\ref{fig:tetrahedron}.

Define the \emph{tetrahedron form} $\Omega_\Sigma \in \alttwo W(\Sigma)^*$ as
\begin{equation*}
\Omega_\Sigma  = -\frac{1}{2} \left ( d(e\!+\!e') \wedge
d(f\!+\!f') + d(f\!+\!f') \wedge d(g\!+\!g') +  d(g\!+\!g') \wedge d(e\!+\!e') \right )
\end{equation*}
Here we abbreviate $d(e+e') = de + de'$ and similarly for the other
edges.  It is easy to check that this $2$-form does not depend on the
labeling (as long as it satisfies the conditions above).  As in the
case of triangle forms, $\Omega_\Sigma$ is naturally a $2$-form on
the space of weights for any oriented simplicial complex containing
$\Sigma$.

A simple calculation using the definition above gives the following:
\begin{lem}
\label{lem:tet-is-sum-of-tri}
The tetrahedron form is equal to the sum of the triangle forms of its
oriented boundary faces, i.e.
$$ \Omega_\Sigma = \sum_{\sigma \in \partial \Sigma} \omega_\sigma. $$\noproof
\end{lem}

Now consider a triangulation $\Delta_M$ of an oriented $3$-manifold with
boundary $S$, and let $\Delta_S$ denote the induced triangulation of
the boundary.  Denote the sum of the tetrahedron forms by
$$ \Omega_{\Delta_M} = \sum_{\Sigma \in \tets{\Delta_M}} \!
\Omega_\Sigma \: \in \:
\alttwo W(\Delta_M)^*.$$
In fact, due to cancellation in this sum, the $2$-form defined above ``lives'' on the boundary:

\begin{lem}
\label{lem:symplectic-pullback}
The form $\Omega_{\Delta_M}$ is equal to the pullback of
$\omega_{\Delta_S}$ under the restriction map $W(\Delta_M) \to
W(\Delta_S)$.
\end{lem}

\begin{proof}
By Lemma \ref{lem:tet-is-sum-of-tri} we have
$$ \Omega_{\Delta_M} = \sum_{\Sigma \in \tets{\Delta_M\!}} \; \sum_{\sigma
  \in \partial \Sigma} \omega_\sigma.$$ In this sum, each interior
triangle of $\Delta_M$ appears twice (once with each orientation) and
so these terms cancel.  The remaining terms are the elements of
$\tris{\Delta_S}$ with the boundary orientation, so we are left with
the sum \eqref{eqn:omega-delta} defining $\omega_{\Delta_S}$.
The result is the pullback of $\omega_{\Delta_S}$ by the restriction map
because in the formula above, we are considering $\omega_\sigma$ as an
element of $\alttwo W(\Delta_M)^*$ rather than $\alttwo W(\Delta_S)^*$.
\end{proof}

\subsection{The four-point condition}

Given four points in a $\Lambda$-tree, Lemma
\ref{lem:four-points-lambda-tree} implies that there is always a
labeling $\{p,q,r,s\}$ of these points such that the distance function
satisfies
\begin{equation}
\label{eqn:four-point}
d(p,q) + d(r,s) = d(p,s) + d(r,q).
\end{equation}
We call this the \emph{weak four-point condition} to distinguish it
from the stronger four-point condition of Lemma
\ref{lem:four-points-lambda-tree} which also involves an inequality.

If we think of $p,q,r,s$ as labeling the vertices of a $3$-simplex
$\Sigma$, then the pairwise distances give a weight function $w :
\edges{\Sigma} \to \Lambda$.  Condition \eqref{eqn:four-point} is
equivalent to the existence of opposite edge pairs $\{e,e'\}, \{f,f'\}
\subset \edges{\Sigma}$ such that
\begin{equation}
\label{eqn:weight-four-point}
 w(e) + w(e') = w(f) + w(f').
\end{equation}

Given a simplicial complex $\Delta$, let
$W_4(\Delta,\Lambda)$ denote the set of $\Lambda$-valued weights such
that in each $3$-simplex of $\Delta$ there exist opposite edge pairs
so that \eqref{eqn:weight-four-point} is satisfied.

The following basic properties of $W_4(\Delta,\Lambda)$ follow
immediately from the definition of this set (and the relation between
the four-point condition and $4$-tuples in $\Lambda$-trees):

\begin{lem}\mbox{}
\label{lem:four-point-properties}
\begin{rmenumerate}
\item The set $W_4(\Delta,\Lambda)$ is a finite union of subspaces
(i.e.~$\Lambda$-submodules) of $W(\Delta,\Lambda)$; each subspace
corresponds to choosing opposite edge pairs in each of the $3$-simplices of $\Delta$.

\item If $\varphi: \Lambda \to \Lambda'$ is a homomorphism, then we
have $\varphi_*( W_4(\Delta,\Lambda)) \subset W_4(\Delta,\Lambda')$.

\item If $f : \verts{\Tilde{\Delta}} \to T$ is an equivariant map to a
$\Lambda$-tree, then $w_f \in W_4(\Delta,\Lambda)$
\end{rmenumerate}
\noproof
\end{lem}

Ultimately, the isotropic condition in Theorem
\ref{thm:embedding-isotropic} arises from the following property of
the set $W_4(\Delta) = W_4(\Delta,\R)$:

\begin{lem}
\label{lem:four-point-isotropic}
Let $M$ be an oriented $3$-manifold and $\Delta_M$ a triangulation.
Then $W_4(\Delta_M)$ is a finite union of $\,\Omega_{\Delta_M}$-isotropic
subspaces of $W(\Delta_M)$. 
\end{lem}

\begin{proof}
Let $V \subset W(\Delta_M)$ be one of the subspaces comprising
$W_4(\Delta_M)$ (as in Lemma \ref{lem:four-point-properties}.(i)).
Then for each $\Sigma \in \tets{\Delta_M}$ we have opposite edge pairs
$\{e,e'\}$ and $\{f,f'\}$ such that \eqref{eqn:weight-four-point}
holds, or equivalently, on the subspace $V$ the equation
$$ d(e+e') = d(f+f') $$
is satisfied.  Substituting this into the definition of the
tetrahedron form $\Omega_\Sigma$ gives zero.  Since
$\Omega_{\Delta_M}$ is the sum of these forms, the subspace $V$ is
isotropic.
\end{proof}

\subsection{Construction of the isotropic cone}
\label{sec:construction-of-isotropic}

We now combine the results on triangulations, weight functions, and
the symplectic structure of $\MF(S)$ with the constructions of
Proposition \ref{prop:tree-map} to prove Theorem
\ref{thm:embedding-isotropic}.

\begin{proof}[{Proof of Theorem \ref{thm:embedding-isotropic}}]
Let $\Upsilon$ be a finite set of maximal, generic train tracks such
that any measured foliation on $S$ is carried by one of them.
For each $\tau \in \Upsilon$, let $\Delta_M^\tau$ be an extension of
$\Delta_\tau$ to a triangulation of $M$.

Define
$$ \L_\tau = i^*(W_4(\Delta_M^\tau)) \cap \ML(\tau)$$
where $i^* : W(\Delta_M^\tau) \to W(\Delta_\tau)$ is the restriction
map (i.e. the map that
restricts a weight to the edges that lie on $S$).  That is, an element
of $\L_\tau$ is a measured foliation carried by $\tau$ whose associated weight
function on $\Delta_\tau$ can be extended to $\Delta_M^\tau$ in such a
way that it satisfies the weak $4$-point condition in each simplex.

Let $\L_M = \bigcup_{\tau \in \Upsilon} \L_\tau$.  By Lemmas
\ref{lem:symplectic-on-weights} and \ref{lem:four-point-isotropic},
the set $\L_M$ is an isotropic piecewise linear cone in $\MF(S)$. 
We need only show that for $\nu$ and $T_\nu \into T$ as in the statement of
the Theorem we have $[\nu] \in \L_M$.

Given such $\nu$ and $T_\nu \into T$, let $\tau \in \Upsilon$ carry $\nu$
and abbreviate $\Delta_M = \Delta_M^\tau$.  Let $F :
\verts{\Tilde{\Delta}_M} \to T$ and $w = w_F \in W(\Delta_M,\Lambda)$
be the map and associated weight function given by Proposition
\ref{prop:tree-map}.  By Lemma \ref{lem:four-point-properties}.(iii),
we have $w \in W_4(\Delta_M,\Lambda)$.

Let $\varphi : \Lambda \to \R$ be a left inverse to the inclusion
$\ordinc : \R \to \Lambda$ associated with the isometric embedding
$T_\nu \into T$; such a map exists by Lemma \ref{lem:left-inverse}.
Then $\varphi_* : W(\Delta_\tau,\Lambda) \to W(\Delta_\tau)$ is
correspondingly a left inverse to $\ordinc_* : W(\Delta_\tau) \into
W(\Delta_\tau,\Lambda)$.  Since by Proposition
\ref{prop:tree-map}.(ii) we have that $i^*(w) \in
W(\Delta_\tau,\Lambda)$ is the image of $[\nu] \in \ML(\tau)$ under
this inclusion, it follows that $\varphi_*(i^*(w)) =
i^*(\varphi_*(w))$ also represents $[\nu]$.

By Lemma \ref{lem:four-point-properties}.(ii) we have $\varphi_*(w)
\in W_4(\Delta_M)$, so we have shown that $[\nu] \in
i^*(W_4(\Delta_M^\tau)) \cap \ML(\tau) = \L_\tau \subset \L_M$, as
desired.
\end{proof}

\section{The isotropic cone: Straight maps and length functions}
\label{sec:isotropic2}

In this section we introduce refinements of Theorem
\ref{thm:embedding-isotropic} that will be used in the proof of the
main theorem.  These refinements replace with isometric embedding
hypothesis of Theorem \ref{thm:embedding-isotropic} with weaker
conditions relating the trees carrying actions of $\pi_1S$ and
$\pi_1M$.

\subsection{Straight maps}

We first recall (and generalize) the notion of a straight map, which
is a certain type of morphism of trees.  

Let $X \in \T(S)$ be a marked Riemann surface structure on $S$ and
$\phi \in Q(X)$ a holomorphic quadratic differential.  Recall that
there is a dual $\R$-tree $T_\phi$ and projection $\pi : \Tilde{X} \to
T_\phi$, and that nonsingular $|\phi|$-geodesics in $\Tilde{X}$
project to geodesics in $T_\phi$.  Let $\I_\phi$ denote the set of all
geodesics in $T_\phi$ that arise in this way (including both segments
and complete geodesics).

Let $T$ be an $\R$-tree.  Following \cite{dumas:holonomy}, we say that
a map $f : T_\phi \to T$ is \emph{straight} if it is an isometric
embedding when restricted to any element of $\I_\phi$.  Thus, for
example, an isometric embedding of $T_\phi$ is straight map, but the
converse does not hold (see e.g.~\cite[Lem.~6.5]{dumas:holonomy}).

Note that straightness of a map $T_\phi \to T$ depends on the
differential $\phi$ and not just on the isometry type of the dual
tree; Figure \ref{fig:straightness} shows an example of differentials
with isometric dual trees but distinct notions of straightness.

\begin{figure}
\begin{center}
\includegraphics[width=0.7\textwidth]{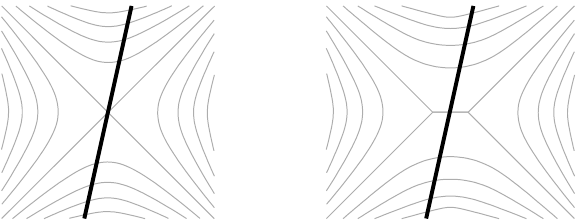}
\caption{Quadratic differentials with isometric dual trees may induce
  different notions of straight mapping: The local foliation
  pictures shown here have isometric leaf spaces, but the indicated path is
  required to map isometrically by a straight map in one case (right)
  but not in the other (left).\label{fig:straightness}}
\end{center}
\end{figure}

More generally, if $T$ is a $\Lambda$-tree, we say that a map $f :
T_\phi \to T$ is straight if there is an order-preserving map $\ordinc
: \R \to \Lambda$ such that the restriction of $f$ to each element of
$\I_\phi$ is an isometric embedding with respect to $\ordinc$.  As in
the case of $\R$-trees, isometric embeddings (now in the sense of
section \ref{sec:lambda}) are examples of straight maps.

For the degenerate case $\phi = 0$, we make the convention that any
map of the point $T_0$ to a $\Lambda$-tree is straight.

\subsection{Isotropic cone for straight maps}

In the following generalization of Theorem
\ref{thm:embedding-isotropic} we fix a Riemann surface structure on $S
= \partial M$ and consider straight maps instead of isometric embeddings.

\begin{thm}
\label{thm:straight-isotropic}
Let $M$ be an oriented $3$-manifold with connected boundary $S$, and
let $X \in \T(S)$ be a marked Riemann surface structure on $S$. There
exists an isotropic cone $\L_{M,X} \subset \MF(S)$ with the following
property: If $\phi \in Q(X)$ is a holomorphic quadratic differential
such that there exists a $\Lambda$-tree $T$ equipped with an isometric
action of $\pi_1M$ and a $\pi_1S$-equivariant straight map
$$ h : T_\phi \to T$$
then $[\F(\phi)] \in \L_{M,X}$.
\end{thm}

In the proof of Theorem \ref{thm:embedding-isotropic}, the assumption
that the map $T_\phi \to T$ is isometric embedding was only used
through its role in the construction of Section \ref{sec:tree-map}: A
train track carrying $\nu$ gives a map $f : \verts{\Tilde{\Delta}_\tau}
\to T_\nu$ whose associated weight function represents $[\nu]$, and
since $h : T_\nu \to T$ is an isometric embedding, the composition $h
\circ f$ has the same associated weight.

Attempting to reproduce this with the weaker hypotheses of Theorem
\ref{thm:straight-isotropic}, we can again choose a train track $\tau$
carrying $\F(\phi)$ and construct a map $f :
\verts{\Tilde{\Delta}_\tau} \to T_\phi$.  We would then like to
compose $f$ with the straight map $h : T_\phi \to T$ without changing
the associated weight function.  This will hold if the segments in
$T_\phi$ corresponding to the ties of $\tau$ are mapped isometrically
by $h$, so it is enough to know that they correspond to nonsingular
$|\phi|$-geodesic segments in $\Tilde{X}$.  To summarize, we have:

\begin{prop}
\label{prop:tree-straight-map}
Let $\tau$ be a train track that carries $\F(\phi)$ 
such that each tie of $\Tilde{\tau}$
corresponds to a nonsingular $|\phi|$-geodesic
segment in $\Tilde{X}$.  Let $\Delta_M$ be a triangulation of $M$
extending the dual triangulation of $\tau$.  Suppose that there exists
a $\Lambda$-tree $T$ equipped with an isometric action of $\pi_1M$ and
a $\pi_1S$-equivariant straight map
$$ h : T_\phi \to T,$$
relative to an order-preserving embedding $\ordinc: \R \to \Lambda$.  Then there
exists a weight function $w \in W(\Delta_M,\Lambda)$ satisfying
conditions (i)--(ii) of Proposition \ref{prop:tree-map}.  \noproof
\end{prop}

Therefore, while we used an arbitrary finite collection of train track
charts covering $\MF(S)$ in the previous section, we now have a
stronger condition that the carrying train track must satisfy.  The
existence of a suitable finite collection of train tracks that cover
$Q(X)$ is given by:

\begin{lem}
\label{lem:delaunay-track}
For each nonzero $\phi \in Q(X)$ there exists a triangulation $\Delta$
of $X$ and a maximal train track $\tau$ such that:
\begin{rmenumerate}
\item The vertices $\Delta$ are zeros of $\phi$, the edges are
saddle connections of $\phi$, and the triangulation $\Delta$ is dual to
the train track $\tau$ in the sense of Section
\ref{sec:triangulations},
\item The foliation $\F(\phi)$ is carried by $\tau$ in such a way each
edge $e$ of $\Delta$ becomes a tie of the corresponding edge of
$\tau$; in particular,
\item The $\phi$-heights of the edges of $\Delta$ give the weight
function on $\tau$ representing $[\F(\phi)]$.
\end{rmenumerate}
Furthermore, there is a finite set of pairs $(\Delta, \tau)$ such that
the triangulation and train track constructed above can always be
chosen to be isotopic to an element of this set.
\end{lem}

\begin{proof}
Consider a Delaunay triangulation $\Delta$ of the singular Euclidean
surface $(X,|\phi|)$ with vertices at the zeros of $\phi$, as in
\cite[Sec.~4]{masur-smillie}.  Such a triangulation has nonsingular
$|\phi|$-geodesic segments as edges and is defined by the condition
that each triangle has a circumcircle (with respect to the singular
Euclidean structure $|\phi|$) which is ``empty'', i.e.~has no zeros of
$\phi$ in its interior.  There are only finitely many Delaunay
triangulations of a given singular Euclidean surface, and for
generic $X$ and $\phi$ there is a unique one.

First suppose that this triangulation $\Delta$ has no horizontal
edges.  Each triangle has two ``vertically short'' edges whose heights
sum to that of the third edge, and we construct a train track $\tau$
by placing a switch in each triangle so that the incoming branches at
the switch are dual to the short edges of the triangle (as shown in
Figure \ref{fig:delaunay-track}).  The complementary regions of $\tau$
are disk neighborhoods of the vertices of the triangulation, so $\tau$
is maximal.  Thus $\Delta, \tau$ satisfy condition (i).

After cutting along leaf segments near singularities, an isotopy
pushes the leaves of $\F(\phi)$ into a small neighborhood of the train
track.  Throughout this isotopy the image of an edge $e$ of the
triangulation in the dual tree remains the same, and so it corresponds
to the tie $r_e$ of the train track.  The length of the image segment
in $T_\phi$ is the height of the geodesic edge, so properties
(ii)--(iii) follow.

It remains to consider the possibility that the Delaunay triangulation
has horizontal edges.  In this case we can still form a dual train
track but it is not clear whether the dual to a horizontal edge should
be incoming or outgoing at the switch in a given triangle.  To
determine this, we consider a slight deformation of $\phi$ to a
quadratic differential $\phi'$ with the same zero structure but no
horizontal saddle connections.  (A generic deformation preserving the
multiplicities of zeros will have this property.)  For a small enough
deformation, the same combinatorial triangulation can be realized
geodesically for $\phi'$, and the heights of the previously horizontal
edges determine how to form switches for $\tau$.

Finally we show that only finitely many isotopy classes of pairs
$(\Delta,\tau)$ arise from this construction.  In fact, it suffices to
consider $\Delta$ alone since filling in the train track
$\tau$ involves only finitely many choices (incoming and outgoing
edges for each switch).  

The construction of $\Delta$ is independent of scaling $\phi$ so we
assume that $|\phi|$ has unit area, i.e.~$\|\phi\| = 1$ where
$\|\param\|$ is the $L^1$ norm.  The resulting family of metrics (the
unit sphere in $Q(X)$) is compact, and in particular the diameters of
these spaces are uniformly bounded.  By \cite[Thm.~4.4]{masur-smillie}
this diameter bound also gives an upper bound, $R$, on the length of
each edge of the Delaunay triangulation.  The number of zeros of
$\Tilde{\phi}$ in a ball of $|\phi|$-radius $R$ in $\Tilde{X}$ is
uniformly bounded (again by compactness of the family of metrics, and
the fixed number of zeros of $\phi$ on $X$), so the edges that appear
in $\Delta$ belong to finitely many isotopy classes of arcs between
pairs of zeros.  Thus, up to isotopy, only finitely many
triangulations can be constructed of these arcs.
\end{proof}

\begin{figure}
\begin{center}
\includegraphics[width=0.75\textwidth]{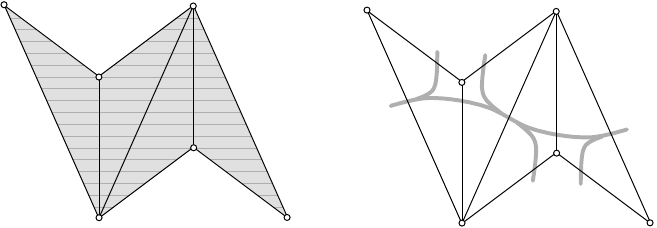}
\caption{A geodesic triangulation for a quadratic differential with
  vertices at the zeros and the associated train track carrying
  the measured foliation.\label{fig:delaunay-track}}
\end{center}
\end{figure}

With these preliminaries in place, it is straightforward to generalize
the proof Theorem \ref{thm:embedding-isotropic}:

\begin{proof}[{Proof of Theorem \ref{thm:straight-isotropic}}]
Let $\Upsilon_X$ denote the finite set of train tracks given by Lemma
\ref{lem:delaunay-track}, and extend each dual triangulation
$\Delta_\tau$ to a triangulation $\Delta_M^\tau$ of $M$.  Define
$$ \L_{M,X} = \bigcup_{\tau \in \Upsilon_X} \L_{\tau},$$
where $\L_{\tau} = i^*(W_4(\Delta_M^\tau)) \cap \ML(\tau)$.  As
before, Lemmas \ref{lem:symplectic-on-weights} and
\ref{lem:four-point-isotropic} show that this set is an isotropic
cone in $\MF(S)$.

If $\phi \in Q(X)$ and $h : T_\phi \to T$ is a straight map as in the
statement of the Theorem, then by Lemma \ref{lem:delaunay-track} and
Proposition \ref{prop:tree-straight-map} we have a train track $\tau
\in \Upsilon_X$ and weight function $w \in W(\Delta_M^\tau,\Lambda)$
such that $i^*(\varphi_*(w)) \in W(\Delta_\tau)$ represents $[\F(\phi)]$.
Here we retain the notation of the previous section, i.e.~$\varphi$
denotes a left inverse of $\ordinc : \R \to \Lambda$ and $i^* :
W(\Delta_M^\tau) \to W(\Delta_\tau)$ restricts a weight function to
the boundary triangulation.

By Proposition \ref{prop:tree-straight-map}, the weight $w$ is associated
to a map $\Tilde{\Delta}_M^{\tau (0)} \to T$.  As in the proof of
Theorem \ref{thm:embedding-isotropic}, this implies $w \in
W_4(\Delta_M^\tau,\Lambda)$ and therefore $i^*(\varphi_*(w)) \in
\L_\tau$.  We conclude $[\F(\phi)] \in \L_{M,X}$.
\end{proof}

\subsection{Isotropic cone for length functions}
  \label{sec:length-isotropic}

In this section we introduce a further refinement to the isotropic
cone construction that addresses special properties of abelian actions
of groups on $\R$-trees (which are described below).

Keeping the notations $M, S, X$ of the previous section, suppose that
we have $\phi \in Q(X)$ and a $\pi_1S$-equivariant straight map
$h : T_\phi \to T$ as in Theorem \ref{thm:straight-isotropic}.  Then
the image $h(T_\phi) \subset T$ is naturally an $\R$-tree preserved by $\pi_1S$.
Let $\ell : \pi_1S \to \R$ denote the translation length function of
this action and write $T_\ell = h(T_\phi)$.  Then $T_\ell$ is the intermediate
step in a factorization of $h$ as a straight map followed by an
isometric embedding:
$$
\begin{tikzpicture}[baseline=(current bounding box.center)]
\matrix (m) [matrix of math nodes, column sep=4em, row sep = 1.5em, text height=1.5ex, text depth=0.25ex]
{ T_\phi & T_\ell & T\\ };
\path[->>,font=\scriptsize] (m-1-1) edge node[above] {straight} (m-1-2);
\path[right hook->,font=\scriptsize] (m-1-2) edge node[above] {embed} (m-1-3);
\end{tikzpicture}
$$
Theorem \ref{thm:straight-isotropic} shows that this situation forces
$[\F(\phi)]$ to lie in an isotropic cone.

The generalization we now consider is to replace $T_\ell$ with a pair
of trees $T_\ell, T_\ell'$ on which $\pi_1S$ acts with the same length
function $\ell$---we say these actions are \emph{isospectral}.
We suppose that one of these is the image of a straight map while the
other isometrically embeds in a $\Lambda$-tree $T$ with a $\pi_1M$ action.  From this
weaker connection between $T_\phi$ and $T$, i.e.
$$
\begin{tikzpicture}[baseline=(current bounding box.center)]
\matrix (m) [matrix of math nodes, column sep=5em, row sep = 1.5em, text height=1.5ex, text depth=0.25ex]
{ T_\phi & T_\ell & T_\ell' & T\\ };
\path[->>,font=\scriptsize] (m-1-1) edge node[above] {straight} (m-1-2);
\path[<->,dotted,font=\scriptsize] (m-1-2) edge node[above] {isospectral}(m-1-3);
\path[right hook->,font=\scriptsize] (m-1-3) edge node[above] {embed} (m-1-4);
\end{tikzpicture}
$$
we can still conclude $[\F(\phi)] \in \L_{M,X}$.  The
following theorem makes this precise.

\begin{thm}
\label{thm:length-isotropic}
Let $T$ be a $\Lambda$-tree on which $\pi_1M$ acts.  Let
$T_\ell,T_\ell'$ be $\R$-trees on which $\pi_1S$ acts minimally with
length function $\ell$.  Let $\phi \in Q(X)$ be a holomorphic
quadratic differential such that there exists a $\pi_1S$-equivariant
straight map $h : T_\phi \to T_\ell$ and a $\pi_1S$-equivariant
isometric embedding $k : T_\ell' \to T$.  Then $[\F(\phi)] \in
\L_{M,X}$.
\end{thm}

Evidently this theorem would follow directly from Theorem
\ref{thm:straight-isotropic} if the isospectrality condition
implied the existence of an isometry $T_\ell \simeq T_\ell'$, for this
isometry would allow $h$ and $k$ to be composed, giving a straight map
$T_\phi \to T$.  This approach works for some length functions but not for
others, so before giving the proof we discuss the relevant dichotomy.

\subsection{Abelian and non-abelian actions}

Recall that an isometric action of a group $\Gamma$ on an $\R$-tree is
called \emph{abelian} if the associated length function has the form
$\ell(g) = |\chi(g)|$ where $\chi: \Gamma \to \R$ is a homomorphism;
otherwise, the action (or length function) is called
\emph{non-abelian}.  We have the following fundamental result of
Culler and Morgan:

\begin{thm}[\cite{culler-morgan}]
\label{thm:culler-morgan}
Let $T_\ell, T_\ell'$ be $\R$-trees equipped with minimal, isospectral
actions of a group $\Gamma$.  If the length function $\ell$ is non-abelian,
then there is an equivariant isometry $T_\ell \to T_\ell'$.
\end{thm}

As remarked above, this shows that the conclusion of Theorem
\ref{thm:length-isotropic} follows from Theorem
\ref{thm:straight-isotropic} whenever the length function is
non-abelian.  Thus we assume from now on that $\ell = |\chi|: \pi_1S
\to \R$ is an abelian length function.  In this case there may be many
non-isometric trees on which $\pi_1S$ acts with this length function
\cite{brown:bns}.

An \emph{end} of an $\R$-tree is an equivalence class of rays, where
two rays are equivalent if their intersection is a ray.
An abelian action of $\pi_1S$ on an $\R$-tree $T_\ell$ has a fixed end
(see \cite[Cor.~2.3]{culler-morgan} or \cite[Thm.~7.5]{alperin-bass}).
The fixed end has an associated \emph{Busemann function}
$\beta : T_\ell \to \R$, that intertwines the action of $\pi_1S$ on
$T_\ell$ with the translation action on $\R$ induced by $\chi$
\cite[Thm.~7.6]{alperin-bass}.  Furthermore the function $\beta$ is
unique up to adding a constant.  Here we use the term ``Busemann
function'' following e.g.~\cite[Sec.~2]{levitt:bns}, which is consistent
with its use in the theory of metric spaces of non-positive curvature
\cite[Sec.~II.8]{bridson-haefliger}; the same object is called an
\emph{end map} and discussed in \cite[Sec.~2.3]{chiswell},
while the above-cited result in \cite{alperin-bass} simply calls the
map $\alpha$.

The following result from \cite[Sec.~6]{dumas:holonomy} shows that composition with
a Busemann function preserves straightness of maps from $T_\phi$:

\begin{lem}
\label{lem:straight-to-R}
Let $T_\ell$ be an $\R$-tree equipped with an abelian action of
$\pi_1S$ by isometries, and let $\beta : T \to \R$ denote a
Busemann function of a fixed end.  If $h : T_\phi \to T_\ell$ is an
equivariant straight map, then $\beta \circ h : T_\phi \to \R$ is also straight.
\noproof
\end{lem}

Effectively this result will allow us to replace $T_\ell$ with $\R$ in
the hypotheses of Theorem \ref{thm:length-isotropic} since any
straight map can be composed with a Busemann function, preserving
straightness and without changing the length function.

In the proof of Theorem \ref{thm:straight-isotropic}, the straightness
of $h : T_\phi \to T$ was only used to conclude that the map is
isometric when applied to the endpoints of each segment in $T_\phi$
that corresponds to one of the $\phi$-geodesic edges of a
triangulation of $X$ (furnished by Lemma \ref{lem:delaunay-track}).
To generalize the proof to the situation of Theorem
\ref{thm:length-isotropic}, it will therefore suffice to show that if
there exists a straight map $T_\phi \to T_\ell$, then there also
exists a partially-defined map $T_\phi \dashrightarrow T_\ell'$ that
is ``locally straight'' in that it is an isometry when applied to the
endpoints of any of these segments.  Since these segments arise from
lifting the finite set of edges of a triangulation of $X$, they lie in
finitely many $\pi_1S$-equivalence classes.  Thus, Theorem
\ref{thm:length-isotropic} is reduced to:

\begin{thm}
\label{thm:local-straightness}
Let $T_\ell$ and $T'_\ell$ be $\R$-trees on which $\pi_1S$ acts
minimally and isospectrally, with abelian length function $\ell$, and
suppose $h : T_\phi \to T_\ell$ is an equivariant straight map, for
some $\phi \in Q(X)$.

Let $\I \! \subset \! \I_\phi$ be a set of segments in
$T_\phi$ that arise from nonsingular $\phi$-geodesic segments in
$\Tilde{X}$ and suppose $\I\!$ contains only finitely many
$\pi_1S$-equivalence classes.  Let $E \subset T_\phi$ be the set of
endpoints of elements of $\I$.

Then there exists an equivariant map $h': E \to T_\ell'$ such that
for any segment $J \in \I$ with endpoints $x,y$, we have
$$ d(h(x),h(y)) = d(h'(x),h'(y)).$$
\end{thm}

The proof will depend on properties of a certain endomorphism of the
tree $T_\ell'$ related to the end fixed by $\pi_1S$.

Given an $\R$-tree $T$ and an end $e$, for any $x \in T$ and $s\geq0$
let $P_s(x)$ denote the point on the ray from $x$ to $e$ such that
$d(x,P_s(x)) = s$.  Then $P_s : T \to T$ is a weakly contracting map,
and if $\pi_1S$ acts on $T$ fixing $e$, then $P_s$ is
$\pi_1S$-equivariant.  We call $P_s$ the \emph{pushing map} of
distance $s$ for the end $e$.

\begin{lem}
\label{lem:pushing}
Let $T$ be an $\R$-tree and let $\beta : T \to \R$ denote a Busemann
function of an end $e$.  Then for any $p,q \in T$ there exists $s_0\geq0$
such that $d(P_s(p),P_s(q)) = |\beta(p) - \beta(q)|$ for all $s \geq
s_0$, where $P_s : T \to T$ is the pushing map for the end $e$.
\end{lem}

\begin{proof}
The ray from $p$ to $e$ and the ray from $q$ to $e$ overlap in a ray
from $o$ to $e$, where $o$ is a point on the geodesic segment from $p$ to $q$, and
the Busemann function satisfies $|\beta(p) - \beta(q)| = | d(p,o) -
d(q,o) |$.  Let $r(t)$ parameterize the ray from $o$ to $e$, $t \geq
0$.  Then for $s \geq \max(d(p,o), d(q,o))$ we have $P_s(p) = r(s -
d(p,o))$ and $P_s(q) = r(s - d(q,o))$.  Since $r$ is an isometry onto
its image, we have $d(P_s(p),P_s(q)) = |d(p,o) - d(q,o)|$.
\end{proof}

Using the pushing map we can now give the

\begin{proof}[{Proof of Theorem \ref{thm:local-straightness}}]
Enlarging $\I$ if necessary, we can take this set and its set of
endpoints $E$ to be $\pi_1S$-invariant.

Let $E_0 \subset E$ be a finite subset containing exactly one point from
each $\pi_1S$-orbit in $E$.  Let $\beta,\beta'$ be Busemann
functions of the fixed ends of $\pi_1S$ acting on $T_\ell, T_\ell'$.
  For each
$x \in E_0$, choose a point $g'(x) \in T_\ell'$ such that $\beta'(g'(x)) =
\beta(h(x))$, giving a map $g' : E_0 \to T_\ell'$.  This is possible since
the map $\beta' : T' \to \R$ admits a section, e.g.~any complete
geodesic $\R \to T'$ that extends a ray representing the fixed end.

Using the action of $\pi_1S$ on $T_\ell$, we extend $g'$ to
an equivariant map $g' : E \to T'$ which then satisfies $\beta'(g'(x))
= \beta(h(x))$ for all $x \in E$.

For any $s \geq 0$ let $h'_s(x) = P_s'(g'(x))$ where $P'_s : T_\ell' \to T_\ell'$
is the pushing map for the fixed end of $\pi_1S$. By Lemma
\ref{lem:straight-to-R}, for any segment $J \in \I$ with
endpoints $x,y \in E$ we have
$$d(h(x),h(y)) = |\beta(h(x)) - \beta(h(y))| = | \beta'(g'(x)) -
\beta'(g'(y)) |$$
and by Lemma \ref{lem:pushing} there exists $s_J \geq 0$ such that for
all $s\geq s_J$ we have
$$ d(h'_s(x), h'_s(y)) = |\beta(g'(x)) - \beta(g'(y))| =
d(h(x),h(y)).$$
Taking $s$ larger than the maximum of $s_J$ as $J$ ranges over a
finite set representing each $\pi_1S$-orbit in $\I$, the
above condition holds for each such representative, and by
equivariance, for each $J \in \I$.  Then $h' = h'_s : E \to T_\ell'$ is the
desired map.
\end{proof}

As remarked above this completes the proof of Theorem \ref{thm:length-isotropic}.

\subsection{Application: Floyd's Theorem}
  \label{sec:floyd}

In this section we discuss some context for Theorems
\ref{thm:embedding-isotropic}, \ref{thm:straight-isotropic}, and
\ref{thm:length-isotropic}.
The contents of this section are not used in the sequel.

Theorem \ref{thm:embedding-isotropic} and its refinements can be seen as
generalizations of the following theorem of Floyd \cite{floyd:boundary-curves}:

\begin{thm}
\label{thm:floyd}
Let $M$ be a compact, irreducible $3$-manifold with boundary $S$.
Then the set of boundary curves of two-sided incompressible,
$\partial$-incompressible surface in $M$ is contained in a 
finite union of half-dimensional piecewise linear cells in $\MF(S)$.
\end{thm}

Note that the correspondence between measured laminations and measured
foliations on surfaces (see e.g.~\cite{Levitt}) allows us to consider
the boundary of a surface in $M$ as an element of $\MF(S)$. 
The original statement in \cite{floyd:boundary-curves} uses the
language of measured laminations.

Floyd's theorem answers a question of Hatcher, who established a
similar result for manifolds with torus boundary
\cite{hatcher:boundary-curves}.  Hatcher's theorem is often used
through its corollary that a knot complement manifold has finitely
many boundary slopes.

In both cases the half-dimensional set is constructed as an isotropic
cone in the symplectic space $\MF(S)$, and these results can be
compared to the more elementary (co)homological version: As a
consequence of Poincar\'e duality, the image of the connecting map
$$H_2(M,\partial M) \xrightarrow{\delta} H_1(\partial M)$$
is isotropic with respect to the intersection pairing.  Dually, the
image of the map $H^1(M) \to H^1(\partial M)$ induced by inclusion of
the boundary is isotropic for the cup product.

To show the connection with our results, we derive Floyd's theorem
from Theorem \ref{thm:embedding-isotropic} under the additional
assumption that the boundary $S$ is incompressible:

\begin{proof}[{Proof of Theorem \ref{thm:floyd}} (incompressible
boundary case)]
Let $F$ be an incompressible and $\partial$-incompressible surface in
$M$.  The preimage $\Tilde{F}$ of $F$ in $\Tilde{M}$ is a collection
of planes, separating $\Tilde{M}$ into a countable family of
complementary regions.  The adjacency graph of these regions, with one
vertex for each region and one unit-length edge for each plane, gives
an $\R$-tree (which comes from an underlying $\Z$-tree) on which
$\pi_1M$ acts by isometries.  In this tree, the distance between two
vertices is the minimum number of intersections between $\Tilde{F}$
and a path between the corresponding complementary regions in
$\Tilde{M}$.

Similarly, the boundary curves $\partial F$ lift to a collection of
lines separating $\Tilde{S}$ and give a dual tree $T_{\partial F}$ on
which $\pi_1S$ acts by isometries.  The equivalence between
laminations and foliations allows us to identify $\partial F$ with
measured foliation class in $\MF(S)$; under this correspondence,
$T_{\partial F}$ becomes the dual tree of that measured foliation (in
the sense of Section \ref{sec:dual-trees}).

Since the boundary is incompressible, the inclusion $S \into M$ lifts
to $\Tilde{S} \to \Tilde{M}$ which induces a map $T_{\partial F} \to
T_F$.  This map of trees is an isometric embedding: It is weakly
contracting, since minimizing the number of intersections of a path in
$\Tilde{S}$ with $\Tilde{\partial F}$ is a more constrained problem
than minimizing intersections of a path in $\Tilde{M}$ with
$\Tilde{F}$.  However, if an isotopy of such a path in $\Tilde{M}$
were to decrease the number of intersections with $\Tilde{F}$ (i.e.~if
the map $T_{\partial F} \to T_F$ strictly contracted any distance),
then putting the isotopy in general position relative to $\Tilde{F}$
would reveal a boundary compression of $F$.  Since $F$ is
$\partial$-incompressible, this is a contradiction.

Applying Theorem \ref{thm:embedding-isotropic} to $T_{\partial F} \to
T_F$ we conclude that $\partial F \in \L_M$.  Since $\L_M$ is an
isotropic piecewise linear cone, the desired conclusion follows.
\end{proof}

Comparing Floyd and Hatcher's proofs with that of Theorem
\ref{thm:embedding-isotropic} shows that the same ``cancellation''
phenomenon is at work in both cases.  Briefly, the connection is as
follows.  Floyd and Hatcher analyze weight functions on \emph{branched
  surfaces} that carry all of the incompressible,
$\partial$-incompressible surfaces in $M$.  Weights on a branched
surface satisfy a linear condition at each singular vertex.  When the
Thurston form is applied to a pair of weights on the boundary train
track of a branched surface, these vertex conditions lead to pairwise
cancellation of terms in the Thurston form, giving an isotropic space
of boundary weights.

The finite set of branched surfaces that are used in this argument
come from a construction of Floyd-Oertel \cite{floyd-oertel} which is
based on normal surface theory and a triangulation of the
$3$-manifold.  In this way, the weight conditions at the singular
vertices of a branched surface are dual to the weak $4$-point
condition \eqref{eqn:four-point} in each $3$-simplex that defines the
cone $\L_M$ in our approach, and the role of the spaces
$W_4(\Delta_M)$ in the proof of Theorem \ref{thm:embedding-isotropic}
is analogous to that of the space of boundary weights of a branched
surface in the arguments of Floyd and Hatcher.

\section{The K\"ahler structure of $Q(X)$}
\label{sec:kahler}

The goal of this section is to introduce a K\"ahler metric on $Q(X)$
and then to show that the foliation map $\F : Q(X) \to \MF(S)$
identifies the underlying symplectic space with the Thurston
symplectic structure on $\MF(S)$.  The K\"ahler metric we construct
has singularities but we show that it is smooth relative to a
stratification of $Q(X)$.

\subsection{The stratification}
  \label{sec:stratification}

Let $Z$ be a manifold.  A \emph{stratification} of $Z$ is a locally
finite collection of locally closed submanifolds $\{ Z_i \: | \: i \in
I \}$ of $Z$, the \emph{strata}, indexed by a set $I$ such that
\begin{enumerate}
\item $Z = \bigcup_{i \in I} Z_i$
\item $Z_i \cap \bar{Z_j} \neq 0$ if and only if $Z_i \subset
\bar{Z_j}$
\end{enumerate}
These conditions induce a partial order on $I$ where $i \leq j$ if $Z_i
\subset \bar{Z_j}$.  A stratification of a complex manifold $Z$ is a
\emph{complex analytic stratification} if the closure and boundary of
each stratum (i.e.~$\bar{Z_i}$ and $\bar{Z_i} \setminus Z_i$)
are complex analytic sets.

Let $\Q(S)$ denote the space of holomorphic quadratic differentials on
marked Riemann surfaces diffeomorphic to $S$, i.e.~the set of all
pairs $(X,\phi)$ where $X \in \T(S)$ and $\phi \in Q(X)$.  This space
is a vector bundle over $\T(S)$ isomorphic to the cotangent bundle
$T^*\T(S)$.  Let $s_0 : \T(S) \to \Q(S)$ denote the zero section.

There is a natural complex analytic stratification of $\Q(S)$
according to the numbers and types of zeros of the quadratic
differential (see \cite{veech:quadratic-differentials}
\cite{masur-smillie}).  Specifically, let the \emph{symbol} of a
nonzero quadratic differential $\phi$ be the pair $(\n, \epsilon)$
where $\n = (n_1, \ldots, n_k)$ is the list of multiplicities (in
weakly decreasing order) of the zeros of $\phi$, and where $\epsilon =
\pm1$ according to whether $\phi$ is the square of a holomorphic
$1$-form ($\epsilon=1$) or not ($\epsilon=-1$).  Thus we have $\sum_i
n_i = 4g-4$ and there are finitely many possible symbols; we denote
the set of all such symbols by $\symb$.

Given $\pi \in \symb$ let $\Q(S,\pi)$ denote the set of quadratic
differentials with symbol $\pi$.  This set is a manifold, with local
charts described below (in Section \ref{sec:period-coordinates}).  The
stratification of $\Q(S)$ is formed by the sets $\Q(S,\pi)$ and the
zero section $s_0(\T(S))$.

There is a related stratification of a fiber $Q(X)$ with the following
properties:

\begin{lem}
\label{lem:smooth-stratification}
For each $X \in \T(S)$, there exists a complex analytic stratification
$\{ Q_i(X) \}$ of $Q(X)$ such that:
\begin{rmenumerate}
\item Each stratum is a connected and $\C^*$-invariant.
\item The symbol is constant on each stratum $Q_i(X)$, and 
\item If $q \in Q_i(X)$ and $v \in T_{q}Q_i(X)$, then the
meromorphic function $v/q$ has at most simple poles.
\end{rmenumerate}
\end{lem}

\begin{proof}
A complex analytic stratification can always be refined so that a
given complex analytic subset becomes a union of strata (see
\cite{whitney} \cite[Thm.~1.6, p.~43]{goresky-macpherson}), and a further refinement
can be taken so that the strata are connected.  Here \emph{refinement}
refers to changing the stratification in such a way that each new
stratum is entirely contained in one of the old strata.

Applying this to the stratification of $\Q(S)$ discussed above and the
closed subvariety $Q(X)$ we obtain a stratification of $\Q(S)$ such
that the symbol is constant on each stratum and so that $Q(X)$ is a
union of strata.  In particular there is an induced stratification $\{
Q_i(X) \}$ of $Q(X)$ satisfying (ii).  The original
stratification of $\Q(S)$ is $\C^*$-invariant, and the strata of the
refinement can be constructed using finitely many operations that
preserve this invariance (i.e.~boolean operations and passage from a
complex analytic set to its singular locus or to an irreducible
component), so property (i) also follows.

Thus the proof is completed by the
following lemma, which shows that property (iii) is a consequence of
property (ii).
\end{proof}

\begin{lem}
Let $M \subset Q(X)$ be a submanifold on which the symbol is constant.
Then for any $(q,\dot{q}) \in TM$, the function $\dot{q}/q$ has at
most simple poles on $X$.
\end{lem} 

\begin{proof}
Let $q_t$ be a smooth family of quadratic differentials in $M$ with
$q_0 = q$ and with tangent vector $\dot{q}$ at $t=0$.

Let $p \in X$ be a zero of $q$ of order $k > 0$, and choose a local
coordinate $z$ in which $z(p) = 0$ and $q = z^k dz^2$.  Since $q_t$
has the same symbol as $q$ for small $t$, in a neighborhood of $p$ we can write
$$ q_t = \alpha_t^*(z^k dz^2)$$
where $\alpha_t$ is a smooth family of holomorphic functions defined on
$\{ |z| < \epsilon \}$ and $\alpha_0(z) = z$.  This is equivalent to
the statement that the family of polynomial differentials $(z^k + a_{k-2} z^{k-2} +
\ldots + a_0)dz^2$ is a universal deformation of $z^k dz^2$ (see
\cite[Prop.~3.1]{hubbard-masur}).
Since $\alpha_t^*(z^k dz^2) =
\alpha_t(z)^k \left ( \alpha_t'(z)  \right )^2 dz^2$, a
calculation gives
$$ \dot{q} = z^{k-1} \left ( k \dot{\alpha} + 2 z \dot{\alpha}' \right ) \: dz^2$$
and $\dot{q}$ has a zero of order at least $k-1$ at $p$.  It follows
that $\dot{q}/q$ has at most simple poles.
\end{proof}

\subsection{The K\"ahler form}

The vector space $Q(X)$ is a complex manifold with a global
parallelization which identifies $T_\phi Q(X) \simeq Q(X)$ for any $\phi
\in Q(X)$.  We consider the hermitian pairing $\langle \param , \param
\rangle_\phi$ on $T_\phi Q(X)$ defined by
\begin{equation}
\label{eqn:pairing}
\langle \psi_1, \psi_2 \rangle_\phi := \int_X \frac{\psi_1
  \bar{\psi}_2}{4 |\phi|}.
\end{equation}
Note that in this expression we consider
$\psi_1 \bar{\psi_2}/|\phi|$ as a complex-valued quadratic form on
$TX$, and we integrate the corresponding complexified volume form.
With respect to a local choice of a holomorphic $1$-form
$\sqrt{\phi}$, the integrand can also be written as a
$$\frac{i}{2} \left ( \frac{\psi_1}{2 \sqrt{\phi}} \wedge
\frac{\bar{\psi_2}}{2 \bar{\sqrt{\phi}}} \right ) .$$ A branched double
covering of $X$ can be used to globalize this interpretation (as in
the proof of Theorem \ref{thm:symplectomorphism} below).

Let $g_\phi$ and $\omega_\phi$ the real and imaginary parts of this
hermitian pairing, i.e.
$$ \langle \psi_1, \psi_2 \rangle_\phi  = g_\phi(\psi_1,\psi_2) + i \:\omega_\phi(\psi_1,\psi_2).$$
Similarly, we write $\|\psi\|_\phi^2 = g_\phi(\psi,\psi) = \langle
\psi, \psi \rangle_\phi$.

The pairing is not defined for all vectors
because the function $\psi_1 \bar{\psi}_2 / |\phi|$ is not necessarily
integrable on $X$.  However, it is defined on the strata $Q_i(X)$:

\begin{thm}
\label{thm:kahler}
For each stratum $Q_i(X) \subset Q(X)$ we have:
\begin{rmenumerate}
\item The pairing $\langle \psi_1, \psi_2 \rangle_\phi$ is well-defined and
positive-definite on the tangent bundle $T Q_i(X)$,
\item The alternating form $\omega_\phi$ on $T Q_i(X)$ can be expressed as
$$ \omega_\phi = \frac{i}{2} \partial \bar{\partial} N,$$
where $N : Q(X) \to \R$ is defined by $N(\phi) = \|\phi\|$.  In
particular $\omega_\phi$ is closed, and thus
\item The hermitian form $\langle \param
, \param \rangle_\phi$ defines a K\"ahler structure on $Q_i(X)$.
\end{rmenumerate}
\end{thm}

A formula for the second derivatives of $N$ equivalent to (ii)
above was derived by Royden (see \cite[Lem.~1]{royden}) in the case of the
open stratum consisting of differentials with at most simple zeros.
Royden also analyzes the failure of $N$ to be twice differentiable in
certain directions transverse to the other strata, however the fact
that $N$ is $C^2$ when restricted to a stratum and the analogous
derivative formula follow easily by similar methods.  We describe the
necessary adaptation of his argument below.

\begin{proof}
The function $|z|^{-1}$ is integrable in a neighborhood of $0$ in $\C$, so if
$\psi_1/\phi$ has at most simple poles, then $\langle \psi_1, \psi_2
\rangle_\phi$ is finite for all $\psi_2 \in Q(X)$.  By part (iii) of
Lemma \ref{lem:smooth-stratification}, this holds for all $(\phi,\psi_1) \in
TQ_i(X)$, so the pairing is well-defined there.  For any $\psi \neq
0$, the function $\frac{|\psi|^2}{|\phi|}$ is positive except for
finitely many zeros, and thus
$$ \|\psi\|_\phi^2 = \int_X \frac{|\psi|^2}{4 |\phi|} > 0.$$

Now we consider the existence of derivatives of the function $N(\phi)
= \int_X |\phi|$.  The lack of smoothness of $|\phi|$ at the zeros of
$\phi$ is the only problem: If $K \subset X$ is compact and contains
no zeros of $\phi$, then $\phi \mapsto \int_{K} |\phi|$ is $C^\infty$
on a neighborhood of $\phi$ in $Q(X)$, and its derivatives are
obtained by differentiating inside the integral.  Thus our strategy
will be to determine the resulting formula for $\partial
\bar{\partial}N$ away from the zeros and then show that the zeros
contribute nothing.

Fix $\phi \in Q_i(X)$ and for $\epsilon > 0$ let $X_\epsilon$ denote
an open neighborhood of the zero set of $\phi$ such that each
connected component of $X_\epsilon$ is a disk containing a single zero
of $\phi$, and where each such disk admits a local holomorphic
coordinate $z$ in which the restriction of $\phi$ is identified with
$z^k dz^2$ on the open disk $\{ |z| < \epsilon \} \subset \C$.  Here
$k \in \N$ depends on the component (and is equal the multiplicity of the
zero of $\phi$ it contains).  Such standard disk neighborhoods exist
for all $\epsilon$ sufficiently small.

Then we can write $N(\phi) = N^\epsilon_0(\phi) + N^\epsilon_1(\phi)$
where:
\begin{equation*}
\begin{split}
N^\epsilon_0(\phi) &= \int_{X_\epsilon} |\phi|\\
N^\epsilon_1(\phi) &= \int_{X \setminus X_\epsilon} |\phi|
\end{split}
\end{equation*}
As explained above, function $N^\epsilon_1(\phi)$ is smooth on a
neighbhorhood of $\phi$ in $Q(X)$, so it restricts to a smooth
function on $Q_i(X)$.  From now on we consider $N^\epsilon_0$ and
$N^\epsilon_1$ as functions on $Q_i(X)$.  We claim:
\begin{itemize}
\item[(I)]  $N^{\epsilon}_0$ is smooth on a neighborhood of
$\phi$ in $Q_i(X)$,
\end{itemize}
and that at the point $\phi$ we have
\begin{itemize}
\item[(II)] $\partial \bar{\partial} N^{\epsilon}_1 \to -2i \omega_\phi$
as $\epsilon \to 0$, and 
\item[(II)] $\partial \bar{\partial} N^{\epsilon}_0 \to 0$
as $\epsilon \to 0$.
\end{itemize}
Since (I) allows the expression $\partial \bar{\partial}N = \partial \bar{\partial}
N^\epsilon_1 + \partial \bar{\partial} N^\epsilon_0$, taking
$\epsilon \to 0$ and using (II) and (III) gives the desired formula
$\frac{i}{2} \partial \bar{\partial} N =  \omega_\phi$.  Thus the Theorem is
reduced to these claims.

We begin with (II), which amounts to differentiating inside the
integral defining $N^\epsilon_1$. Let $a : U \to Q_i(X)$ be a local
holomorphic parameterization where $U \subset \C^m$ is a neighborhood
of the origin and $a(0) = \phi$. Consider the pointwise norm
$|a(\zeta)|$ as a function of $\zeta$.  Differentiating
$|a| = \sqrt{a \bar{a}}$ we find
$$\frac{\partial^2 |a|}{\partial \zeta_k \partial \bar{\zeta_l}}(0) = 
\frac{a_k \bar{a_l}}{4 |\phi|}$$
where for brevity we have written $a_k$ for
$\frac{\partial{a}}{\partial \zeta_k}(0)$, and this formula is valid at any
point where $\phi \neq 0$.  Since the vectors $a_k$ form a basis for $T_\phi
Q_i(X)$, this equivalent to the statement that for all $\psi_1, \psi_2
\in T_\phi Q_i(X)$ we have $\partial
\bar{\partial} n(\psi_1,\psi_2) = (-2i) \frac{1}{4|\phi(z)|} \Im(\psi_1(z)
\bar{\psi_2}(z))$ where $n(\psi) = |\psi(z)|$ and $z$ is any point
with $\phi(z) \neq 0$.  Integrating over the set $X \setminus X_\epsilon$,
which is compact and which does not contain any zeros of $\phi$, we
have the corresponding expression
$$ \partial \bar{\partial} N^\epsilon_1(\psi_1,\psi_2) =
 (-2i) \int_{X \setminus
  X_\epsilon} \frac{\Im(\psi_1 \bar{\psi_2})}{4|\phi|}.$$
Since $\psi_1,\psi_2 \in T_\phi Q_i$, the form
$\frac{1}{4 |\phi|} \Im(\psi_1 \bar{\psi_2})$ is integrable on $X$, and
by definition its integral over $X$ is $\omega_\phi(\psi_1,\psi_2)$.  Since the
measure of $X_\epsilon$ goes to zero as $\epsilon \to 0$, claim (II) follows.

Now we consider what happens near the zeros of $\phi$.  Here it will
be essential that we are working in a stratum, so the symbol is
constant.  This means that the zeros move holomorphically as a function of the point
in $Q_i(X)$ and with constant multiplicity.  In the same coordinate system where $\phi$ restricted to a
component of $X_\epsilon$ becomes $z^k dz^2$ on $\Delta_\epsilon := \{ |z| < \epsilon
\}$, each $\psi$ in a neighborhood $U$ of $\phi$ in 
$Q_i(X)$ can therefore be expressed as
$$ e^{h(\psi,z)} (z - u(\psi))^k dz^2 $$
where $u : U \to \Delta_\epsilon$ and $h(\psi,z) : U \times
\Delta_\epsilon \to \C$ are holomorphic functions and 
$u(\phi) = 0$.  Thus $N^\epsilon_0(\psi)$ is the sum of finitely many terms of the form
$$ \int_{|z| < \epsilon} |e^{h(\psi,z)} (z - u(\psi))^k|
|dz|^2 = \int_{|z| < \epsilon} e^{\Re h(\psi,z)} |z - u(\psi)|^k |dz|^2. $$
Using the change of variable $w = z - u(\psi)$ this integral becomes
\begin{equation}
\label{eqn:good-variation-form}
\int_{|w + u(\psi)| < \epsilon} e^{\Re h(\psi,w + u(\psi))} |w|^k |dw|^2.
\end{equation}
In this form, the integrand is $C^\infty$ as a function of $\psi$ and all of its
$\psi$-derivatives are continuous in $w$.  Furthermore, on the
boundary curve $|w + u(\psi)| = \epsilon$ the integrand is smooth in
both $\psi$ and $w$ (since for $\psi$ near $\phi$ this curve avoids
the locus $w=0$ where the integrand may fail to be smooth).

These conditions are exactly what we need to differentiate under the
integral in computing derivatives of \eqref{eqn:good-variation-form},
and since $N_0^\epsilon$ is a sum of finitely many terms of this form,
we find that it is smooth.  This establishes (II).

Finally, we must estimate the derivative of $N_0^\epsilon$ using
\eqref{eqn:good-variation-form}.  The second $\psi$-derivatives of the
integral at $\psi = \phi$ split into an interior term (an integral over
$|w| < \epsilon$) and boundary terms (integrals over $|w| =
\epsilon$).  The boundary terms involve up to second derivatives of
the boundary curve as a function of $\psi$ and the first partial
derivatives of the integrand with respect to both $\psi$ and $w$.  The
interior term involves the second $\psi$-partial derivatives of the
integrand.  This gives an overall estimate for any second partial
derivative of the function $N_0^\epsilon$ at $\phi$ of the form
$$ O(A \, B_{2,0}^{\mathrm{int}} + L \, B_{1,1}^{\mathrm{bdy}} \, C_2)$$
where
\begin{itemize}
\item $A = \pi \epsilon^2$ is the area of the region of integration at
$\psi = \phi$,
\item $B_{2,0}^{\mathrm{int}}$ is an interior upper bound (i.e.~on
$|w| < \epsilon$) for the integrand and its
$\psi$-partial derivatives of order at most $2$
\item $B_{1,1}^{\mathrm{bdy}}$ is a boundary upper bound (i.e.~on $|w| =
\epsilon$) for the integrand and its first partial derivatives with
respect to $\psi$ and $w$, 
\item $L = 2 \pi \epsilon$ is the length of the boundary curve, and
\item $C_2$ is a bound on the $\psi$-derivatives of the boundary.
\end{itemize}
Since $h$ and $u$ are holomorphic, their derivatives of any fixed order
are uniformly bounded once we take $U$ and
$\epsilon$ small enough.  Examining the integrand of
\eqref{eqn:good-variation-form} we then find $B_{2,0}^{\mathrm{int}} =
O(\epsilon^k)$, $B_{1,1}^{\mathrm{bdy}} = O(\epsilon^{k-1})$ (with the
$w$-derivative of the integrand on $|w|=\epsilon$ being the dominant
term), and $C_2 = O(1)$.  Hence $\partial \bar{\partial} N_0^\epsilon$
is $O(\epsilon^2\! \cdot \! \epsilon^k + \epsilon \!\cdot\! \epsilon^{k-1}
\!\cdot\! 1) = O(\epsilon^k)$ as $\epsilon \to 0$. Since $k > 0$ this
establishes (III).
\end{proof}

Before proceeding to relate the K\"ahler structure of Theorem
\ref{thm:kahler} to the symplectic structure of $\MF(S)$ we will need
to describe convenient local coordinates for both spaces.  After
discussing suitable period and train track coordinates, we return to
the matter of relating these spaces in Section \ref{sec:symplectomorphism}.

\subsection{Double covers and periods}

For any $\phi \in Q(X)$ let $\Tilde{X}_\phi$ denote the Riemann
surface of the locally-defined one-form $\sqrt{\phi}$ on $X$,
i.e.~$\Tilde{X}_\phi$ is a branched double cover of $X$ in which
$\phi$ is canonically expressed as the square of a $1$-form (also
denoted by $\sqrt{\phi}$). 
By construction $\Tilde{X}_\phi$ has a holomorphic involution $\sigma$
such that $\sigma^*\sqrt{\phi} = -\sqrt{\phi}$ and $X = \Tilde{X}_\phi /
\sigma$.

The one-form $\sqrt{\phi}$ on $\Tilde{X}_\phi$ has \emph{absolute
  periods} obtained by integration along cycles in
$H_1(\Tilde{X}_\phi)$ and \emph{relative periods} obtained by
integration along cycles in $H_1(\Tilde{X}_\phi,\Tilde{Z}_\phi)$
where $\Tilde{Z}_\phi$ is the set of zeros of $\sqrt{\phi}$.
Since $\sigma^*\sqrt{\phi} = -\sqrt{\phi}$, these integrals vanish for
cycles invariant under $\sigma$ and nontrivial periods are only
obtained from cycles in the $-1$-eigenspace, which we denote by
$$ H_1^{\odd}(\phi) := \{ c \in H_1(\Tilde{X}_\phi, \Tilde{Z}_\phi; \C)
\: | \: \sigma_*c = -c \}.$$
Collectively, the periods of $\sqrt{\phi}$ determine its cohomology
class as an element of
$$ H^1_{\odd}(\phi) := \{ \theta \in H^1(\Tilde{X}_\phi ,\Tilde{Z}_\phi;
\C) \: | \: \sigma^*\theta = -\theta \}.$$

Note that a saddle connection $I$ of $\phi$ determines an element $[I]
\in H_1^{\odd}(\phi)$ by taking the difference of its two lifts to
$\Tilde{X}_\phi$.  The result is well-defined up to sign.  We can
therefore consider relative periods of $\sqrt{\phi}$ along such saddle
connections.

Since integration of $\sqrt{\phi}$ gives a local natural coordinate
for $\phi$, the relative period of a saddle connection is simply its
displacement vector in such a coordinate system.  In particular, the
height of a saddle connection $I$ is given by
\begin{equation}
\label{eqn:saddle-height}
\phi\text{-height}(I) = \left | \Im \int_{[I]} \sqrt{\phi}  \right |
\end{equation}

\subsection{Period coordinates for strata}
  \label{sec:period-coordinates}

Let $\Q(S,\pi)$ be a stratum in $\Q(S)$.  The topological type of the
double cover $\Tilde{X}_\phi \to X$ is determined by the symbol of
$\phi$, so in a small neighborhood $U$ of $\phi$ in $\Q(S,\pi)$ we can
trivialize the family of double covers and (co)homology groups.
Thus we can regard each space $H^1_{\odd}(\psi)$, where $\psi \in U$, as an
instance of a single cohomology space $H^1_{\odd}(\pi)$ that is
determined by topological information contained in $(X,\phi)$; we let
$H_1^{\odd}(\pi)$ denote the corresponding trivialization of the
family of homology groups.  When considering a class $a \in
H_1^{\odd}(\pi)$ we write $a_{\psi}$ for a representing cycle in
$H_1^{\odd}(\psi)$.

Using this local trivialization, the cohomology class of $\sqrt{\phi}$
determines the \emph{relative period map}
$$ \per : U \to H^1_{\odd}(\pi).$$
Explicitly, as a linear function on cycles $a \in H_1^{\odd}(\pi)$
the map is given by
$$ \per(\phi)(a) = \int_{a_\phi} \sqrt{\phi}.$$
This map provides local coordinates for strata
\cite{veech:quadratic-differentials}
\cite[Sec.~28]{veech:teichmuller-geodesic-flow} \cite{masur-smillie}:

\begin{thm}
The relative period construction gives local biholomorphic coordinates for
$\Q(S,\pi)$, i.e.~for any sufficiently small open set $U \subset
\Q(S,\pi)$ the period map $\per : U \to H^1_{\odd}(\pi)$ is a
diffeomorphism onto an open set. In particular, we have $\dim_\C
H^1_-(\pi) = \dim_\C \Q(S,\pi)$.\noproof
\end{thm}

While the result above applies to strata in $\Q(S)$, each stratum
$Q_i(X)$ of $Q(X)$ is a complex submanifold of $\Q(S,\pi)$ for some
$\pi \in \symb$, so we have:

\begin{cor}
\label{cor:relative-period}
Let $\phi \subset Q_i(X)$ be a quadratic differential with symbol
$\pi$.  Then there is an open neighborhood of $\phi$ in $Q_i(X)$ in
which the relative period map to $H^1_-(\pi)$ is biholomorphic onto
its image.\noproof
\end{cor}

Later we will need  the following formula for the derivative of the relative period
coordinates:

\begin{lem}[Douady-Hubbard]
\label{lem:period-derivative}
Let $\phi \in Q_i(X)$ and $\psi \in T_\phi Q_i(X)$.  Then for any
$a \in H_1^{\odd}(\pi)$ we have
$$d\per_\phi(\psi)(a) = \int_{a_\phi} \frac{\psi}{2\sqrt{\phi}}.$$
\end{lem}\hfill \qedsymbol

The proof in \cite{douady-hubbard} is for differentials with simple
zeros, however, the argument only uses the fact that the period of a
saddle connection for a family $\phi_t$ (where $\phi_0 = \phi$ and
$\ddtzero\phi= \psi$) can be expressed in the form
$$ \int_{a_t}^{b_t} \sqrt{\phi_t(z)} dz $$
where $a_t$ and $b_t$ are smooth paths traced out by the zeros of
$\phi_t$ as $t$ varies near $0$.  The assumptions on $\phi, \psi$ in
the Lemma above imply that $\psi$ is tangent to a family of
differentials whose zeros have constant multiplicity, and so the same
argument applies.

\subsection{Adapted train tracks}

We now consider train track coordinates for $\MF(S)$ compatible with
the relative period construction described above.  The following
refinement of Lemma \ref{lem:delaunay-track} ensures that we can
always choose these coordinates so that the foliation in question lies
in the interior of the train track chart.

\begin{lem}
\label{lem:adapted-train-track}
For each nonzero $\phi \in Q(X)$ there exists a triangulation $\Delta$
of $X$ by saddle connections and a dual maximal train track $\tau$
satisfying conditions (i)-(iii) of Lemma \ref{lem:delaunay-track} and
such that none of the saddle connections in $\Delta$ are horizontal.
In particular, the point $[\F(\phi)]$ lies in the interior of the
train track chart $\MF(\tau)$.
\end{lem}

\begin{proof}
Each face of the train track chart $\MF(\tau)$ is defined by the
weight of some branch of the track being zero.  In this case the
weights are heights of saddle connections, so excluding horizontal
edges will result in $\F(\phi)$ being in the interior of the chart.

If the Delaunay triangulation of Lemma \ref{lem:delaunay-track} has no
horizontal edges, or if $\phi$ itself has no horizontal saddle
connections, then we are done.  Otherwise we must alter the
construction of the triangulation to eliminate the horizontal edges.
Note that $\phi$ has only finitely many horizontal saddle connections.

Consider the Teichm\"uller geodesic $(X_t,\phi_t)$ determined by $e^{i
  \theta} \phi$.  The Riemann surfaces and quadratic differentials in
this family are identified by locally affine maps, so saddle
connections of $\phi$ are also saddle connections of $\phi_t$ and
vice versa.  Furthermore, if $\theta \neq \pi$ then horizontal saddle
connections of $\phi$ have $\phi_t$-length growing exponentially in
$t$.  If we choose $\theta \neq \pi$ so that the geodesic is recurrent
in moduli space (a dense set of directions have this property
\cite{kleinbock-weiss}), then by choosing $t$ large enough we can
assume that $X_t$ has bounded $|\phi_t|$-diameter while the
$\phi$-horizontal saddle connections are arbitrarily long with respect
to $|\phi_t|$.

Since the length of an edge of the Delaunay triangulation is bounded
by the diameter of the surface \cite[Thm.~4.4]{masur-smillie}, this
shows that for large $t$ the $\phi$-horizontal saddle connections are
not edges of the Delaunay triangulation of $\phi_t$.  Thus the
Delaunay triangulation of $\phi_t$ gives the desired
triangulation by non-horizontal saddle connections of $\phi$.
\end{proof}

\subsection{The symplectomorphism}
\label{sec:symplectomorphism}

Hubbard and Masur showed that for any $X \in \T(S)$, the foliation map
$\F : Q(X) \to \MF(S)$ is a homeomorphism \cite{hubbard-masur}.  We
now show that this map relates the K\"ahler structure on $Q(X)$
introduced above to the Thurston symplectic structure of $\MF(S)$:

\begin{thm}
\label{thm:symplectomorphism}
For any $X \in \T(S)$, the map $\F : Q(X) \to \MF(S)$ is a
real-analytic stratified symplectomorphism.  That is:
\begin{rmenumerate}
\item For any $\phi \in Q_i(X)$ there exists an open neighborhood $U
\subset Q_i(X)$ of $\phi$ in its stratum and a train track
coordinate chart $\MF(\tau) \subset \MF(S)$ covering $\F(U)$ so that
the restriction
$$ \F : U \to \MF(\tau) $$ 
is a real-analytic diffeomorphism onto its image, and
\item The derivative $d\F_\phi$ defines a symplectic linear map from
$T_\phi Q_i(X)$ into $W(\tau) \simeq T_{\F(\phi)} \MF(\tau)$, where
$T_\phi Q_i(X)$ is equipped with the symplectic form $\omega_\phi$
and $W(\tau)$ is given the Thurston symplectic form.
\end{rmenumerate}
\end{thm}

\begin{prooflist}
\item Let $\phi \in Q_i(X)$.  Applying Lemma
\ref{lem:adapted-train-track} we obtain a neighborhood $U \subset
Q_i(X)$ of $\phi$ and a train track $\tau$ that carries the
horizontal foliation of each $\psi \in U$ by assigning to each branch
the height of an associated edge of the $\psi$-geodesic triangulation.

Lift $\tau$ and the dual $\phi$-geodesic triangulation $\Delta$ to the
cover $\Tilde{X}_\phi$, obtaining a triangulation $\hat{\Delta}$ and double
covering of train tracks $\hat{\tau} \to \tau$.  Orient the edges of
$\hat{\Delta}$ so that the integral of $\Im \sqrt{\phi}$ over any edge is
positive.  (This integral is nonzero because the original
triangulation did not have any $\phi$-horizontal edges.)  Then the
integral of $\Im \sqrt{\phi}$ over an edge $\hat{e}$ is the
$\phi$-height of the corresponding edge $e$ of $\Delta$.

The covering train tracks and oriented triangulations obtained in this
way for other $\psi \in U$ are naturally isotopic to $\hat{\tau}$ and
$\hat{\Delta}$, so this construction extends throughout $U$.  Thus for
all $\psi \in U$ we have realized the weights on $\tau$ defining
$[\F(\psi)]$ as the imaginary parts of periods of $\sqrt{\psi}$, which
by Corollary \ref{cor:relative-period} are real-analytic functions.

It remains to show that the derivative of the map to $\ML(\tau)$ is an
isomorphism, so that after shrinking $U$ appropriately we have a
diffeomorphism onto an open set.  However this is a consequence of the
proof of (ii) below since the Thurston symplectic form is
nondegenerate.

\item We need to show that any $\psi_1,\psi_2 \in T_\phi
Q_i(X)$ satisfy
\begin{equation}
\label{eqn:symplectic-goal}
\omega_\phi(\psi_1,\psi_2) = \omega_{\th} \left ( d \F_\phi(\psi_1),
d\F_\phi(\psi_2) \right ).
\end{equation}
We begin by analyzing the left hand side.  Let $\hat{\psi_i}$ denote the
lift of $\psi_i$ to the double cover $\Tilde{X}_\phi$ and define
$$ \theta_i = \frac{\hat{\psi_i}}{2 \sqrt{\phi}} \in \Omega(\Tilde{X}_\phi).$$
These $1$-forms are holomorphic because all poles of $\psi_i / \phi$
are simple and occur at branch points of the covering $\Tilde{X}_\phi
\to X$.  Since the $2$-form $\frac{i}{2} \theta_1 \wedge \bar{\theta}_2$ is
the lift of the integrand of $\langle \psi_1, \psi_2
\rangle_{\phi}$ to the degree-$2$ cover $\Tilde{X}_\phi$, we have
$$ \omega_\phi(\psi_1,\psi_2) = \frac{1}{2} \Im \int_{\Tilde{X}_\phi}
\frac{i}{2} \theta_1 \wedge \bar{\theta}_2 = \frac{1}{4} \Re
\int_{\Tilde{X}_\phi} \theta_1 \wedge \bar{\theta}_2.$$
For any holomorphic $1$-forms $\theta_i$ we have
$$ \Re(\theta_1 \wedge \bar{\theta}_2) = 2 (\Re \theta_1) \wedge (\Re
\theta_2) = 2 (\Im \theta_1) \wedge (\Im \theta_2)$$
so we can express the integral above as
\begin{equation}
\label{eqn:omega-phi-reduction}
\omega_\phi(\psi_1,\psi_2) = \frac{1}{2} \int_{\Tilde{X}_\phi} (\Im
\theta_1) \wedge (\Im \theta_2) = \frac{1}{2} [ \Im \theta_1] \cdot
[\Im \theta_2],
\end{equation}
where in the last expression $[\alpha]$ denotes the de Rham
cohomology class of a closed $1$-form $\alpha$ and $[\alpha] \cdot
[\beta]$ is the cup product.

Now consider the pairing $\omega_{\th}\left ( d \F_\phi(\psi_1),
d\F_\phi(\psi_2) \right )$.  The tangent vector $d \F_\phi(\psi_i) \in
W(\tau)$ is a weight function whose value on a branch $e$ is the
derivative of the height of the associated edge $e'$ of $\Delta$.  The
height of an edge is the imaginary part of the period of
$\sqrt{\phi}$, so Lemma \ref{lem:period-derivative} gives a formula
for the derivatives of these periods.  Namely, after lifting to the
covering train track $\hat{\tau}$ we find that $d \F_\phi(\psi_i)$
corresponds to the weight function $\hat{w}_i \in W(\hat{\tau})$
defined by
$$ \hat{w}_i(e) = \int_{\hat{e}} \Im \theta_i,$$
where $e$ is a branch of $\tau$ (identified with its dual edge of
$\Delta$) and $\hat{e}$ is an associated oriented edge of
$\hat{\Delta}$.  The orientation of $\hat{\Delta}$ induces a
consistent orientation of $\hat{\tau}$ so that all intersections of
$\hat{\tau}$ with $\hat{\Delta}$ become positively oriented.  In terms
of this orientation, the expression above shows that the de Rham
cohomology class $[\Im \theta_i]$ is Poincar\'e dual to the cycle
$$ \hat{c}_i = \sum_{e \in \hat{\tau}} w_i(e) \vec{e}.$$
Using the formula \eqref{eqn:thurston-intersection} for the Thurston
form as a homological intersection of such cycles and the duality of
intersection and cup product, we have
$$  \omega_{\th}( d \F_\phi(\psi_1), d \F_\phi(\psi_2)) = \frac{1}{2} \left
( \hat{c}_1 \cdot \hat{c}_2 \right ) = \frac{1}{2} [\Im \theta_1]
\cdot [\Im \theta_2].$$ With \eqref{eqn:omega-phi-reduction} this
gives the desired equality between symplectic pairings.
\end{prooflist}

\begin{remark}
The smoothness of the foliation map when restricted to a set of
quadratic differentials with constant symbol is implicit in
\cite{hubbard-masur}.  Because we consider only tangent vectors to
strata in $Q(X)$, the subtle issues that arise from breaking up
high-order zeros (and which underlie the failure of differentiability
for the full map $Q(X) \to \ML(S)$) do not arise here.
\end{remark}

\subsection{Application: The Hubbard-Masur constant}
\label{sec:hubbard-masur-constant}

Here we mention an application of Theorem \ref{thm:symplectomorphism}
that is not used in the sequel.  It is immediate from the definition
\eqref{eqn:pairing} that the hermitian form $\langle \psi_1, \psi_2
\rangle_\phi$ is invariant under the action of $S^1 \simeq \{ e^{i
  \theta} \}$ on $Q(X)$ by scalar multiplication:
$$ \langle c \psi_1, c \psi_2\rangle_{c \phi} = \langle
\psi_1, \psi_2\rangle_{\phi} \; \text{ if }\; |c| = 1.$$
It follows that this $S^1$-action preserves the volume form
associated to the stratified K\"ahler structure on $Q(X)$, and thus the
symplectomorphism with $\MF(S)$ gives:
\begin{cor}
\label{cor:volume-preserving}
The action of $S^1$ on $\MF(S)$ induced by the foliation map $\F :
Q(X) \to \MF(S)$ preserves the volume form associated to the Thurston
symplectic structure. \noproof
\end{cor}

In particular this corollary applies to the \emph{antipodal involution} $i_X
: \MF(S) \to \MF(S)$ which corresponds to multiplication by $-1$ in
$Q(X)$.  This map exchanges the vertical and horizontal measured
foliations of any quadratic differential on $X$.

Let $b(X) \subset \MF(S)$ denote the unit ball of the extremal length
function on $X$:
$$ b(X) = \{ [\nu] \in \MF(S) \: | \: \Ext_{[\nu]}(X) \leq 1 \}. $$
Equivalently $b(X)$ is the image of the $L^1$ norm ball in $Q(X)$
under the foliation map, so it is invariant under $i_X$.

Let $\Lambda(X)$ denote the volume of this set with respect to the
Thurston symplectic form on $\MF(S)$; this defines the
\emph{Hubbard-Masur function} $\Lambda : \T(S) \to \R^+$.  This
function appears as a coefficient in various counting problems related
to the action of the mapping class group $\Mod(S)$ on $\T(S)$ studied
in \cite{ABEM}.

Using Corollary \ref{cor:volume-preserving}, Mirzakhani has shown
(personal communication):
\begin{thm}[Mirzakhani]
\label{thm:mirzakhani}
The Hubbard-Masur function is constant.  That is, the volume of $b(X)$
depends only on the topological type of $S$ and is independent of the
point $X \in \T(S)$.
\end{thm}

The following argument is based on the above-cited communication with
Mirzakhani.  An analogous statement in a different dynamical context
is established in \cite{yue:flip}.

\begin{proof}
We will use the antipodal map $i_X$ to show that the derivative of
$\Lambda$ vanishes identically.  Since $\T(S)$ is connected it will
then follow that $\Lambda$ is constant.

Let $S(X) = \bdy b(X)$ denote the extremal length unit sphere in
$\MF(S)$.  We think of this as a family of hypersurfaces in $\MF(S)$
parameterized by $X \in \T(S)$.

Fix a point $X_0 \in T(S)$.  For any other point $X \in \T(S)$,
both $S(X_0)$ and $S(X)$ intersect each ray in $\MF(S)$ in a single
point, so we can consider $S(X)$ as obtained from $S(X_0)$ by scaling
each point $[\nu] \in S(X_0)$ by a positive real number 
\begin{equation}
\label{eqn:scaling-function}
\left ( \Ext_{[\nu]}(X_0) / \Ext_{[\nu]}(X) \right )^{1/2}.
\end{equation}
Regarding this expression as a function of $[\nu]$, we have described
the spheres $S(X)$ as a family of ``radial graphs'' over $S(X_0)$.
Using this description, the derivative of this family at $X=X_0$ (if
it exists) is a vector field along $S(X_0)$ which is \emph{radial},
i.e.~it is a pointwise multiple of the vector field
$\frac{\partial\:}{\partial t}$ generating the $\R^+$ action.
Since the scaling function relating $S(X_0)$ to $S(X)$ is a quotient
of powers of the extremal length functions, differentiability of this
family at $X=X_0$ is a consequence of Gardiner's formula
\cite{gardiner:minimal-norm}, which states that the derivative of
extremal length is given by
$$ \left . \frac{d}{dt} \Ext_{[\nu]}(X_t) \right |_{t=0} = 2 \Re \int_{X_0}
\mu \, \F^{-1}([\nu]) $$
where $X_t$ is smooth a path in $\T(S)$ and $\mu$ is a Beltrami coefficient on
$X_0$ representing $\left . \frac{d\:}{dt} X_t \right |_{t=0}$.
Differentiating \eqref{eqn:scaling-function} using this formula, we
find that the derivative of $S(X)$ at $X = X_0$ is the continuous vector field
\begin{equation*}
V_\mu([\nu]) = - \left ( \Re \int_{X_0} \mu \, \F^{-1}([\nu]) \right ) \:
\frac{\partial \:}{\partial t}.
\end{equation*}
As above $\frac{\partial \:}{\partial t}$ is the vector field on
$\MF(S)$ generating the $\R^+$-action.  The derivative of the volume
enclosed by $S(X)$ at $X=X_0$ is therefore the integral over $S(X_0)$
of the interior product of this vector field with the Thurston volume form,
$$ d \Lambda_{X_0}(\mu) = \int_{S(X_0)} V_\mu \: \lrcorner \: \omega_\th^n$$
where $n = \frac{1}{2} \dim_\R \MF(S)$.
Since it corresponds to the $L^1$ norm sphere in $Q(X)$, the sphere
$S(X_0)$ is invariant under the antipodal involution.  By Corollary
\ref{cor:volume-preserving} the volume form $\omega_\th^n$ is also
$i_X$-invariant.  But since $\F^{-1}(i_X([\nu])) =
-\F^{-1}([\nu])$, the vector field $V_\mu$ is odd under this
involution (i.e.~$i_X^*(V_\mu) = - V_\mu$) as is the integrand $V_\mu
\: \lrcorner \: \omega_{Th}^n$.  Since the integral of an odd form
over $S(X_0)$ vanishes we have $d \Lambda_{X_0}(\mu) = 0$.
\end{proof}

\section{Character varieties and holonomy}
\label{sec:3-manifolds}

\subsection{Character varieties}
  \label{sec:character-varieties}

Let $G$ be one of the complex algebraic groups $\SL_2\C$ or $\PSL_2\C$
and let $\Gamma$ be a finitely generated group.  We denote by
$\sR(\Gamma,G) := \Hom(\Gamma,G)$ the $G$-representation variety of
$\Gamma$, which carries an action of $G$ by conjugation.  The
categorical quotient
$$\X(\Gamma,G) := \sR(\Gamma,G) \sslash G$$
is the \emph{character variety}, or more precisely, the variety of
characters of representations of $\Gamma$ in $G$.  See
\cite{culler-shalen} \cite[Sec.~II.4]{morgan-shalen:valuations-trees}
\cite{heusener-porti} for detailed discussion of these spaces.  Both
$\sR(\Gamma,G)$ and $\X(\Gamma,G)$ are affine algebraic varieties
defined over $\QQ$.  The ring $\QQ[\X(\Gamma,\SL_2\C)]$ is generated
by the \emph{trace functions} $\{ t_\gamma \}_{\gamma \in \Gamma}$
which are induced by the conjugation-invariant functions on
$\sR(\Gamma,G)$ defined by
$$ t_\gamma(\rho) =  \tr(\rho(\gamma)).$$
Similarly the ring $\QQ[\X(\Gamma,\PSL_2\C)]$ is generated by the
squares of trace functions.

There are two types of natural maps between character varieties that we
will use in the sequel.  First, the covering map $\SL_2\C \to
\PSL_2\C$ induces a map of character varieties
$$ r : \X(\Gamma, \SL_2\C) \to \X(\Gamma,\PSL_2\C),$$
which is a finite-to-one, proper, and whose image is a union of
irreducible components; in fact, the group $H_1(\Gamma,\Z / 2 \Z)$
acts on $\X(\Gamma,\SL_2\C)$ by biregular maps, and $r$ is the
quotient mapping for this action
\cite[Sec.~V.1]{morgan-shalen:degenerations-iii}.  Secondly, if $\phi
: \Gamma \to \Gamma'$ is a group homomorphism, then composing
representations with $\varphi$ induces a map of character varieties
$$ \varphi^* : \X(\Gamma',G) \to \X(\Gamma,G),$$
which is a regular map.

These constructions are functorial in the sense that the maps $r$ and
$\varphi^*$ fit into a commutative diagram
\begin{equation}
\label{eqn:functorial}
\begin{tikzpicture}[baseline=(current bounding box.center)]
\matrix (m) [matrix of math nodes, column sep=2em, row sep = 1.5em, text height=1.5ex, text depth=0.25ex]
{
\X(\Gamma',\SL_2\C) & \X(\Gamma,\SL_2\C)\\
\X(\Gamma',\PSL_2\C) & \X(\Gamma,\PSL_2\C)\\
};
\path[->,font=\scriptsize]
(m-1-1) edge node[above] {$\varphi^*$} (m-1-2)
(m-2-1) edge node[above] {$\varphi^*$} (m-2-2)
(m-1-1) edge node[left] {$r$} (m-2-1)
(m-1-2) edge node[right] {$r$} (m-2-2);
\end{tikzpicture}
\end{equation}

Since character varieties we consider are for groups of the form $\Gamma
=\pi_1N$ where $N$ is a compact $2$- or $3$-manifold, we often use the
abbreviated notation
$$ \X(N,G) := \X(\pi_1N, G).$$
Note that this algebraic variety does not depend on a choice of
orientation for $N$.

\subsection{The Morgan-Shalen compactification}
  \label{sec:morgan-shalen}

In \cite{morgan-shalen:valuations-trees} a compactification of $\X(\Gamma,\SL_2\C)$ is defined by mapping
$\X(\Gamma,\SL_2\C)$ into the infinite-dimensional projective space
$\P(\R^{\Gamma}) := (\R^{\Gamma} \setminus \{0\}) / \R^+$ by
\begin{equation}
\label{eqn:morgan-shalen}
[\rho] \mapsto \left ( \log (|t_\gamma(\rho)| + 2) \right
)_{\gamma \in \Gamma}.
\end{equation}
The image of this map is precompact and taking the closure gives the
\emph{Morgan-Shalen compactification} of $\X(\Gamma,\SL_2\C)$.
A boundary point $[ \ell ]$ of this compactification is a projective
equivalence class of functions $\ell : \Gamma \to \R$, and any
function arising this way is the translation length function of an
action of $\Gamma$ on an $\R$-tree by isometries.  These $\R$-trees
are constructed algebraically, using valuations on the function fields
of subvarieties of $\X(\Gamma,\SL_2\C)$.  The intermediate stages of
this algebraic construction also involve $\Lambda$-trees of higher
rank.  For later use we will now recall some key steps in their
construction.

\subsection{Valuation constructions}
In what follows we consider irreducible subvarieties $V \subset
\X(\Gamma,\SL_2\C)$, and $k$ will denote a countable subfield of $\C$
over which $V$ is defined.  (For example if $\X(\Gamma,\SL_2\C)$ is
irreducible we can take $V = \X(\Gamma,\SL_2\C)$ and $k=\QQ$.)  The
function field $k(V)$ of such a variety is a finitely generated
extension of $k$, and we consider $k$-valuations $v : k(V)^* \to
\Lambda$, where $\Lambda$ is an ordered abelian group.  Without loss
of generality we can assume that $\Lambda$ has finite rank
\cite[Ch.~5, Sec.~10]{zariski-samuel-ii}.  A valuation is
\emph{supported at infinity} if there exists a regular function $f \in
k[V]$ with $v(f) < 0$.

Boundary points of the Morgan-Shalen compactification correspond to
valuations as follows:

\begin{thm}[{\cite[Thm.~I.3.6]{morgan-shalen:valuations-trees}}]
\label{thm:valuation-from-boundary-point}
If $V \subset \X(\Gamma,\SL_2\C)$ is an irreducible subvariety defined over
$k$ and $[\ell]$ is a boundary point of $V$ in the Morgan-Shalen
compactification, then there exists a valuation $v : k(V)^* \to
\Lambda$ such that
\begin{rmenumerate}
\item $v$ is \emph{supported at infinity},
\item If $\Lambda_1 \subset \Lambda$ is the minimal nontrivial convex
subgroup, then for each $\gamma \in \Gamma$ either $v(t_\gamma) > 0$ or
$v(t_\gamma) \in \Lambda_1$, and 
\item There is an order-preserving embedding $p : \Lambda_1 \to \R$ such
that $\ell(\gamma) = p(\max(-v(t_\gamma),0))$. \noproof
\end{rmenumerate}
\end{thm}

Note that the embedding $p$ is unique up to multiplication by a
positive constant (by Theorem \ref{thm:hahn}) and that condition (iii)
above shows that $\ell$ can be recovered from the valuation $v$.

The link between valuations and $\Lambda$-trees is given by:

\begin{thm}[{\cite[Thm.~II.4.3 and Lem.~II.4.5]{morgan-shalen:valuations-trees}}]
\label{thm:tree-from-valuation}
If $V \subset \X(\Gamma,\SL_2\C)$ is an irreducible subvariety defined over
$k$ and $v : k(V)^* \to \Lambda$ is a valuation supported at infinity,
then there is an isometric action of $\Gamma$ on a $\Lambda$-tree
whose translation length function $\ell : \Gamma \to \Lambda$
satisfies
$$ \ell(\gamma) = \max(-v(t_\gamma),0).$$\noproof
\end{thm}

\begin{remark}
The statement of Lemma II.4.5 in \cite{morgan-shalen:valuations-trees}
involves only $\R$-trees, however a $\Lambda$-tree satisfying the
conditions above is constructed as part of its proof.  The lemma also
involves an additional condition on the valuation, (equivalent to (ii)
of Theorem \ref{thm:valuation-from-boundary-point} above), but this
condition is only used at the final step to produce an $\R$-tree from
the $\Lambda$-tree.  Additional discussion of the $\Lambda$-tree
construction underlying Theorem \ref{thm:tree-from-valuation} can be
found in \cite[Thm.~16]{morgan:bulletin} and
\cite[pp.~232--233]{shalen:introduction}.
\end{remark}

\subsection{The extension variety}
  \label{sec:extension}

Let $M$ be a compact $3$-manifold with connected boundary $S$, and let
$i_* : \pi_1S \to \pi_1M$ be the map induced by inclusion of the
boundary.  As discussed above, such a homomorphism induces a map of
character varieties
$$ i^* : \X(M,G) \to \X(S,G).$$
We call this the \emph{restriction map}.  Since it is a regular map of
algebraic varieties, the image of $i^*$ is a constructible set which
contains a Zariski open subset of its closure.

Considering the case $G = \SL_2\C$, we denote the closure of the image by
$$ \E_M := \overline{i^*(\X(M,\SL_2\C))}^{\text{Zariski}},$$
which we call the \emph{extension variety}, since its generic points
are conjugacy classes of representations of $\pi_1S$ that are trivial
on $\ker(i_*)$ and which admit an extension from $i_*(\pi_1S)$ to its
supergroup $\pi_1M$.  Note that $\E_M$ is an algebraic subvariety of
$\X(S,\SL_2\C)$.

Since points in the Morgan-Shalen boundary of $\X(S,\SL_2\C)$ correspond to
length functions of actions of $\pi_1S$ on $\R$-trees, it is natural
to expect that the length functions that arise as boundary points of
$\E_M$ would have a similar extension property.  We now show that this
is true if we allow the extended length function to takes values in a
higher-rank group, $\R^n$ with the lexicographical order.

\begin{thm}
\label{thm:length-function-extension}
Let $[\ell]$ be a boundary point of $\E_M$ in the Morgan-Shalen
compactification of $\X(S,\SL_2\C)$.  Then $\ell : \pi_1S \to \R$ extends to a
length function of an action of $\pi_1M$ on a $\R^n$-tree, i.e.~there
exists a function $\hat{\ell} : \pi_1M \to \R^n$ such that
\begin{rmenumerate}
\item The group $\pi_1M$ acts isometrically on a $\R^n$-tree with
translation length function $\hat{\ell}$.

\item For each $\gamma \in \pi_1S$ we have 
$$\hat{\ell}(i_*(\gamma)) = i_n(\ell(\gamma))$$
where $i_* : \pi_1S \to \pi_1M$ is the map induced by the inclusion of
$S$ as the boundary of $M$ and $i_n : \R \to \R^n$ is the
order-preserving inclusion as the last (least significant) factor.
\end{rmenumerate}
\end{thm}

\begin{proof}
Since $[\ell]$ is a boundary point of $\E_M$, it is a boundary point
of one of its irreducible components.  Let $\E_M^0$ be such a
component, and let $\X^0_M$ be a corresponding irreducible component
of $\X(M,\SL_2\C)$ so that $i^*(\X^0_M)$ contains a Zariski open subset of $\E_M^0$.

Since $i^* : \X^0_M \to \E_M^0$ is dominant, it induces an extension
of function fields $k(\E_M^0) \into k(\X^0_M)$, where $k$ is a finite
extension of $\QQ$ over which $\E_M^0$ and $\X^0_M$ are defined.
Note that when considering $k(\E_M^0)$ as a subfield of $k(\X^0_M)$,
the element of $k(\E_M^0)$ represented by the trace function
$t_\gamma$, $\gamma \in \pi_1S$, is identified with element of
$k(\X^0_M)$ represented by the trace function $t_{i_*(\gamma)}$.

Let $v : k(\E_M^0)^* \to \Lambda$ and $p : \Lambda_1 \to \R$ be the
valuation and embedding associated to $\ell$ by Theorem
\ref{thm:valuation-from-boundary-point}.  Since $k(\X^0_M)$ is
finitely generated over $k(\E_M^0)$, the standard extension theorem for
valuations (see \cite[Thm.~5', p.~13]{zariski-samuel-ii},
\cite[Lem.~II.4.4]{morgan-shalen:valuations-trees}) gives an ordered abelian group $\Lambda'$
such that $\Lambda \subset \Lambda'$ and $\Lambda' \subset m \Lambda$
for some $m \in \N$, and a valuation 
$$v' : k(\X^0_M)^* \to \Lambda'$$
so that $v'(f) = v(f)$ for any $f \in k(\E_M)$.  Since $\Lambda'
\subset m \Lambda$ it follows that the minimal convex subgroups
satisfy $\Lambda_1 \subset \Lambda'_1$.

By Lemma \ref{lem:super-hahn}, we have a commutative diagram of
order-preserving embeddings
\begin{equation}
\label{eqn:inclusions}
\begin{tikzpicture}[baseline=(current bounding box.center)]
\matrix (m) [matrix of math nodes, column sep=1.7em, row sep = 1.7em, text height=1.5ex, text depth=0.25ex]
{
\Lambda_1 & \Lambda'\\
\R & \R^n\\
};
\path[->,font=\scriptsize]
(m-1-1) edge (m-1-2)
(m-1-2) edge node[right] {$F$} (m-2-2)
(m-1-1) edge node[left] {$p$} (m-2-1)
(m-2-1) edge node[below] {$i_n$} (m-2-2);
\end{tikzpicture}
\end{equation}
We can arrange that the left vertical map in this diagram agrees with the
embedding $p : \Lambda_1 \to \R$ considered above; this is possible
since there is a unique such embedding up to scale (by Theorem
\ref{thm:hahn}), and both vertical maps in the diagram can be
multiplied by an arbitrary positive constant while preserving
commutativity and order.

Applying Theorem \ref{thm:tree-from-valuation} to $v'$ we obtain a
$\Lambda'$-tree $T'$ on which $\pi_1M$ acts by isometries with length
function $\ell'$.  Let $T = T' \tensor_{\Lambda'} \R^n$ be the
$\R^n$-tree associated to $T'$ by the embedding $F$, and let
$\hat{\ell} : \pi_1M \to \R^n$ be its length function.  Condition (i)
is satisfied by definition.

It remains to verify condition (ii).  For any $\gamma \in
\pi_1S$ we have $t_{i_*(\gamma)} \in k(\X^0_M)$ and the length
function $\hat{\ell}$ satisfies:

\begin{align}
\label{eqn:length-extension-argument}
\notag \hat{\ell}(i_*(\gamma)) &= F(\ell'(i_*(\gamma))) & \text{ by definition
  of }T\\
\notag &= F(\max(-v'(t_{i_*(\gamma)}),0)) & \text{ by Theorem \ref{thm:tree-from-valuation}}\\
 &= F(\max(-v(t_\gamma),0)) & \text{ since }v'\text{ extends }v\\
\notag &= i_n(p(\max(-v(t_\gamma),0))) & \text{ by commutativity of \eqref{eqn:inclusions}}\\
\notag &= i_n(\ell(\gamma)) & \text{ by Theorem
  \ref{thm:valuation-from-boundary-point}}
\end{align}
\end{proof}

\begin{remark}
It is natural to ask whether the extended length function $\hat{\ell}$
of Theorem \ref{thm:length-function-extension} can always be taken to
be $\R$-valued, thus avoiding the introduction of $\R^n$-trees
($n>1$).  The proof above shows a potential obstruction.  For any
$\gamma \in \pi_1S$ the valuation $v(t_\gamma)$ is either positive or
it lies in the rank-$1$ convex subgroup $\Lambda_1$, but it is not clear
whether this holds for the extended valuation $v'$ applied to a trace
function of an element in $\pi_1M$.  A rank-$1$ subgroup containing
the negative valuations of all trace functions is needed in order to
apply the construction of
\cite[Sec.~II.4]{morgan-shalen:valuations-trees} to produce an
$\R$-tree from the $\Lambda$-tree while preserving the action of
$\pi_1M$ and the length function.

Since these valuations are associated to boundary points of the
Morgan-Shalen compactification, this is effectively a question about
comparing the rate of growth of traces in a sequence of
$\pi_1S$-representations in $\E_M$ to that of an associated sequence
of $\pi_1M$-representations.  Alternatively, in the terminology of
\cite[Sec.~6]{dumas:holonomy}, we ask whether the \emph{local scales} of
$\PSL_2\C$-representations of $\pi_1M$ are comparable (within a
uniform multiplicative constant) to those of the restrictions to
$i_*(\pi_1S)$, or if such a uniform comparison is possible for
\emph{some} sequence representing any given boundary point.
\end{remark}

\subsection{The holonomy variety}
  \label{sec:holonomy-variety}

Let $X$ be a marked Riemann surface structure on $S$.  Here we allow
that the complex structure of $X$ induces an orientation opposite that
of $S$, so either $X \in \T(S)$ or $X \in \T(\bar{S})$.

The vector space $Q(X)$ of holomorphic quadratic differentials can be
identified with the set of complex projective structures
($\CP^1$-structures) on $X$.  Here $0 \in Q(X)$ corresponds to the
projective structure induced by the uniformization of $X$.

Each $\CP^1$ structure on $X$ has an associated holonomy
representation $\pi_1S \to \PSL_2\C$, which is well-defined up to
conjugacy.  Considering the conjugacy class of the holonomy
representation as a function of the projective structure gives the
\emph{holonomy map}
$$\hol : Q(X) \to \X(S,\PSL_2\C).$$  
This map can be lifted through $r : \X(S,\SL_2\C) \to \X(S,\SL_2\C)$
in several ways; the set of such lifts is naturally in bijection with the set $\mathrm{Spin}(X)$
of spin structures on $X$, which is a finite set acted upon simply transitively by
$H_1(S, \Z / 2\Z)$.  For each $\spin \in \mathrm{Spin}(X)$ we denote the
corresponding lifted holonomy map by
$$\hol_\spin : Q(X) \to \X(S,\SL_2\C),$$
so $\hol = r \circ \hol_\spin$.  The maps $\hol$ and
$\hol_\spin$ are proper holomorphic embeddings
\cite[Thm.~11.4.1]{gkm}.

Define $\sH_{X,\spin} := \hol_\spin(Q(X))$, which is
therefore a complex analytic subvariety of $\X(S,\SL_2\C)$.  Taking
the union of these subvarieties we obtain the \emph{holonomy variety}
$$ \sH_X := \bigcup_{\spin \in \mathrm{Spin}(X)}
\sH_{X,\spin} \: \subset \: \X(S,\SL_2\C), $$ an analytic variety with irreducible components
$\sH_{X,\spin}$.  Equivalently, we have $\sH_X = r^{-1}(\hol(Q(X)))$.

We will be interested in the limiting behavior of $\sH_X$ in the
Morgan-Shalen compactification and how this relates to the
parameterizations of its components by $Q(X)$.
Consider a divergent sequence $\phi_n \in Q(X)$.  Since the unit
sphere in $Q(X)$ is compact, by passing to a subsequence we can assume
that $c_n \phi_n \to \phi$ as $n \to \infty$, where $c_n \in \R^+$ is
a suitable sequence of scale factors with $c_n \to 0$.  We call a
limit point $\phi$ obtained this way a \emph{projective limit} of
$\{\phi_n\}.$ Projective limits in $Q(X)$ are related to limits of
holonomy representations in $\X(S,\SL_2\C)$ as follows:

\begin{thm}[{\cite[Thm.~A]{dumas:holonomy}}]
\label{thm:holonomy-limits}
If $\phi_n \in Q(X)$ is a divergent sequence with projective limit
$\phi$, then any accumulation point of $\hol_\spin(\phi_n)$ in the
Morgan-Shalen boundary is represented by an $\R$-tree $T$ that admits
an equivariant, surjective straight map $T_\phi \to T$. \noproof
\end{thm}

\subsection{Intersections and the isotropic cone}

We now consider the intersection of the holonomy variety and the
extension variety.  Since $\sH_{X,\spin}$ is parameterized by
$Q(X)$, the intersection $\sH_{X,\spin} \cap \E_M$ is
parameterized by a subset of $Q(X)$ which we denote by
$$ \V_{M,\spin} = \hol_\spin^{-1}(\E_M) \subset Q(X),$$
and $\sH_X \cap \E_M$ is the union of the images of these sets under
the respective holonomy maps.

Combining the main results of Section \ref{sec:isotropic2} with
theorems
\ref{thm:length-function-extension}--\ref{thm:holonomy-limits}, we
have the following characterization of the limit points of
$\V_{M,\spin}$:

\begin{thm}
\label{thm:extensible-holonomy}
Let $\{ \phi_n \} \subset \V_{M,\spin}$ be a
divergent sequence with projective limit $\phi$.  Then $[\F(\phi)] \in
\L_{M,X}$ where $\L_{M,X}$ is the isotropic cone of Theorem
\ref{thm:straight-isotropic}.
\end{thm}

Note that in this theorem we regard $\L_{M,X}$ as a subset of $\MF(S)$
regardless of whether $X \in \T(S)$ or $X \in \T(\bar{S})$.  This is
possible since the natural identification $\MF(S) \simeq \MF(\bar{S})$
preserves the property of being an isotropic cone (while changing the
sign of the symplectic form).

\begin{proof}
Let $[\ell]$ be an accumulation point of $\hol_\spin(\phi_n)$ in the
Morgan-Shalen compactification.  By Theorem \ref{thm:holonomy-limits},
there is a surjective, equivariant straight map $T_\phi \to T$ where
$T$ is an $\R$-tree on which $\pi_1S$ acts with length function
$\ell$.

By Theorem \ref{thm:length-function-extension}, the length function
$\ell$ extends to a length function $\hat{\ell} : \pi_1M \to \R^n$ of
an action of $\pi_1M$ on a $\R^n$-tree $\hat{T}$.  By Lemma
\ref{lem:subtree} there is an $\R$-tree $T' \subset \hat{T}$ on which
$\pi_1S$ acts with length function $\ell = \left .\hat{\ell}
\right|_{\pi_1S}$ (where we use the embedding $i_n : \R \to \R^n$ as
the least significant factor to identify $\R$ with the minimal convex
subgroup of $\R^n$).  Note that $T$ and $T'$ are then isospectral
$\R$-trees in the terminology of Section \ref{sec:length-isotropic}.

Finally we apply Theorem \ref{thm:length-isotropic} to the isospectral
$\R$-trees $T,T'$, straight map $T_\phi \to T$, the inclusion of $T'$
as a subtree of $\hat{T}$ to conclude that $[\F(\phi)] \in \L_{M,X}$.
\end{proof}

\section{Analytic geometry in $Q(X)$}
\label{sec:analytic-curves}

Theorem \ref{thm:extensible-holonomy} shows that the large-scale
behavior of the set $\V_{M,\spin} \subset Q(X)$ is constrained by an isotropic
cone in $\MF(S)$.  The goal of this section is to complete the proof
of Theorem \ref{thm:main-intersection} by showing that only a discrete
set can satisfy this constraint.

We begin with some generalities on real and complex limit points of
sets in a complex vector space.

\subsection{Real and complex boundaries}
  \label{sec:real-and-complex}

The vector space $\C^n$ can be compactified to $\CP^{n}$ by
adjoining a hyperplane at infinity $\CP^{n-1} = (\C^n \setminus
\{0\})/\C^*$.  When considering this compactification we regard
$\CP^{n-1}$ as the \emph{complex boundary} of $\C^n$, and write $\cbdy
\C^n = \CP^{n-1}$.

Given a set $R \subset \C^n$, let $\partial_\C R \subset \CP^{n-1}$
denote its set of accumulation points in complex boundary $\partial_\C
R = \bar{R} \cap \CP^{n-1}$ where $\bar{R}$ is the closure of $R$ in
$\CP^n$.

In the real analogue of these constructions we identify the sphere
$S^{2n-1}$ with the set of rays from the origin in $\C^n$; for any $q
\in S^{2n-1}$ let $r_q \subset \C^n$ denote the corresponding open
ray.  There is a corresponding compactification $\bar{B}^{2n} = \C^n
\sqcup S^{2n-1}$ where $\bar{B}^{2n}$ is the closed ball of dimension
$2n$; here a sequence $z_k$ converges to $q \in S^{2n-1}$ if it can be
rescaled by positive real constants so as to converge to a point in
$r_q$.  In this sense $S^{2n-1}$ is the \emph{real boundary} of $\C^n$
and write $\rbdy \C^n = S^{2n-1}$.

Given a set $R \subset \C^n$, let $\partial_\R R \subset S^{2n-1}$
denote the accumulation points of $R$ in the boundary of the real
compactification $\bar{B}^{2n}$.  This real boundary of $R$
corresponds to a set of open rays in $\C^n$, and we denote by
$\Cone_\R(R)$ the union of these rays, i.e.
$$ \Cone_\R(R) = \bigcup_{q \in \partial_\R R} r_q.$$
Equivalently $\Cone_\R(R)$ is the set of projective limits of
the set $R$ in the sense of Section \ref{sec:holonomy-variety}.

Mapping a ray in $\C^n$ to the complex line it spans induces the
\emph{Hopf fibration} $\Pi : S^{2n-1} \to \CP^{n-1}$, which gives
$S^{2n-1}$ the structure of a principal $S^1$-bundle. Identifying
$S^{2n-1}$ with the unit sphere in $\C^n$, the map $\Pi$ is
the restriction of the quotient map $\hat{\Pi} : (\C^n \setminus
\{0\}) \to \CP^{n-1}$, which is holomorphic.  It is immediate from the
definitions that for any set $R \subset \C^n$ we have $\cbdy R =
\Pi(\rbdy R)$.

We call a fiber of $\Pi$ a \emph{Hopf circle}. For any Hopf circle $C
= \Pi^{-1}(p) \subset S^{2n-1}$, the set $\bigcup_{q \in C} r_q =
\hat{\Pi}^{-1}(p) \subset \C^n$ is a punctured complex line, and more
generally if $I \subset S^{2n-1}$ is an open arc of a Hopf circle,
then $\bigcup_{q \in I} r_q$ is an open sector in a complex line.

The following elementary lemma allows us to recognize totally real submanifolds
of $\CP^{n-1}$ arising as boundaries of cones in $\C^n$:

\begin{lem}
\label{lem:totally-real}
Let $L \subset \C^n \setminus \{0\}$ be an $\R^+$-invariant and
totally real submanifold of dimension $m$.  Then $\cbdy L$ (i.e.~the
projection of $L$ to $\CP^{n-1}$) is an immersed totally real
submanifold of $\CP^{n-1}$ of dimension $m-1$.
\end{lem}

\begin{proof}
By $\R^+$-invariance, intersecting $L$ with the unit sphere gives a
manifold $L_1 \subset S^{2n-1}$ of dimension $m-1$ naturally
identified with $\rbdy L$, and $\cbdy L = \Pi(L_1)$.  Let $x \in L_1$.
Because $L$ is totally real, the tangent space to $T_xL_1$ is
transverse to $i \R \cdot x = \ker d \Pi_x$, thus $\left . \Pi
\right|_{L_1}$ is an immersion.  Since the $\C$-span of $T_xL_1$ has
complex dimension $m-1$ and is transverse to $\C \cdot x$, the
differential $d \hat{\Pi}$ maps it injectively and complex-linearly to
$T_{\Pi(x)} \CP^{n-1}$.  The image of the totally real subspace
$T_xL_1$ under such a map is totally real, and the lemma follows.
\end{proof}

While the real and complex boundary constructions have been described
for $\C^n$, they apply naturally to any finite-dimensional
complex vector space; we will apply them to $Q(X) \simeq \C^{3g-3}$.

\subsection{Real and complex boundaries of $\V_{M,\spin}$}
  \label{sec:boundary-of-vm-epsilon}

Using the symplectic properties of the foliation map (Theorem
\ref{thm:symplectomorphism}), we can now translate the properties of
$\V_{M,\spin}$ established in Theorem
\ref{thm:extensible-holonomy} into analytic conditions satisfied by
its real and complex boundaries:

\begin{thm}
\label{thm:real-and-complex}
Let $\V_{M,\spin} = \hol_\spin^{-1}(\E_M) \subset Q(X)$ be
the set of holomorphic quadratic differentials corresponding to
projective structures with holonomy in the extension variety of the
$3$-manifold $M$.  Then:

\begin{rmenumerate}
\item $\cbdy \V_{M,\spin}$ is locally contained in a totally real manifold.\\
That is, either $\cbdy \V_{M,\spin} = \emptyset$ or there exists $p \in \cbdy
\V_{M,\spin}$, a neighborhood $U$ of $p$ in $\cbdy Q(X)$, and a totally real,
real-analytic submanifold $N \subset U$ such that $(\cbdy \V_{M,\spin} \cap
U) \subset N$.

\item $\rbdy \V_{M,\spin}$ does not contain an open arc of any Hopf circle.\\
That is, for any $p \in \cbdy \V_{M,\spin}$ the intersection
$\Pi^{-1}(p) \cap \rbdy \V_{M,\spin}$ has empty interior in the
relative topology of $\Pi^{-1}(p) \simeq S^1$.
\end{rmenumerate}
\end{thm}

\begin{prooflist}
\item Define
$$\sC_{M,\spin} := \Cone_\R(\V_{M,\spin}),$$
and suppose that $\cbdy \V_{M,\spin} \neq \emptyset$ so that
$\sC \neq \emptyset$.  Note that $\rbdy \V_{M,\spin} = \rbdy
\sC_{M,\spin}$ and $\cbdy \V_{M,\spin} = \cbdy
\sC_{M,\spin}$.  By Theorem \ref{thm:extensible-holonomy} we
have $\sC_{M,\spin} \subset \F^{-1}(\L_{M,X})$ and so
$$ \cbdy \V_{M,\spin} \subset \cbdy \F^{-1}(\L_{M,X}).$$
Thus it will suffice to locally cover $\cbdy \F^{-1}(\L_{M,X})$ by a
totally real, real-analytic manifold and to ensure that this set
contains a limit point of $\sC_{M,\spin}$.

Among the strata $Q_i(X)$ intersected by $\sC_{M,\spin}$, let $Q_k(X)$ be a
maximal element (one not contained in the boundary of another stratum
intersecting $\sC_{M,\spin}$).  Thus $Q_k(X)$ has an open tubular neighborhood
$U_0 \subset Q(X)$ disjoint from all other strata intersecting $\sC_{M,\spin}$,
so that $\emptyset \neq (\sC_{M,\spin} \cap U_0) \subset Q_k(X)$.  Furthermore
since $Q_k(X)$ is $\C^*$-invariant, we can choose $U_0$ to be
$\C^*$-invariant as well.  In particular the set of complex lines in
$U_0$ is an open set $\cbdy U_0 \subset \cbdy Q(X)$.

By Theorem \ref{thm:symplectomorphism}, the intersection
$(\F^{-1}(\L_{M,X}) \cap Q_k(X))$ is a semianalytic set, i.e.~it is
locally defined by finitely many equations and inequalities of
real-analytic functions.  Indeed, for each $p \in (\sC_{M,\spin} \cap U_0)$ the
theorem gives an open neighborhood $V$ of $p$ in $Q_k(X)$ such that
$\F : V \to \ML(\tau)$ is a real-analytic map into a train track chart,
and $\ML(\tau) \cap \L_{M,X}$ is a union of convex cones in linear
subspaces (thus semianalytic).

Since $(\F^{-1}(\L_{M,X}) \cap Q_k(X))$ is a union of rays in $Q(X)$,
its real boundary is the same as its intersection with the unit sphere
in $Q(X)$ and is also semianalytic.  Thus the set
$$ \cbdy (\F^{-1}(\L_{M,X}) \cap Q_k(X)) = \Pi(\rbdy
(\F^{-1}(\L_{M,X}) \cap Q_k(X))$$ is the image of a semianalytic set
under a proper real-analytic mapping, i.e.~a subanalytic set.  Such
sets can be can be stratified by connected, real-analytic, subanalytic
manifolds (see \cite{hironaka:subanalytic} \cite{hardt:stratification}), and
furthermore stratifications of $\rbdy (\F^{-1}(\L_{M,X}) \cap Q_k(X))$
and $\cbdy (\F^{-1}(\L_{M,X}) \cap Q_k(X))$ can be chosen so that
$\Pi$ maps strata to strata and so that the differential of $\Pi$ has
constant rank on each stratum \cite[Cor.~4.4]{hardt:stratification}.

Let $N \subset \cbdy (\F^{-1}(\L_{M,X}) \cap Q_k(X))$ be a stratum maximal
among those intersecting $\cbdy \V_{M,\spin}$. Taking a tubular neighborhood
$U$ of $N$ in $\cbdy U_0$ gives an open set in $\cbdy Q(X)$ in which
$\emptyset \neq (\cbdy \V_{M,\spin} \cap U) \subset (N \cap U)$.

It remains to show that the real-analytic manifold $N$ is totally
real. Since $\Pi$ is surjective, for each $p \in N$ there is a stratum
$$N' \subset \rbdy (\F^{-1}(\L_{M,X}) \cap Q_k(X))$$ such that $\Pi(N')$
contains an open neighborhood of $p$ in $N$.  Thus $N'' :=
\Cone_\R(N')$ is a real-analytic submanifold of $Q(X)$ and $N''
\subset (\F^{-1}(\L_{M,X}) \cap Q_k(X))$.

Theorem \ref{thm:symplectomorphism} then implies that $N''$ locally
maps by $\F$ into a finite union of isotropic subspaces of a train
track chart $\ML(\tau)$, and that in these local charts $\F$ is a
real-analytic, symplectic map.  Thus $N''$ is isotropic with respect
to the symplectic structure of $Q_k(X)$ given by Theorem
\ref{thm:kahler}.  Since the symplectic form on $Q_k(X)$ is induced by
a K\"ahler structure, an isotropic manifold is totally real.  We have
therefore described an open neighborhood of an arbitrary point $p \in
N$ as $\Pi(\rbdy N'') = \cbdy N''$ where $N'' \subset Q(X)$ is
$\R^+$-invariant and totally real.  It follows by Lemma
\ref{lem:totally-real} that $N$ itself is totally real.

\item Suppose on the contrary that $\rbdy \V_{M,\spin}$ contains an open arc
of a Hopf circle, or equivalently that the set $\sC_{M,\spin}$ contains an
open sector $D$ in a complex line in $Q(X)$.  Then $D$ lies in a
stratum $Q_i(X) \subset Q(X)$, so after possibly shrinking $D$ we can
apply Theorem \ref{thm:symplectomorphism} as in the previous paragraph
to conclude that $D$ is totally real, a contradiction.
\end{prooflist}

In order to show that $\V_{M,\spin}$ is discrete, we will derive a
contradiction from the conditions (i)--(ii) of Theorem
\ref{thm:real-and-complex} and the assumption that $\V_{M,\spin}$ contains an
analytic curve.  The next few paragraphs develop necessary machinery
for analyzing the real and complex boundaries of such curves, after
which we return to the proof of Theorem \ref{thm:main-intersection} in
sections \ref{sec:hopf-circles}--\ref{sec:discreteness}.

\subsection{Tangent cones and analytic curves}

A general reference for the following material is
\cite{chirka:complex-analytic-sets}.  If $E$ is a subset of $\C^n$ and
$p \in \C^n$, the \emph{tangent cone} of $E$ at $p$ is the set $C(E,p)
\subset \C^n$ of points of the form
$$ \lim_{k \to \infty} t_k(z_k - p) $$
where $z_k \in E$, $z_k \to p$, and $t_k \to 0^+$ as $k \to \infty$.  The following
basic properties of the tangent cone follow immediately from the
definition:
\begin{enumerate}
\item $C(E,p)$ is a closed $\R^+$-invariant set.

\item $C(E,p) \neq 0$ if and only if $p$ lies in the closure
$\bar{E}$.
\item For any $E_1,E_2 \subset \C^n$ we have $C(E_1 \cap
E_2,p) \subset C(E_1,p) \cap C(E_2,p)$.
\item If $E$ is a submanifold in a neighborhood of $p \in E$,
then $C(E,p)$ is the tangent space $T_pE$.
\end{enumerate}
It follows from (4) that if $E$ is an
analytic curve (i.e.~a complex analytic set of dimension $1$) and $p
\in E$ is a smooth point then $C(E,p)$ is a complex line.  More
generally, the tangent cone of an analytic curve at any point is a
finite union of complex lines.

Furthermore, just as a complex submanifold is locally a graph over its
tangent space, in a neighborhood of a point an analytic curve can be
parameterized as follows (see \cite[Sec.~1.6.1]{chirka:complex-analytic-sets}):

\begin{lem}
\label{lem:analytic-curve-parameterization}
Let $U$ be an open neighborhood of $0 \in \C^n$ and let $E \subset U$
be an analytic curve containing $0$ and irreducible at that point.
Suppose that $C(E,0) = \{ (w_1,0,\ldots,0) \: | \: w_1 \in \C \}$.  Then there exists a
natural number $m$ and a holomorphic map $f : (\Delta,0) \to (E,0)$
such that
$f(\zeta) = (f_1(\zeta), \ldots, f_n(\zeta))$ where
$$ f_1(\zeta) = \zeta^m$$
and for each $k > 1$ we have
$$ f_k(\zeta) = \zeta^{m_k} h_k(\zeta)$$
where $m_k > m$ and either $h_k(\zeta) \equiv 0$ or $h_k(\zeta)$ is
holomorphic with $h_k(0) \neq 0$.  The image $f(\Delta)$ contains $E
\cap U'$ for some neighborhood $U'$ of $0$.  \noproof
\end{lem}

Of course one can permute coordinates to obtain a similar
parameterization when $C(E,0)$ is any of the coordinate axes in
$\C^n$.  Also note that in this lemma the number $m$ is the
\emph{multiplicity} of $E$ at $p$, and $m>1$ if and only if $p$ is a
singular point.

\subsection{Analytic curves near a totally real manifold}

Two results characterizing the behavior of an analytic curve near a
totally real submanifold of $\C^n$ will be essential in the sequel.
The first describes the tangent cone of an analytic curve in the
complement of a totally real manifold at a boundary point.  A (closed)
\emph{complex half-line} is a set of the form $\{ L(x + i y) \: | \: y
\geq 0 \}$ where $L : \C \to \C^n$ is an injective complex linear map.

\begin{thm}[{Chirka \cite[Prop.~19]{chirka:boundaries}}]
\label{thm:analytic-curve-tangent}
Let $E \subset (U \setminus M)$ be an analytic curve where $U \subset
\C^n$ is open and $M \subset U$ is a closed, totally real submanifold
of class $C^k$ for some $k > 1$.
  Then for any $p \in (\bar{E} \cap
M)$ the tangent cone $C(E,p)$ is a nonempty finite union of complex
lines and half-lines.
\end{thm}

With additional regularity for the submanifold $M$, one has the
following extension result:

\begin{thm}[{Chirka \cite[Sec.~1]{chirka:boundaries}, Alexander \cite{alexander:continuing}}]
\label{thm:analytic-curve-extension}
Let $E \subset (U \setminus M)$ be an analytic curve where $U \subset
\C^n$ is open and $M \subset U$ is a closed, totally real,
real-analytic submanifold.  Then for any $p \in (\bar{E} \cap M)$, the
set $E$ admits an analytic continuation near $p$,
i.e.~there exists a neighborhood $U'_p$ of $p$ and an analytic curve
$E'_p \subset U'_p$ such that $(E \cap U'_p) \subset E'_p$.
\end{thm}

Further discussion of this result can be found in \cite[Sec.~20.5]{chirka:complex-analytic-sets}.

\subsection{Hopf circles}
  \label{sec:hopf-circles}

If $E \subset \C^n$ is an algebraic curve, then $\partial_\C E$ is a
finite set and $\partial_\R C$ is the union of Hopf circles lying over
$\partial_\C E$.  The next theorem establishes a similar property of
$\partial_\R E$ when the algebraic assumption is replaced by the
condition that $\partial_\C E$ locally lie in a totally real manifold.

\begin{thm}
\label{thm:hopf-arc}
Let $E$ be an analytic curve in $\C^n$ and suppose that for some $p
\in \cbdy E$ there is a neighborhood $U$ of $p$ in $\cbdy \C^n$ and a
totally real, real-analytic submanifold $N$ of $U$ such that
$(\partial_\C E \cap U) \subset N$.  Then $\partial_\R E$ contains an
open arc of a Hopf circle.
\end{thm}

In the proof we will consider $\C^n$ as an affine chart of its
compactification $\CP^n = \C^n \cup \cbdy \C^n$, but other affine
charts of $\CP^n$ will also be used.  In order to distinguish among
them, we use the notation $\C^n_z$ for the original affine chart (in
which $E$ is an analytic curve) with coordinates $z_1, \ldots, z_n$,
and $\C^n_w$ will denote another affine chart with coordinates $w_1,
\ldots, w_n$.

\begin{proof}
Let $V$ be an open neighborhood of $p$ in $\CP^n$ such that $V \cap
\cbdy \C^n_z = U$.  After possibly shrinking $V$ and $U$ we can assume
that $V$ lies in an affine chart $\C^n_w$ of $\CP^n$.

By Theorem \ref{thm:analytic-curve-extension} we can assume, after
further shrinking $V$, that $E \cap V$ is a subset of an analytic
curve $E' \subset V$.  
One of the irreducible components of $E'$ at $p$ must intersect $E$,
so let $\hat{E}'$ denote such a component.  Then $\hat{E} = \hat{E}'
\cap E$ is an analytic curve in $U \cap \C^n_z$ such that $\partial_\R
\hat{E} \subset N$, and $p$ is an accumulation point of $\hat{E}$.  By
Theorem \ref{thm:analytic-curve-tangent} the set $C(\hat{E},p)$
contains a complex half-line $H$.  Since $\hat{E} \subset \hat{E}'$,
the complex line $L$ containing $H$ is one of the finite set of lines
comprising $C(\hat{E}',p)$.

There are now two cases to consider, based on the relative position of
$L$ and the hyperplane $\cbdy \C^n_z$:
\begin{enumerate}
\item $L$ is transverse to $\cbdy \C^n_z$ (equivalently, it intersects $\C^n_z$)
\item $L$ is contained in $\cbdy \C^n_z$ 
\end{enumerate}
Intuitively, these two cases correspond to whether the analytic curve
$E$ meets the hyperplane at infinity $\cbdy \C^n_z$ transversely or
tangentially at $p$.

\boldpoint{Case 1.}  By changing the affine chart $\C^n_w$ and making
an affine change of coordinates in $\C^n_z$ and $\C^n_w$ we put
$p$, $L$, and $H$ in a standard position; specifically, we 
suppose that in a homogeneous coordinate system for $\CP^n$ the
inclusions of $\C^n_z$ and $\C^n_w$ are given by
\begin{equation*}
\begin{split}
(z_1, \ldots, z_n) &\mapsto [z_1:\cdots:z_n:1],\\
(w_1, \ldots, w_n) &\mapsto [w_1:\cdots:w_{n-1}:1:w_n],
\end{split}
\end{equation*}
that $L$ is the $w_n$-axis in $\C^n_w$, and that $H \subset L$ is
defined by $\Im(w_n) \geq 0$.

Let $f : \Delta \to \C^n_w$ be a local parameterization of $\hat{E}'$
at $0$ as in Lemma \ref{lem:analytic-curve-parameterization}; note
that here the tangent cone of $\hat{E}'$ is the $w_n$-axis.  Then we
have for any $\zeta \neq 0$ that
\begin{equation*}
\begin{split}
f(\zeta) & = \left ( \zeta^{m_1} h_1(\zeta) , \ldots , \zeta^{m_{n-1}}
h_{n-1}(\zeta) , \zeta^{m} \right ) \in \C^n_w\\
& = \left ( \zeta^{m_1 - m} h_1(\zeta) , \ldots , \zeta^{m_{n-1} - m}
h_{n-1}(\zeta) , \zeta^{-m} \right ) \in \C^n_z.
\end{split}
\end{equation*}
Since $H$ is the tangent cone of $\hat{E} \subset \hat{E}'$ as a
subset of $\C^n_w$, for each
$\theta \in [0,\pi]$ we have a sequence $\zeta_k \to 0$ such that
$\lim_{k \to \infty} \arg(\zeta_k^m) = \theta$ and $f(\zeta_k) \in
\hat{E}$.
As a point in $\C^n_z$, $f(\zeta_k)$ lies on the same ray as
$|\zeta_k|^m f(\zeta_k)$ which has coordinates 
$$\left (|\zeta_k|^m
\zeta_k^{m_1 - m} h_1(\zeta_k) , \ldots , |\zeta_k|^m \zeta_k^{m_{n-1}
  - m} h_{n-1}(\zeta) , |\zeta_k|^m \zeta^{-m} \right ) \in \C^n_z.$$
Since $m_i > 0$ and $h_i(\zeta_k)$ is bounded as $\zeta_k \to 0$ for
each $1 \leq i \leq m-1$, the sequence $|\zeta_k|^m f(\zeta_k) \in
\C^n_z$ converges to $(0,\ldots,0,e^{-i \theta}) \in \partial_\R E$.
Since $\theta \in [0,\pi]$ was arbitrary, we find that $\partial_\R E$
contains half of the Hopf circle containing $(0,\ldots,0,1)$,
completing this case.

\boldpoint{Case 2.}  We begin as before, altering the argument as
necessary.

Choose coordinates so that $z_i$ and $w_i$ are related
to one another as above but now we take $L$ to be the $w_1$-axis and
$H$ the subset with $\Im(w_1) \geq 0$.  Parameterizing $\hat{E}'$ and
calculating as above we find that
\begin{equation*}
\begin{split}
f(\zeta) &= \left ( \zeta^m, \zeta^{m_2} h_2(\zeta), \ldots, \zeta^{m_n}
h_n(\zeta) \right ) \in C^n_w\\
& = \left ( \zeta^{m-m_n} h_n(\zeta)^{-1}, \zeta^{m_2 - m_n} h_2(\zeta)
h_n(\zeta)^{-1}, \ldots \right .\\
& \left . \;\;\;\;\;\;\;\; \ldots, \zeta^{m_{n-1} - m_n}
h_{n-1}(\zeta) h_n(\zeta)^{-1}, \zeta^{-m_n} h_n(\zeta)^{-1} \right ) \in C^n_z.
\end{split}
\end{equation*}
The coordinate expression in $\C^n_z$ is well-defined since
$h_n(\zeta) \neq 0$ for $\zeta$ in a small punctured neighborhood of
zero: Indeed, while Lemma \ref{lem:analytic-curve-parameterization}
includes the possibility that $h_n(\zeta) \equiv 0$, this would mean
that $\hat{E}' \cap \C^n_z = \emptyset$, contradicting the assumption
that $\hat{E}'$ intersects $E$.  By a further linear change of
coordinates we can also assume without loss of generality that $h_n(0)
= 1$.
As before the condition that $H$ is the tangent cone of $\hat{E}$
gives for each $\theta \in [0,\pi]$ a sequence $\zeta_k \to 0$ such
that $\lim_{k \to \infty} \arg(\zeta_k^m) = \theta$ and $f(\zeta_k)
\in \hat{E}$.
As a point in $\C^n_z$, $f(\zeta_k)$ lies on the same ray as
$|\zeta_k|^{m_n} f(\zeta_k)$ which has coordinates
\begin{multline*}
 \bigl ( |\zeta_k|^{m_n} \zeta_k^{m-m_n} h_n(\zeta_k)^{-1},
|\zeta_k|^{m_n} \zeta^{m_2 - m_n} h_2(\zeta_k)
h_n(\zeta_k)^{-1}, \\
 \ldots, |\zeta_k|^{m_n} \zeta_k^{m_{n-1} - m_n}
h_{n-1}(\zeta_k) h_n(\zeta_k)^{-1}, |\zeta_k|^{m_n} \zeta_k^{-m_n} h_n(\zeta_k)^{-1}
\bigr ) \in \C^n_z.
\end{multline*}
As $k \to \infty$ we find $|\zeta_k|^{m_n} f(\zeta_k) \to
(0,\ldots,0,e^{-(m_n/m) i \theta'}) \in \partial_\R E$ for some
$\theta' \equiv \theta$ mod $2 \pi$.  Allowing $\theta$ to vary over
$[0,\pi]$ we find that $\partial_\R E$ contains an arc of a Hopf
circle.
\end{proof}

\subsection{Discreteness}
  \label{sec:discreteness}

Using the above results on analytic curves near a totally real
manifold, the discreteness of $\sH_X \cap \E_M$ now follows easily:

\begin{proof}[Proof of Theorem \ref{thm:main-intersection} (connected boundary)]
Suppose on the contrary that the intersection $\sH_X \cap \E_M$ is not
discrete.  Since this set is a finite union of analytic subvarieties
$\{ \sH_{X,\spin} \cap \E_M \: | \: \spin \in \mathrm{Spin}(X)
\}$, at least one of these subvarieties is not discrete.  Thus there
exists some $\spin \in \mathrm{Spin}(X)$ so that $\sH_{X,\spin}
\cap \E_M$ contains an analytic curve, as does its preimage
$\V_{M,\spin} = \hol^{-1}(\E_M)$.

By the maximum principle $\V_{M,\spin}$ is non-compact and
$\cbdy \V_{M,\spin} \neq \emptyset$, so part (i) of Theorem
\ref{thm:real-and-complex} implies that $\cbdy \V_{M,\spin}$ is
locally contained in a real-analytic, totally real manifold.  But then
Theorem \ref{thm:hopf-arc} gives an open arc of a Hopf circle
contained in $\rbdy \V_{M,\spin}$, contradicting part (ii) of
Theorem \ref{thm:real-and-complex}.
\end{proof}

\section{Skinning maps: The connected boundary case}

\subsection{Hyperbolic structures}

Let $M$ be a compact, irreducible, atoroidal $3$-manifold with
connected incompressible boundary $S = \partial M$ of genus
$g \geq 2$.  Then the interior $M^\circ$ admits a complete hyperbolic
structure by Thurston's Geometrization Theorem for Haken manifolds.
Let $\AH(M) \subset \X(M,\PSL_2\C)$ denote the set of isometry classes
of marked hyperbolic structures on $M^\circ$.  The closed set $\AH(M)$
lies in the smooth locus of $\X(M,\PSL_2\C)$ (see
\cite[Sec.~8.8]{kapovich}), and its interior $\GF(M)$ consists of
the convex cocompact hyperbolic structures.  The quasiconformal
deformation theory of Kleinian groups gives a holomorphic
parameterization of $\GF(M)$ by the Teichm\"uller space $\T(S)$ (see
e.g.~\cite{bers:spaces-of-kleinian} \cite{kra:deformation-spaces}).
We denote this \emph{Ahlfors-Bers parameterization} by
\begin{equation*}
\begin{split}
\T(S) &\longrightarrow \GF(M) \subset \X(M,\PSL_2\C)\\
X &\longmapsto \rho^M_X
\end{split}
\end{equation*}

\subsection{Quasi-Fuchsian groups}

The set $\QF(S) \subset \X(S, \PSL_2\C)$ of characters of
quasi-Fuchsian representations is an open subset of the smooth locus
in $\X(S,\PSL_2\C)$ that has a natural parameterization by the product
of Teichm\"uller spaces $\T(S) \times \T(\bar{S})$.  As a set $\QF(S)$
does not depend on the orientation of $S$, but the orientation
\emph{is} used to distinguish the factors in this parameterization.
This coordinate system for $\QF(S)$ is a particular case of the
Ahlfors-Bers coordinates, since $\QF(S) = \GF(S \times I) \subset \X(S
\times I,\PSL_2\C) = \X(S,\PSL_2\C)$.  We write $Q(X,Y)$ for the point
in $\QF(S)$ corresponding to the pair $(X,Y) \in \T(S) \times
\T(\bar{S})$.

Given $Y \in \T(\bar{S})$, the \emph{Bers slice} is the subset
$$ B_Y = \{ Q(X,Y) \: | \: X \in \T(S) \} \subset \QF(S), $$
that is, $B_Y$ is a ``horizontal slice'' of the product structure of
the quasi-Fuchsian space.  Similarly we can define the vertical Bers
slices $B_X = \{ Q(X,\param) \}$ for $X \in \T(S)$.

Note $B_Y$ is naturally in one-to-one correspondence with $\T(S)$.  An
inequality of Bers shows that the lengths of geodesics in the
hyperbolic $3$-manifold corresponding to a quasi-Fuchsian group
$Q(X,Y)$ have an upper bound in terms of lengths in the uniformization
of either $X$ or $Y$ \cite[Thm.~3]{bers:boundaries}).  Since one of
these is fixed in a Bers slice, there are uniform bounds on the traces
of any finite set of elements in $\pi_1S$ over $B_Y$, and thus:

\begin{lem}[Bers]
\label{lem:bers-precompact}
For each $Y \in \T(\bar{S})$, the Bers slice $B_Y$ is precompact. \noproof
\end{lem}

Each point in the Bers slice $B_Y$ induces a $\CP^1$-structure on $Y$
by taking the quotient of one of the domains of discontinuity of the
associated quasi-Fuchsian group.  This construction is a local inverse
of the holonomy map, i.e.~it gives an open subset of $Q(Y)$ (the
\emph{Bers embedding}) that maps biholomorphically onto $B_Y$ by
$\hol$.  In particular we have $B_Y \subset r(\sH_Y)$, where $r :
\X(S,\SL_2\C) \to \X(S,\PSL_2\C)$ is the map induced by the covering
$\SL_2\C \to \PSL_2\C$.  This applies equally to the vertical slices,
i.e.~$B_X \subset r(\sH_X)$ for $X \in \T(S)$.

\subsection{The skinning map}

The restriction map $i^* : \X(M,\PSL_2\C) \to \X(S,\PSL_2\C)$ sends
$\GF(M)$ into $\QF(S)$, and in terms of the Ahlfors-Bers
parameterization it is the identity on one Teichm\"uller space factor,
i.e.
$$ i^*(\rho^M_X) = Q(X,\sigma_M(X))$$
which defines a map
$$ \sigma_M : \T(S) \to \T(\bar{S}),$$
the \emph{skinning map of $M$}.

Fibers of the skinning map are related to sets $r(\sH_Y \cap \E_M)$ as
follows:

\begin{lem}
\label{lem:analytic-reduction}
The preimage $\sigma_M^{-1}(Y)$ is in bijection with a precompact
set $F_Y \subset r(\sH_Y \cap \E_M)$.
\end{lem}

\begin{proof}
Given $Y \in \T(\bar{S})$ we consider the set of quasi-Fuchsian groups
$F_Y = i^*(\GF(M)) \cap B_Y$, i.e.
$$F_Y = \{ Q(X,Y) \: | \: \text{There exists} \: X \in \T(S) \: \text{such
  that} \:
i^*(\rho^M_X) = Q(X,Y) \}.$$ 
From the definition of the skinning map it is immediate that
$F_Y$ is in bijection with the preimage
$\sigma_M^{-1}(Y)$ by
$$ Q(X,Y) \in F_Y  \; \iff \; X \in \sigma_M^{-1}(Y),$$
Furthermore 
$F_Y \subset B_Y \subset r(\sH_Y)$, and precompactness of $F_Y$
follows from that of $B_Y$, so it remains only to show that $F_Y \subset r(\E_M)$.

By definition $\E_M$ contains as a Zariski dense subset
$i^*(\X(M,\SL_2\C))$, using the commutative diagram
\eqref{eqn:functorial} for $r,i^*$ we have
$$i^*(r(\X(M,\SL_2\C))) = r(i^*(\X(M,\SL_2\C)))  \subset r(\E_M).$$
Since $F_Y \subset i^*(\GF(M))$, it is enough to know that $\GF(M)
\subset r(\X(M,\SL_2\C))$, i.e.~that the $\PSL_2\C$-representations
arising from hyperbolic structures on $M$ can be lifted to $\SL_2\C$.
This is a well-known consequence of the parallelizability of
$3$-manifolds (see \cite{culler:lifting} or
\cite[Thm~3.1.1]{culler-shalen} for details).
\end{proof}

Finally, using this lemma we have the

\begin{proof}[Proof of Theorem \ref{thm:main} (connected boundary)]
Suppose on the contrary that $\sigma_M^{-1}(Y)$ is infinite.  Then by
the previous lemma the set $r(\sH_Y \cap \E_M)$ has an infinite and
precompact subset, and therefore an accumulation point.  Since $\sH_Y
\cap \E_M$ is discrete by Theorem \ref{thm:main-intersection} and the
map $r : \X(S,\SL_2\C) \to \X(S,\PSL_2\C)$ is proper, this is a
contradiction.
\end{proof}

\section{Disconnected boundary and tori}
\label{sec:disconnected}

The previous sections established the main theorems for a $3$-manifold
with connected boundary.  We now adapt the statements and proofs to
the more general case of a compact oriented $3$-manifold $M$ whose
boundary has at least one connected component that is not a torus.

As in the introduction, we denote by $\bdy_0M$ the union of the
non-torus boundary components of $M$.  Let $S_1, \ldots, S_m$ denote
the connected components of $\bdy_0M$, each equipped with the boundary
orientation.  Let $\bdy_1M = \bdy M \setminus \bdy_0 M$ denote the
union of the torus boundary components of $M$.

\subsection{Measured foliations and Teichm\"uller spaces}
In several cases we first need to adapt definitions of spaces
associated to a surface to the disconnected case.  We define the
measured foliation space $\MF(\bdy_0M)$ and Teichm\"uller space
$\T(\bdy_0M)$
to be the cartesian
products of the spaces corresponding to the connected components
$S_i$, e.g.
$$ \MF(\bdy_0M) = \prod_{i=1}^m \MF(S_i).$$
The sum of the symplectic forms of the factors (using the boundary
orientation) gives the Thurston symplectic form on $\MF(\bdy_0M)$.

For a point $X = (X_1, \ldots X_m) \in \T(\bdy_0M)$, we denote by
$Q(X)$ the direct sum of quadratic differential spaces,
$$ Q(X) = \bigoplus_{i=1}^m Q(X_i).$$
The product of foliation maps of the factors gives the homeomorphism
$\F : Q(X) \to \MF(\bdy_0M)$.

For the $3$-manifold $M$ and its fundamental group, in some cases we
must treat torus boundary components differently.  Instead of
considering arbitrary isometric actions of $\pi_1M$ on $\R$-trees, we
restrict attention to actions in which each subgroup of $\pi_1M$
represented by a component of $\bdy_1M$---that is, each
\emph{boundary torus subgroup}---has a fixed point.

\subsection{Isotropic cones}
The isotropic cone construction (Theorem \ref{thm:length-isotropic},
the main result of sections \ref{sec:isotropic1}--\ref{sec:isotropic2})
generalizes to the disconnected case as follows:

\begin{thm}
\label{thm:length-isotropic-general}
For each $X = (X_1, \ldots, X_M) \in \T(\bdy_0M)$ there exists an
isotropic piecewise linear cone $\L_{M,X} \subset \MF(\bdy_0M)$ with the following
property:

Let $T$ be a $\Lambda$-tree on which $\pi_1M$ acts so that each
boundary torus subgroup has a fixed point.  Let $\ordinc : \R \to
\Lambda$ be an order-preserving embedding.  For each $i$ with $1 \leq i \leq m$
suppose that we have:
\begin{itemize}
\item A pair $T_{\ell_i}, T_{\ell_i}'$ of $\R$-trees on which $\pi_1S_i$ acts
minimally with length function $\ell_i$,
\item A holomorphic quadratic differential $\phi_i \in Q(X_i)$,
\item A $\pi_1S_i$-equivariant straight map $T_{\phi_i} \to
T_{\ell_i}$ with respect to $\ordinc$, and
\item A $\pi_1S_i$-equivariant isometric embedding $k : T_{\ell_i}'
\to T$.
\end{itemize}
Then $[\F(\phi)] \in \L_{M,X}$, where $\phi = (\phi_1, \ldots,
\phi_m)$.
\end{thm}

\begin{proof}
First we adapt the definition of the cone $\L_{M,X}$ from the
connected case.  We choose a finite set of triangulations of $M$ that
extend the triangulations of $S_i$ given by Lemma
\ref{lem:delaunay-track}.  For each such triangulation $\Delta_M$ we
have a space $W_4(\Delta_M,\Lambda)$ of weights satisfying the
$4$-point condition in each $3$-simplex, but we now also require these
weights to be identically zero on the edges of each boundary torus.

As in Lemmas \ref{lem:symplectic-pullback} and
\ref{lem:four-point-isotropic} we find that the restriction of
$W_4(\Delta_M,\R)$ to the edges of the boundary triangulation gives a
finite union of isotropic subspaces for a symplectic form that is the
sum of the Thurston forms for the triangulated surfaces $S_i$ and a
similar alternating $2$-form for the weight space of each boundary
torus.  Since the weights in $W_4(\Delta_M)$ are identically zero
in the torus components, these subspaces are still isotropic
when projected to the product of non-torus factors.  Thus $W_4(\Delta_M)$
gives an isotropic cone in a product of train-track charts
$\MF(\tau_i)$ for the surfaces $S_i$, and taking the union of these
over the finite set of triangulations of $M$ gives the cone $\L_{M,X}
\subset \MF(\bdy_0M)$.

Now we show that $[\F(\phi)] \in \L_{M,X}$, or equivalently that the
train track coordinates of $[\F(\phi)]$ are obtained by restricting an
element of $W_4(\Delta_M)$ to the boundary.  Here $\Delta_M$ is the
triangulation from our finite set in which the edges on $S_i$ can be
realized $\phi_i$-geodesically.

Using the straight maps $T_{\phi_i} \to T_{\ell_i}$ and either the
isometry $T_{\ell_i} \to T_{\ell_i}'$ of Theorem
\ref{thm:culler-morgan} (in the case of a non-abelian length function)
or the partially-defined map $T_{\ell_i} \dashrightarrow T_{\ell_i}'$
of Theorem \ref{thm:local-straightness} (in the abelian case), we
obtain a map from the non-torus boundary vertices, $\verts{\Tilde{\Delta}_M}
\cap \bdy_0M$, to the tree $T$.  We extend this over the vertices
on the torus boundary components by mapping all vertices in a
given boundary torus to a fixed point of the associated subgroup of
$\pi_1M$.

Extending over the remaining (interior) vertices of $\Delta_M$ as in
Propositions \ref{prop:tree-map} and \ref{prop:tree-straight-map}, we
obtain a map $\verts{\Tilde{\Delta}_M} \to T$ whose associated weight function
$w$ lies in $W_4(\Delta_M,\Lambda)$; note that since all vertices of a
boundary torus are mapped to a single point of $T$, the associated
weight vanishes on edges of the torus boundary components as
required.  We push forward by a left inverse of $\ordinc$ to obtain an
element of $W_4(\Delta_M)$ whose values on the non-torus boundary
edges give the train track coordinates of $[\F(\phi)]$.  Thus
$[\F(\phi)] \in \L_{M,X}$.
\end{proof}

\subsection{K\"ahler structure and symplectomorphism}

The results of Section \ref{sec:kahler} generalize easily to
disconnected surfaces by taking products of the spaces, maps, and
stratifications considered there.

Specifically, for any $X \in \T(\bdy_0M)$, Lemma \ref{lem:smooth-stratification}
provides a stratification of $Q(X_i)$.  There is an induced
\emph{product stratification} of the product space $Q(X) = \bigoplus_i
Q(X_i)$ consisting of products of strata in the factors.  Note that the origin
$\{0\} \in Q(X_i)$ is the minimal stratum in each factor, so $Q(X)$
now has nontrivial (positive-dimensional) strata consisting of
quadratic differentials that are zero on one or more of the boundary
components.

Similarly we take the product of the stratified K\"ahler structures on
the factors $Q(X_i)$ to obtain a K\"ahler structure on $Q(X)$, smooth
relative to the product stratification.  Applying Theorem
\ref{thm:symplectomorphism} to each factor of the map $\F: Q(X) \to
\MF(\bdy_0M)$ we obtain:

\begin{thm}
\label{thm:symplectomorphism-general}
For any $X \in \T(\bdy_0M)$, the map $\F : Q(X) \to \MF(\bdy_0M)$ is a
real-analytic stratified symplectomorphism.  That is, if $Q_k(X)$ is a
stratum of $Q(X)$ then:
\begin{rmenumerate}
\item For any $\phi = (\phi_1, \ldots, \phi_m) \in Q_k(X)$ there
exists an open neighborhood $U \subset Q_k(X)$ of $\phi$ and a product
of train track coordinate charts $\prod_i \MF(\tau_i) \subset
\MF(\bdy_0M)$ covering $\F(U)$ so that the restriction
$$ \F : U \to \prod_i \MF(\tau_i) $$ 
is a real-analytic diffeomorphism onto its image, and
\item The derivative $d\F_\phi$ defines a symplectic linear map from
$T_\phi Q_k(X)$ into $\bigoplus_i W(\tau_i)$, where
$T_\phi Q_k(X)$ is equipped with the symplectic form $\sum_i
\omega_{\phi_i}$ and $\bigoplus_i W(\tau_i)$ is given the Thurston symplectic form.
\end{rmenumerate}
\noproof
\end{thm}

To make sense of this statement in case the differentials in the
stratum $Q_k(X)$ are identically zero in some factor, say $Q(X_i)$, we
adopt the convention that $\tau_i$ is the empty train track and that
$\MF(\tau_i) = W(\tau_i) = \{0\}$ is a point representing the
representing the empty foliation on $X_i$, which is the image of $0$
under the map $\F : Q(X_i) \to \MF(S_i)$.

\subsection{Character varieties and extension varieties}

We generalize the character variety of a connected surface to
$\bdy_0M$ by taking the product of character varieties of
components
$$\X(\bdy_0M,G) := \prod_i \X(S_i,G),$$
and similarly for the representation variety $\sR(\bdy_0M,G)$.
Note that while $\sR(\bdy_0M,G)$ can also be described as the
representation variety of the free product $\pi_1S_1 \ast \cdots \ast \pi_1S_m$, the
character variety of this free product does \emph{not} agree with our
definition of $\X(\bdy_0M,G)$.  To obtain $\X(\bdy_0M,G)$ from
$\sR(\bdy_0M,G)$ one must take the quotient of by the action of
$G^m$.

The $3$-manifold character variety also requires modification 
to account for the presence of boundary
tori.  As is standard when considering complete hyperbolic structures,
rather than working with the full character variety of $\pi_1M$, we 
consider the subvariety
$$ \X(M,\bdy_1M,G) \subset \X(M,G) $$
consisting of characters of representations that map each boundary
torus subgroup (that is, the fundamental group of each connected
component of $\bdy_1M$) to parabolic elements of $G$.

The inclusion of each boundary component $S_i \into M$ induces a
restriction map $\X(M,\bdy_1M,\SL_2\C) \to \X(S_i,\SL_2\C)$, and
taking the product of these we obtain a regular map
$$ i^* : \X(M,\bdy_1M,\SL_2\C) \to \X(\bdy_0M,\SL_2\C), $$
We define the extension variety $\E_M \subset \X(\bdy_0M,\SL_2\C)$
as the Zariski closure of the image $i^*(\X(M,\bdy_1M,\SL_2\C))$.

The Morgan-Shalen compactification of $\X(\Gamma,\SL_2\C)$ was defined
in Section \ref{sec:morgan-shalen} using trace functions of elements
of $\Gamma$.  On the product $\X(\bdy_0M,\SL_2\C)$ we have trace
functions for the elements of each component $\pi_1S_i$, so the family
of all such functions is indexed by the disjoint union 
$$H := \pi_1 S_1 \sqcup \cdots \sqcup \pi_1S_m.$$
To adapt the Morgan-Shalen compactification to this case we map the
character variety of $\bdy_0M$ to the projective space $\P(\R^H)$ using formula
\eqref{eqn:morgan-shalen} and take its closure.  Indeed, a
compactification in this generality was already discussed in
\cite{morgan-shalen:valuations-trees}, where the map to projective
space arising from an arbitrary collection of regular functions on an
algebraic variety is considered.

A boundary point of the resulting compactification of
$\X(\bdy_0M,\SL_2\C)$ is therefore an $\R^+$-equivalence class
$[\ell]$ where $\ell = (\ell_1, \ldots, \ell_m)$ is a tuple of
functions, $\ell_i : \pi_1S_i \to \R$.  The factor $\ell_i$ is either
a nontrivial length function of an action of $\pi_1S_i$ on an
$\R$-tree or is identically zero (which is the length
function of the action of $\pi_1S_i$ on a point).

Having adapted the definitions of its objects suitably, the Theorem
\ref{thm:length-function-extension} on extensions of length functions
arising from the boundary of $\E_M$ generalizes to:

\begin{thm}
\label{thm:length-function-extension-general}
Let $[\ell]$ be a boundary point of $\E_M$ in the Morgan-Shalen
compactification of $\X(\bdy_0M, \SL_2\C)$, where $\ell = (\ell_1,
\ldots, \ell_m)$.  Then there exists a function $\hat{\ell} : \pi_1M
\to \R^n$ such that
\begin{rmenumerate}
\item The group $\pi_1M$ acts isometrically on a $\R^n$-tree with
length function $\hat{\ell}$ such that each boundary torus
subgroup of $\pi_1M$ has a fixed point, and 
\item The function $\hat{\ell}$ is a simultaneous extension of the
functions $\ell_i$, i.e.~for each $\gamma \in \pi_1S_i$ we
have $\hat{\ell}(i_*(\gamma)) = i_n(\ell(\gamma))$ where $i_* :
\pi_1S_i \to \pi_1M$ is induced by the inclusion of $S_i$ as a
boundary component of $M$ and $i_n(x) = (0, \ldots, 0, x)$.
\end{rmenumerate}
\end{thm}

\begin{proof}
As in the proof of Theorem \ref{thm:length-function-extension} we have
an extension of fields $k(\E_M^0) \to k(\X_M^0)$ where now $\E_M^0
\subset \X(\bdy_0M, \SL_2\C)$ is an irreducible component of
$\E_M$ which has $[\ell]$ in its boundary and $\X_M^0 \subset
\X(M,\bdy_1M,\SL_2\C)$.

The boundary point $[\ell]$ determines a valuation $v : k(\E_M^0) \to
\Lambda$ by the analogue of Theorem
\ref{thm:valuation-from-boundary-point} for subvarieties of the
product of character varieties $\X(\bdy_0M,\SL_2\C)$; like Theorem
\ref{thm:valuation-from-boundary-point}, this case is covered by the
more general comparison of valuation- and projectivization-based
compactifications of
\cite[Thm.~I.3.6]{morgan-shalen:valuations-trees}.

Extending this valuation to the superfield $k(\X_M^0)$ and proceeding
as in the proof of Theorem \ref{thm:length-function-extension} then
gives a function $\hat{\ell} : \pi_1M \to \R^n$ satisfying condition
(ii) above; note that the argument of
\eqref{eqn:length-extension-argument} applies to each function
$\ell_i : \pi_1S_i \to \R$, $1 \leq i \leq m$, giving that
$\hat{\ell}$ is a simultaneous extension.

The resulting length function $\hat{\ell}$ arises from an action of
$\pi_1M$ on a $\R^n$-tree $T$, but we must show that each boundary
torus subgroup has a fixed point.  Since $\X_M^0$ is an irreducible
subvariety of $\X(M,\bdy_1M,\SL_2\C)$, the trace function
$t_\gamma$ of an element $\gamma$ of a boundary torus subgroup
restricts to a constant $\pm 2$ on $\X_M^0$ (the possible traces of
parabolics).  Therefore we have $v'(t_\gamma) = 0$, where $v'$ is the
valuation of $k(\X_M^0)$ from which $T$ is constructed.  By Lemma
\ref{lem:subtree} each boundary torus subgroup leaves invariant a
subtree of $T$ on which it acts with zero length function, and by
\cite[Prop.~II.2.15]{morgan-shalen:valuations-trees} there is a fixed
point.
\end{proof}

\subsection{Holonomy, the isotropic cone, and discreteness}

The construction of the holonomy variety in Section
\ref{sec:holonomy-variety} extends to $\bdy_0M$ by taking
products; that is, for $X \in \T(\bdy_0M)$ or $X \in
\T(\bar{\bdy_0M})$, where $X = (X_1, \ldots, X_m)$, we define
$$ \sH_X := \bigcup_{\spin \in \mathrm{Spin}(X)} \sH_{X,\spin} \;\;
\text{where} \;\; \sH_{X,\spin} := \prod_{i=1}^m
\sH_{X_i,\spin_i}.$$ Here $\spin = (\spin_1, \ldots,
\spin_m)$ is a tuple of spin structures on the components. 

The subvariety of the quadratic differential space corresponding to
the intersection $\sH_{X,\spin} \cap \E_M$ is
$$\V_{M,\spin} := \hol_\spin^{-1}(\E_M) \subset
Q(\bdy_0M).$$ 
Generalizing Theorem \ref{thm:extensible-holonomy} we have:

\begin{thm}
\label{thm:extensible-holonomy-general}
Let $\phi = (\phi_1, \ldots, \phi_m)$ be the projective limit of a
divergent sequence in $\V_{M,\spin} \subset Q(\bdy_0M)$.  Then
$[\F(\phi)] \in \L_{M,X}$ where $\L_{M,X}$ is the isotropic cone of
Theorem \ref{thm:length-isotropic-general}.
\end{thm}

\begin{proof}
Theorems \ref{thm:length-isotropic-general} and
\ref{thm:length-function-extension-general} provide the necessary
generalizations to adapt the proof of Theorem \ref{thm:extensible-holonomy} to
this situation, except for the existence of the straight maps
$T_{\phi_i} \to T_{\ell_i}$.

In the connected boundary case, a straight map $T_\phi \to T_\ell$ is
given by Theorem \ref{thm:holonomy-limits}
(i.e.~\cite[Thm.~A]{dumas:holonomy}).  Both the projective limit
$\phi$ and the representative $\ell$ of the Morgan-Shalen boundary
point are only well-defined up to multiplication by a positive
constant in this case, but it suffices to have such a straight map for
some pair of representatives $\ell$ and $\phi$.

In the disconnected case, both tuples $\phi = (\phi_1, \ldots,
\phi_m)$ and $\ell=(\ell_1, \ldots, \ell_m)$ represent equivalence
classes up to a single multiplicative constant, whereas a factor-wise
application of Theorem \ref{thm:holonomy-limits} would seem to require
a separate multiplicative factor for each connected component of the
boundary.

To remedy this we use choose representatives $\ell, \phi$ as in
\cite{dumas:holonomy}, where it is shown that for a connected
surface $S$, $X \in \T(S)$, and a divergent sequence $\{\phi_n\}
\subset Q(X)$, we can extract a representative $\ell$ of the
Morgan-Shalen limit of $\hol(\phi_n)$ by scaling the functions $\gamma
\mapsto \log (|t_\gamma(\hol(\phi_n))| + 2)$ by the factors
$\|\phi_n\|^{-1/2}$ and taking the limit $\ell \in \R^{\pi_1S}$.
Furthermore, the straight map $T_\phi \to T_\ell$ of Theorem
\ref{thm:holonomy-limits} is defined for this function $\ell$ and for
the projective limit $\phi \in Q(X)$ satisfying $\|\phi\| = C$, for a
universal constant $C$.

Correspondingly, for the disconnected case we consider $X \in
\T(\bdy_0M)$ and a sequence of
tuples $(\phi_{1,n}, \ldots, \phi_{m,n} )$
which converges projectively.  Equivalently, the sequence
$$ c_n (\phi_{1,n}, \ldots, \phi_{m,n})$$
converges in $Q(X)$ as $n \to \infty$ where
$$ c_n = \left ( \sum_{i=1}^m \|\phi_{i,n}\| \right )^{-1}.$$
Applying Theorem \ref{thm:holonomy-limits}
to each factor and using the representative length functions and
projective limits discussed above, after passing to a subsequence we obtain straight maps
\begin{equation}
\label{eqn:factorwise-straight}
T_{\phi_i^{(1)}} \to T_{\ell_i^{(1)}}
\end{equation}
 where $\phi_i^{(1)} = C \lim_{n \to \infty} \frac{\phi_{i,n}}{\|\phi_{i,n}\|}$ and $\ell_i^{(1)} : \pi_1
S_i \to \R$ is the limit as $n \to \infty$ of the functions
$$ \gamma \mapsto \frac{1}{\|\phi_{i,n}\|^{\half}} \log
(|t_\gamma(\hol(\phi_{i,n}))| + 2).$$
Since $0 \leq \|\phi_{i,n}\| \leq c_n^{-1}$ we can take a further
subsequence so that for each $i$ the limit
$$ r_i = \lim_{n \to \infty} c_n \|\phi_{i,n}\| \in [0,1],$$
exists.  Then
\begin{equation*}
\begin{split}
\phi :=& \:(r_1 \phi_1^{(1)}, \ldots, r_n \phi_n^{(1)})\\
=& \:C \lim_{n \to
  \infty} c_n ( \phi_{1,n}, \ldots, \phi_{m,n} ) \in Q(X)
\end{split}
\end{equation*}
is a projective limit of the sequence of quadratic differentials and
\begin{equation*}
\begin{split}
\ell :=& \:(r_1^\half \ell_1^{(1)}, \ldots, r_n^\half \ell_n^{(1)})\\ =&\:
\lim_{n \to \infty} \left ( \gamma \mapsto c_n^{1/2} \log
(|t_\gamma(\hol(\phi_{i,n}))| + 2) \right )_{\gamma \in \pi_1{S_i}, \;
  1 \leq i \leq n} \in \R^H
\end{split}
\end{equation*}
represents the limit of the holonomy representations in
the Morgan-Shalen compactification of $\X(\bdy_0M,\SL_2\C)$.  The desired straight
maps
$$ T_{\phi_i} \to T_{\ell_i} $$
are therefore given by \eqref{eqn:factorwise-straight} for each $i$ with $r_i
\neq 0$, since the trees $T_{\phi_i},T_{\ell_i}$ are
obtained from $T_{\phi_i^{(n)}}, T_{\ell_i^{(n)}}$ by multiplying
  the metrics by $r_i^{1/2}$.  In each factor where $r_i = 0$ we have
  $\phi_i = 0$ and $\ell_i = 0$, so each of $T_{\phi_i}$ and
  $T_{\ell_i}$ is a point, and there is nothing to prove.
\end{proof}

Using Theorems \ref{thm:extensible-holonomy-general} and
\ref{thm:symplectomorphism-general} in place of their counterparts for
the connected boundary case (Theorems \ref{thm:extensible-holonomy}
and \ref{thm:symplectomorphism}, respectively), the properties of the
real and complex boundaries of $\V_{M,\spin}$ from Theorem
\ref{thm:real-and-complex} follow for the general case as well.  The
argument of Section \ref{sec:discreteness} then gives the discreteness
of these varieties, completing the proof of Theorem
\ref{thm:main-intersection} in the general case.

\subsection{Skinning maps}

Now suppose that in addition to the hypotheses of Theorem
\ref{thm:main-intersection} that the $3$-manifold $M$ has
incompressible boundary and that its interior admits a complete
hyperbolic structure with no accidental parabolics.  The space $\GF(M)$
of geometrically finite hyperbolic structures on $M$ is naturally a
subset of $\X(M,\bdy_1M,\PSL_2\C)$ which is parameterized by
the Teichm\"uller space $\T(\bdy_0M)$; as before we denote this
parameterization by $X \mapsto \rho_M^X$ where $X = (X_1, \ldots, X_m)$.

The restriction map $i^* : \X(M,\bdy_1M,\PSL_2\C) \to
\X(\bdy_0M,\PSL_2\C)$ sends $\rho_M^X$ to a tuple of
quasi-Fuchsian groups
$$ i^*(\rho_M^X) = ( Q(X_1, Y_1), \ldots, Q(X_m, Y_m) ) \in \prod_{i=1}^m \QF(S_i) $$
and this defines the \emph{skinning map}
\begin{equation*}
\begin{split}
\sigma_M : \T(\bdy_0M) &\to \T(\bar{\bdy_0 M})\\
 (X_1, \ldots, X_m) &\mapsto (Y_1, \ldots, Y_m).
 \end{split}
\end{equation*}

With this definition in hand, the generalization of the proof of the
main theorem from the connected case is straightforward:

\begin{proof}[Proof of Theorem \ref{thm:main} (general case)]
Suppose on the contrary that $\sigma_M^{-1}(Y)$ is infinite.  As in
Lemma \ref{lem:analytic-reduction}, it follows from the definition of
$\sigma_M$ that the preimage $\sigma_M^{-1}(Y)$ is in bijection with a
subset of the $B_Y \cap r(\E_M)$, where $B_Y = \prod_i B_{Y_i}$ is the
product of Bers slices.  Here $r: \X(\bdy_0M, \SL_2\C) \to
\X(\bdy_0M,\PSL_2\C)$ is the finite-sheeted, proper map induced by
$\SL_2\C \to \PSL_2\C$.  Using Theorem \ref{thm:main-intersection} we find
that $r(\sH_Y \cap \E_M)$ is a discrete set.  But since $B_Y$ is a
precompact subset of $r(\sH_Y)$, the infinite subset we have
identified with $\sigma_M^{-1}(Y)$ has an accumulation point, which is
a contradiction.
\end{proof}

\section{Extended skinning maps of acylindrical manifolds}
\label{sec:acylindrical}

As mentioned in the introduction, the proof of Theorem \ref{thm:main}
can be adapted to show finiteness of fibers of Thurston's continuous
extension of the skinning map
$ \hat{\sigma}_M : \AH(M) \to \T(\bar{\bdy_0M})$ for an acylindrical
manifold $M$.  A proof of the existence of this continuous extension
by Brock, Kent, and Minsky can be found \cite[Sec.~9]{kent:skinning}.

The only property of the extension we will use is the following, which
can be taken as the definition of $\hat{\sigma}_M$: If the image of
$\rho \in \AH(M)$ under the extended skinning map is
$$\Hat{\sigma}_M(\rho) = (Y_1, \ldots, Y_m) \in \T(\bar{\bdy_0M}),$$ 
and if the image of $\rho$ under the restriction map is
$$i^*(\rho) = ( \eta_1, \ldots, \eta_m ) \in \X(\bdy_0M,\PSL_2\C)$$
then $\eta_i$ is the character of a discrete, faithful $\pi_1
S_i$-representation which has an invariant disk in its domain of
discontinuity for which the quotient Riemann surface is $Y_i$.

The analogue of Theorem \ref{thm:main} for $\hat{\sigma}_M$ is:

\begin{thm}
\label{thm:main-extended}
Let $M$ be a compact, oriented, irreducible, acylindrical $3$-manifold
with incompressible boundary that is not empty and not a union of tori.
Let $\hat{\sigma}_M : \AH(M) \to \T(\bar{\bdy_0M})$ denote the
extension of the skinning map of $M$.  Then for each $Y \in
\T(\bar{\bdy_0M})$, the set $\hat{\sigma}_M^{-1}(Y)$ is finite.
\end{thm}

\begin{proof}
First we show that $i^* : \X(M,\bdy_1M,\PSL_2\C) \to \X(\bdy_0M,\PSL_2\C)$ is
injective when restricted to $\AH(M)$.  Suppose
$i^*(\rho) = i^*(\rho')$ with $\rho,\rho' \in \AH(M)$.  We claim the
hyperbolic structures on $M$ associated to $\rho$ and $\rho'$ are
bilipschitz.  Since $i^*(\rho) = i^*(\rho')$, these hyperbolic
structures are actually \emph{isometric} in a neighborhood of each
non-torus end.  Smooth bilipschitz maps of the torus ends can be
obtained by choosing affine maps between the cusp tori and extending
normally.  We have therefore defined a smooth bilipschitz equivalence
on all of the ends, and any diffeomorphic extension to the remaining
compact part is globally bilipschitz.  By Sullivan's
rigidity theorem \cite{sullivan:rigidity}, we conclude $\rho = \rho'$.

Now suppose for contradiction that $\rho^{(n)} \in \AH(M)$ is an infinite
sequence of pairwise distinct points satisfying
$\hat{\sigma}_M(\rho^{(n)}) = Y = (Y_1, \ldots, Y_m)$ for all $n$.  As in the proof of Theorem
\ref{thm:main}, it is enough to show that this leads to an
accumulation point in $r(\sH_Y \cap \E_M)$.

 Consider the sequence
$$i^*(\rho^{(n)}) = ( \eta_1^{(n)}, \ldots, \eta_m^{(n)} ) \in \prod_i
\X(S_i,\PSL_2\C),$$
which is also pairwise distinct since $i^*$ is injective on $\AH(M)$.  By the
definition of $\hat{\sigma}_M$, the character $\eta_i^{(n)}$
corresponds to a discrete $\pi_1S_i$-representation with an invariant
disk in its domain of discontinuity having quotient Riemann surface $Y_i$.  Since
this quotient describes a projective structure on $Y_i$ with holonomy
$\eta_i^{(n)}$, we find that
$i^*(\rho^{(n)}) \in r(\sH_Y)$.  Since $r(\E_M)$ contains the range of
$i^*$, we in fact have $i^*(\rho^{(n)}) \in r(\sH_Y \cap \E_M)$.

While the representations $\eta_i^{(n)}$ are not necessarily
quasi-Fuchsian, a generalization of Bers' inequality
(e.g.~\cite[Prop.~2.1]{ohshika:tame}) shows that the fixed quotient Riemann surface
$Y_i$ constrains $\eta_i^{(n)}$ to lie in a compact subset of
$\X(S_i,\PSL_2\C)$.  Thus $i^*(\rho^{(n)})$ lies in the product of
these compact sets for all $n$, giving an accumulation point of
$r(\sH_Y \cap \E_M)$ and the desired contradiction.
\end{proof}

\nocite{}
\newcommand{\removethis}[1]{}

\end{document}